\documentclass[11pt]{amsart}
\usepackage{amsmath,amssymb,amsthm, amscd}
\usepackage{tikz-qtree}
\usepackage[all]{xy}
\usetikzlibrary{positioning}

\newcommand\Fp{\mathbb{F}_p}

\newcommand\Z{\mathbb{Z}}
\newcommand\C{\mathbb{C}}
\newcommand\Q{\mathbb{Q}}
\newcommand\N{\mathbb{N}_0}

\newcommand{\PP}{\mathbb{P}}      

\DeclareMathOperator{\PGL}{PGL}

\DeclareMathOperator{\ord}{ord}
\DeclareMathOperator{\lcm}{lcm}

\newenvironment{Gal}{{\rm Gal\,}}{}

\newtheorem{theorem}{Theorem}[section]
\newtheorem{definition}[theorem]{Definition}
\newtheorem{lemma}[theorem]{Lemma}
\newtheorem{proposition}[theorem]{Proposition}
\newtheorem{proposition-definition}[theorem]{Proposition-Definition}
\newtheorem{corollary}[theorem]{Corollary}
\newtheorem{conjecture}[theorem]{Conjecture}

\newtheorem{problem}[theorem]{Problem}

\theoremstyle{definition}

\theoremstyle{remark}
\newtheorem*{remark}{Remark}

\title[Powers in orbits of rational functions]{Powers in orbits of rational functions: cases of an arithmetic dynamical Mordell-Lang conjecture}

\author{Jordan Cahn, Rafe Jones, Jacob Spear}
\thanks{MSC 37P05 (primary), 11G05, 37P15 (secondary) \\ The first and third authors' research was supported by Carleton college's HHMI grant for undergraduate science education and the Carleton college department of Mathematics and Statistics.}

\numberwithin{equation}{section}

\begin{document}

\begin{abstract} Let $K$ be a finitely generated field of characteristic zero. We study, for fixed $m \geq 2$, the rational functions $\phi$ defined over $K$ that have a $K$-orbit containing infinitely many distinct $m$th powers. For $m \geq 5$ we show the only such functions are those of the form $cx^j(\psi(x))^m$ with $\psi \in K(x)$, and for $m \leq 4$ we show the only additional cases are certain Latt\`es maps and four families of rational functions whose special properties appear not to have been studied before. 
With additional analysis, we show that the index set $\{n \geq 0 : \phi^{n}(a) \in \lambda(\mathbb{P}^1(K))\}$ is a union of finitely many arithmetic progressions, where $\phi^{n}$ denotes the $n$th iterate of $\phi$ and $\lambda \in K(x)$ is any map M\"obius-conjugate over $K$ to $x^m$. When the index set is infinite, we give bounds on the number and moduli of the arithmetic progressions involved. These results are similar in flavor to the dynamical Mordell-Lang conjecture, and motivate a new conjecture on the intersection of an orbit with the value set of a morphism. 
A key ingredient in our proofs is a study of the curves $y^m = \phi^{n}(x)$. We describe all $\phi$ for which these curves have an irreducible component of genus at most 1, and show that such $\phi$ must have two distinct iterates that are equal in $K(x)^*/K(x)^{*m}$. 
\end{abstract}

\maketitle

\section{Introduction}

Let $K$ be a field and $\phi \in K(x)$ a rational function with coefficients in $K$. We denote by $\phi^n$ the $n$th iterate of $\phi$, which we emphasize is distinct from the $n$th power of $\phi$. A fundamental object in dynamics is the (forward) orbit\footnote{We generally drop the word ``forward" in this article, but we wish to avoid confusion with the backwards orbit $O_\phi^-(a)$, which we use frequently (see p. \pageref{backdef} for a definition). We thus prefer the notation $O_\phi^+(a)$ for the forward orbit rather than the more standard $O_\phi(a)$.} 
$$
O_\phi^+(a) = \{\phi^n(a) : n \geq 0\} $$
of $a \in \mathbb{P}^1(K)$ under the map $\phi$; note that $\phi^0(x) = x$ by convention, and so $a \in O_\phi^+(a)$. 
An overarching goal is to classify the orbits of a given map $\phi$ in terms of salient features of $K$, such as a metric or arithmetic structure. A related goal, which has attracted a large body of work, is to understand the collection of maps that can possess an orbit with certain very special properties. For example, when $K = \C$, Ghioca, Tucker, and Zieve \cite{gtz1, gtz2} show that if $f \in \C[x]$ has degree at least two, then for each $g \in \C[x]$ with degree at least two and such that an orbit of $g$ has infinite intersection with an orbit of $f$, it follows that $f$ has a common iterate with $g$. Thus the existence of a special orbit of $f$ has global implications for $f$; in particular, it implies functional properties of the map $f$. Another example of such a result is due to Silverman \cite[Theorem A]{jhsintegers}. Recall that the degree of $\phi$ may be defined by writing $\phi(x) = A_1(x)/A_2(x)$ with $A_1, A_2 \in K[x]$ relatively prime polynomials, and taking the maximum of the degrees of $A_1$ and $A_2$. Silverman shows that if $\phi(x) \in \Q(x)$ has degree at least two, and there is an orbit of $\phi(x)$ containing infinitely many integers, then $\phi^2(x)$ is a polynomial (a more general result is given in \cite[Theorem B]{jhsintegers}). 
This theme is taken much further in the dynamical Mordell-Lang conjecture \cite[Conjecture 1.5.0.1]{dmlbook}, which posits that if $\Phi$ is an endomorphism of a quasiprojective variety $X$ \label{dml} defined over $\C$, $a$ is any point in $X(\C)$, and $V \subset X$ is any subvariety, then $\{n \geq 0: \Phi^n(a) \in V(\C)\}$ is a union of finitely many arithmetic progressions (note that singletons are considered arithmetic progressions, and thus any finite set is a union of arithmetic progressions). In particular, if $O_\phi^+(a) \cap V(\C)$ is infinite, then $V$ contains a positive-dimensional subvariety that is periodic under the action of $f$. Indeed, let $M > 0$ and $\ell \geq 0$ be such that $\Phi^{kM + \ell}(a) \in V(\C)$ for all $k \geq 0$; then the Zariski closure of $\{\Phi^{kM + \ell}(a) : k \geq 0\}$ is positive-dimensional and invariant under $\Phi^M$. For a summary of the extensive recent work surrounding this conjecture, see \cite{dmlbook}.

From this point forward, we let $K$ be a finitely generated field of characteristic zero, that is, an extension of $\Q$ generated by a finite set of (possibly transcendental) elements; all such fields can be embedded in the complex numbers, and throughout this article we consider $K$ as a subfield of $\C$. Fix an integer $m \geq 2$. Our goal is a study of the $\phi \in K(x)$ possessing a $K$-orbit containing infinitely many distinct $m$th powers in $K$. 
The existence of such an orbit implies infinitely many distinct $K$-rational solutions to the equation $\phi^n(x) = y^m$ for each $n \geq 1$, and hence by Faltings' Theorem the curve $C_n : \phi^n(x) = y^m$ must have an irreducible component of genus at most one, for all $n \geq 1$ (throughout, we take the curve given by rational functions $\frac{A_1(x)}{A_2(x)} = \frac{B_1(y)}{B_2(y)}$ to be that given by $A_1(x)B_2(y) - B_1(y)A_2(x) = 0$). It is easy to see that every irreducible component of $C_n$ has the same genus (Proposition \ref{irred}), and we denote this quantity by $g_n$. We are thus interested in the maps $\phi$ such that $g_n$ is at most one for all $n \geq 1$; our first two results (Theorems \ref{iterative relationship} and \ref{maingenus}) deal with the a priori more general situation where $g_n$ is bounded as $n$ grows. These results, and also Corollary \ref{powercor}, show that the existence of an \textit{arithmetically} special orbit of $\phi$ implies strong conclusions about the global structure of the function. 
\begin{theorem} \label{iterative relationship}
Fix $m \geq 2$ and let $\phi \in \C(x)$ have degree at least two. Then $g_n$ is bounded as $n \to \infty$ if and only if 
\begin{equation} \label{itreleqn}
\text{there exist integers $r > s \geq 0$ such that $\phi^r(x) = \phi^s(x)(\psi(x))^m$ for some $\psi \in \C(x)$.}
\end{equation}
 In that case,
 \begin{enumerate}
 \item[(a)] If $\phi \in K(x)$ for some subfield $K$ of $\C$, then \eqref{itreleqn} holds for some $\psi \in K(x)$.
 \item[(b)] We have $g_n \leq 1$ for all $n \geq 1$.
 \item[(c)] One may take $r \leq m$ if $m \geq 3$, and $r \leq 6$ if $m = 2$.
 \end{enumerate}
\end{theorem}

A full accounting of the possible values of $r$ and $s$ that occur when \eqref{itreleqn} holds may be found in Section \ref{pfmain}. 

As a primary part of our proof of Theorem \ref{iterative relationship}, we show the following theorem. Recall that the post-critical set \text{Postcrit$(\phi)$} \label{postcrit} of a rational function $\phi \in \C(x)$ is $\bigcup_{n\geq 1} \phi^n(C),$
where $C$ is the critical set for $\phi$, i.e., the set of points in $\PP^1(\C)$ at which $\phi$ is not locally one-to-one. A map $\phi \in \C(x)$ of degree at least two is a \textit{Latt\`es map} \label{lattesintro} if there is a linear map $L(t) = at + b$ acting on a complex torus $\C/\Lambda$ and a finite-to-one holomorphic map $\Theta : \C / \Lambda \to \mathbb{P}^1(\C)$ satisfying $\phi \circ \Theta = \Theta \circ L$. Denote by $e_\phi(z)$ the ramification index, or local degree, of $\phi$ at $z \in \PP^1(\C)$; when $z \neq \infty$ and $\phi(z) \neq \infty$, this coincides with the multiplicity of $z$ as a root of $\phi(x) - \phi(z)$ (see 
\cite[p. 12]{jhsdynam} for a full discussion). \label{ephidef} Usefully, the Latt\`es maps are precisely those rational functions $\phi \in \C(x)$ such that there exists a function $r : \PP^1(\C) \to \Z$ satisfying
\begin{equation}  \label{rz}
r(\phi(z)) = e_\phi(z) \cdot r(z) \; \; \text{for all $z \in \mathbb{P}^1(\mathbb{C})$} \qquad \text{and} \qquad \text{$r(z) = 1$ for $z \not\in \text{Postcrit$(\phi)$}$.}
\end{equation}
Such a function $r$ is unique; see Theorem \ref{milnor2} or \cite[Section 4]{milnor} for details. When there exists such a function $r$, the collection of values of $r$ on $\text{Postcrit$(\phi)$}$ is called the \textit{signature} of $\phi$, and the only possible signatures are (2,2,2,2), (3,3,3), (2,4,4), and (2,3,6) \cite[Corollary 4.5]{milnor}.

\begin{theorem} \label{maingenus} 
Fix $m \geq 2$ and let $\phi \in \C(x)$ have degree at least two. Then $g_n$ is bounded as $n \to \infty$ if and only if one of the following holds:
\begin{enumerate}
\item $\phi(x) = cx^j(\psi(x))^m$ with $\psi \in \C(x)$, $0 \leq j \leq m-1$, $c \in \C^*$;
\item $m = 4$ and $\phi$ is a Latt\`es map of signature $(2,4,4)$, with $\{0, \infty\}$ in the post-critical set and $r(0) = r(\infty) = 4$, where $r$ is the function satisfying \eqref{rz}; 
\item $m = 3$ and $\phi$ is a Latt\`es map of signature $(3,3,3)$, with $\{0, \infty\}$ in the post-critical set;
\item $m = 2$ and $\phi$ is a Latt\`es map of signature $(2,2,2,2)$ with $\{0, \infty\}$ in the post-critical set;
\item $m = 2$ and either $\phi(x)$ or $1/\phi(1/x)$ can be written in one of the following ways, where $B, C \in \C^*$, $f, g, h \in \C[x] \setminus \{0\}$, and the numerator and denominator of each fraction have no common roots in $\C$:
\begin{enumerate}
\item $-\frac{f(x)^2}{(x-C)g(x)^2}$ with $f(x)^2 + C(x-C)g(x)^2 = Cxh(x)^2$; 
\item $-\frac{(x-C)f(x)^2}{g(x)^2}$ with $(x-C)f(x)^2 + Cg(x)^2 = xh(x)^2$; 
\item ${B\frac{(x-C)f(x)^2}{g(x)^2}}$ with $B(x-C)f(x)^2 - Cg(x)^2 = -Ch(x)^2$; 
\item ${B\frac{x(x-C)f(x)^2}{g(x)^2}}$ with $Bx(x-C)f(x)^2 - Cg(x)^2 = -Ch(x)^2$;
\end{enumerate}
\end{enumerate}
Moreover, if $K$ is a subfield of $\C$ with $\phi \in K(x)$, then we may take 
\begin{align} 
\label{kpart1} & \text{$\psi \in K(x)$ and $c \in K^*$ in case (1)} \\  \label{kpart2} & \text{$B, C \in K^*$ and $f, g, h \in K[x] \setminus \{0\}$ in case (5).} 
\end{align}
\end{theorem}

The maps in part (5) of Theorem \ref{maingenus} appear not to have been studied before in general. We discuss how to give explicit parameterizations of all such maps in Proposition \ref{paramprop} and the paragraphs following. Important examples of these maps are closely related to the degree-$d$ monic Chebyshev polynomial $T_d$, defined by the equation $T_d(x + x^{-1}) = x^d + x^{-d}$; see \cite[Section 2]{milnor} or \cite[Section 6.2]{jhsdynam} for further properties. The map $(-1)^d(T_d(x + 2)) - 2$ satisfies (5b) when $d$ is odd and (5d) when $d$ is even (see Corollary \ref{mainpoly} for more on these maps). Note that maps of type (5a) and (5c) cannot be M\"obius-conjugate to polynomials (see the proof of Theorem \ref{maingenus} in Section \ref{firstthmsec}).

Combining Faltings' theorem with Theorems \ref{iterative relationship} and \ref{maingenus}, we obtain the main result of this paper. Denote by $\mathbb{P}^1(K)^m$ the set $\{k^m : k \in K\} \cup \{\infty\}$. 

\begin{corollary} \label{powercor}
Let $K$ be a finitely generated field of characteristic zero field, let $\phi \in K(x)$ have degree at least two, and fix $m \geq 2$. If there exists $a \in \mathbb{P}^1(K)$ such that $O_\phi^+(a) \cap \mathbb{P}^1(K)^m$ is infinite, then $\phi$ falls into one of the cases in Theorem \ref{maingenus} and satisfies \eqref{kpart1} and \eqref{kpart2}, and $\phi$ also satisfies \eqref{itreleqn} with $\psi \in K(x)$. 
\end{corollary}

Thus, the infinitude of $O_\phi^+(a) \cap \mathbb{P}^1(K)^m$ implies strong functional properties of $\phi$, similar to the results of \cite{gtz1, gtz2, jhsintegers} mentioned at the beginning of this section.

The proofs of Theorems \ref{iterative relationship} and \ref{maingenus} unfold in two steps. The first is geometric, and involves studying $\phi \in \C(x)$ for which $g_n$ is bounded. The second is arithmetic, and consists of showing that various quantities in the two theorems may be defined over a subfield $K$ of $\C$, when $\phi$ was initially defined over $K$. 
In the geometric part, our study of the genus of $C_n$ is a case of a problem with a long history, which remains largely unresolved: determine all pairs $A, B$ of complex rational functions such that the curve $A(x) = B(y)$ has an irreducible component of genus at most one. In the case where $A$ and $B$ are polynomials, a complete solution is given in \cite{bilutichy} for irreducible components of genus zero with at most two points at infinity (see \cite[pp. 264, 281]{bilutichy} for discussion and references regarding the extensive past work on this problem). Partial results exist for irreducible components of genus one, e.g. \cite{an, az1}, again assuming $A$ and $B$ are polynomials. When $A$ and $B$ are allowed to be non-polynomial rational functions, there are many fewer results available. One example is \cite{pak1}, which classifies all $A, B$ with no common critical values such that $A(x) = B(y)$ has an irreducible component of genus at most one. 

While the general problem is far from resolution, we propose the following variant, which our Theorem \ref{maingenus} resolves for $B(y) = y^m, m \geq 2$. 
\begin{problem} \label{itclass}
Given $B \in \C(x)$, explicitly determine all $\phi \in \C(x)$ such that 
\begin{equation} \label{smallcomp}
\text{for all $n \geq 1$ the curve $\phi^n(x) = B(y)$ has an irreducible component of genus at most one.}
\end{equation} 
\end{problem}
One may also fix $\phi$ and study rational functions $B$ for which \eqref{smallcomp} holds. This is the approach taken in the recent preprint \cite{pak5}, where it is shown, among other results, that if $\phi$ is not a power map, Chebyshev polynomial, or Latt\`es map, and \eqref{smallcomp} holds, then there is a rational Galois covering $h : \PP^1 \to \PP^1$ (depending only on $\phi$) and rational functions $V, V'$ satisfying $\phi \circ h = h \circ V'$ and $\phi^\ell \circ h = B \circ V$ for some $\ell \geq 1$.

The work of Ghioca, Tucker, and Zieve in \cite{gtz1, gtz2} addresses a question of similar flavor to Problem \ref{itclass}, though still quite distinct. There, the authors classify all pairs $f, g$ of complex polynomials such that for every $m,n \geq 1$ the curve $f^n(x) = g^m(y)$ has an irreducible component of genus zero with at most two points at infinity. To do so, they rely on the classification of Bilu and Tichy \cite{bilutichy} mentioned above; thus they already know precisely which curves $A(x) = B(y)$ have the desired property, but they must determine when $A$ and $B$ arise from iteration of lower-degree polynomials. This requires significant and novel results on polynomial decomposition. 

Taking $B(y) = y^m,$ as in Theorem \ref{maingenus}, greatly eases the generally difficult problem of determining the irreducible components of $A(x) = B(y)$, and leads to a considerably simpler genus formula than the one for general curves of the form $A(x) = B(y)$ (see Propositions \ref{irred} and \ref{genus}). Nonetheless, similar to the situation of \cite{gtz1, gtz2}, we are left with the a priori difficult problem of determining the maps $\phi$ such that $A$ can be taken to be an arbitrary iterate of $\phi$.

The arithmetic part of the proofs of Theorems \ref{iterative relationship} and \ref{maingenus} can be found mainly in Section \ref{fielddef}. This aspect of our results, in particular the part of Theorem \ref{iterative relationship} where $\psi$ may be defined over $K$ when $\phi$ is defined over $K$, leads to a result whose conclusion is the same as that of the dynamical Mordell-Lang conjecture. When the intersection set is infinite, we are able to prove the far stronger conclusion that three arithmetic progressions suffice, and we give information on their moduli. Throughout, we denote by $\N$ the set of nonnegative integers. 

\begin{theorem} \label{main ML}
Let $K$ be a finitely generated field of characteristic zero, let $\phi, \lambda \in K(x)$ each have degree at least two, and suppose that $\lambda$ is M\"obius-conjugate (over $K$) to a power map.   
Then for every $a \in \mathbb{P}^1(K)$, the set 
\begin{equation} \label{indexset}
\{n \in \N : \phi^n(a) \in \lambda(\mathbb{P}^1(K))\}
\end{equation}
is a finite union of arithmetic progressions. If $O_\phi^+(a) \cap \lambda(\mathbb{P}^1(K))$ is infinite, then the set \eqref{indexset} is a union of at most three arithmetic progressions, each with modulus $M$ satisfying $M \leq m$ if $m \geq 3$ and $M \leq 6$ if $m = 2$.
\end{theorem}

We emphasize again that we take singletons to be arithmetic progressions of modulus $0$, and so Theorem \ref{main ML} holds trivially when the set \eqref{indexset} is finite. If $O_\phi^+(a) \cap \lambda(\mathbb{P}^1(K))$ is finite, then either $O_\phi^+(a)$ is infinite and the set \eqref{indexset} is finite or $O_\phi^+(a)$ is finite; in either case Theorem \ref{main ML} holds trivially. In Section \ref{example} we give an example where $O_\phi^+(a) \cap \lambda(\mathbb{P}^1(K))$ is infinite and the set \eqref{indexset} cannot be written as a union of two arithmetic progressions, showing that three is best possible. The bound on $M$ in Theorem \ref{main ML} is best possible for $m \geq 3$, regardless of the choice of $K$ (see Lemma \ref{trivmain}); for $m = 2$ the bound can be reduced to $M \leq 4$ using an analysis of the field of definition of Latt\`es maps, which we plan to describe in a future article. 

The proof of Theorem \ref{main ML} quickly reduces to the case $\lambda(x) = x^m$. Indeed, let $\mu \in \PGL_2(K)$ and put $\phi^\mu = \mu^{-1} \circ \phi \circ \mu$. If $\phi^n(a) = \lambda(b)$ for $a, b \in \mathbb{P}^1(K)$, then $(\phi^{\mu})^n(\mu^{-1}(a)) = \lambda^{\mu}(\mu^{-1}(b))$, giving
\begin{equation} \label{Mobinv}
\{n \in \N: \phi^n(a) \in \lambda(\mathbb{P}^1(K))\} = \{n \in \N: (\phi^{\mu})^n(\mu^{-1}(a)) \in \lambda^\mu(\mathbb{P}^1(K))\}.
\end{equation}
 Hence if Theorem \ref{main ML} can be established for $\lambda^{\mu}$ and arbitrary $\phi$ and $a$, it must also hold for $\lambda$ and arbitrary $\phi$ and $a$. Thus if $\mu$ conjugates $\lambda$ to a power map, we have reduced to the case $\lambda(x) = x^m$, as desired. Note that if $\lambda$ is not conjugate over $K$ to a power map, we cannot take advantage of the special geometric properties of the curve $\phi^n(x) = y^m$ mentioned in the discussion following Problem \ref{itclass}, and thus new methods would be required.

The dynamical Mordell-Lang conjecture asserts that once there exist infinitely many instances of the intersection between the geometric object $V(\C)$ and the arithmetic dynamical object $O_\phi^+(a)$ (using the notation from the first paragraph of the introduction), then the intersection must have a structure: its index set must be given by finitely many arithmetic progressions. Theorem \ref{main ML} proves this assertion in the case where $X = \PP^1$ and the geometric object $V$ is replaced by an arithmetic object, namely the set of $K$-values of the morphism $\lambda: X \to X$. We conjecture that a similar conclusion holds for the set of $K$-values of more general morphisms:

\begin{conjecture}[Arithmetic dynamical Mordell-Lang conjecture for $\mathbb{P}^1$] \label{adml}
Let $X = \mathbb{P}^1$ and let $Y$ be a curve defined over a finitely generated field of characteristic zero $K$. Suppose that $\lambda : Y \to X$ is a finite $K$-morphism, and $\phi : X \to X$ is a morphism of degree at least two. Then for any $a \in X(K)$, the set $\{n \in \N : \phi^n(a) \in \lambda(Y(K))\}$ is a finite union of arithmetic progressions.
\end{conjecture} 

Shortly after this paper was posted to the ArXiv, Hyde and Zieve sent us a proof of Conjecture \ref{adml}. Their short argument makes use of the pigeonhole principle, as well as the finiteness of the number of topological covers of a compact Riemann surface with specified degree and branch points. It also yields a proof of Theorem \ref{maingenus} in the case where the curve $\phi^n(x) = y^m$ is irreducible for all $n \geq 1$.

It is interesting to consider whether a similar conclusion to that of Conjecture \ref{adml} holds for $X = \mathbb{P}^j$ with $j \geq 1$, where $Y$ is a projective variety and $\lambda$ is finite onto its image; indeed, one may extend the question further to the case where $X$ and $Y$ are any quasi-projective varieties, and $\phi$ is an endomorphism of $X$.
To see why such a generalization of Conjecture \ref{adml} is plausible, let $Z_n$ ($n \geq 1$) be the subvariety of $X \times Y$ where the morphisms $\phi^n : X \to X$ and $\lambda : Y \to X$ agree. Then there is a natural $K$-morphism $Z_{n+1} \to Z_{n}$ taking $(x, y)$ to $(\phi(x), y)$, which we again denote by $\phi$.  Thus for any $i > j$ there is a finite map $\phi^{i-j} : Z_i(K) \to Z_j(K)$. Suppose that $O_\phi^+(a) \cap \lambda(Y(K))$ is infinite; otherwise the conclusion of Conjecture \ref{adml} holds trivially, as in the paragraph following Theorem \ref{main ML}. Thus $O_\phi^+(a)$ must be infinite, and hence $\phi^i(a) \neq \phi^j(a)$ for $i \neq j$. We label the next observation for future reference: 
\begin{equation} \label{infrat}
\text{for any fixed $n \geq 1$, there are infinitely many $i > n$ with $\phi^{n}(\phi^{i - n}(a)) \in \lambda(Y(K))$,}
\end{equation}
implying there are infinitely many points in $Z_n(K)$ for all $n \geq 1$. If these points are Zariski-dense in $Z_n$, then the Bombieri-Lang conjecture \cite[Conjecture F.5.2.1]{jhsdioph} predicts that $Z_n$ is not a variety of general type. We speculate that under suitable hypotheses this implies a functional relationship among iterates of $\phi$ and $\lambda$, for instance $\phi^r = 
\lambda \circ g$ for some $r \geq 1$ and some $K$-morphism $g : X \to Y$. 

The previous paragraph furnishes an outline for our proof of Theorem \ref{main ML}. In the situation of that theorem, $Z_n$ is a curve, and thus any infinite subset is Zariski dense, and the Bombieri-Lang conjecture is the famous theorem of Faltings \cite[Corollary 2.2, p.12]{LangDG} (see \cite[Theorem E.0.1]{jhsdioph} for an exposition of the number field case). We are left with the problem of determining for which maps $\phi$ the curve $Z_n$ is not of general type, i.e, when $g_n \leq 1$ for all $n \geq 1$. Theorem \ref{iterative relationship} gives the desired functional relationship under the hypothesis that $\deg \phi \geq 2$, and a close analysis of the various cases encountered in the proof of Theorem \ref{iterative relationship} gives the bound of three arithmetic progressions found in Theorem \ref{main ML}, together with the information on $M$. 
\label{faltings}

We close this introduction with three additional results related to Corollary \ref{powercor}. Denote by $K^m$ the set $\{k^m : k \in K\}$. 

\begin{corollary} \label{powercor2}
Let $K$ be a finitely generated field of characteristic zero and let $\phi \in K(x)$ have degree $d \geq 2$. Suppose that there exists $a \in \mathbb{P}^1(K)$ with $O_\phi^+(a) \cap \mathbb{P}^1(K)^m$ infinite, for some $m \geq 5$ with $m \mid d$. Then $\phi(x) = (\psi(x))^m$ for some $\psi \in K(x)$.
\end{corollary}
Corollary \ref{powercor2} follows immediately from Corollary \ref{powercor} and the observation that if $\phi(x) = c(\psi(x))^m$ with $c \not\in K^m$, then for all $a \in K$, $O_\phi^+(a) \cap \mathbb{P}^1(K)^m \subseteq \{a, 0, \infty\}$, and thus is finite. 

When $\phi$ is a polynomial, we can give a particularly concrete version of Corollary \ref{powercor}:

\begin{corollary} \label{mainpoly}
Let $K$ be a finitely generated field of characteristic zero, let $\phi \in K[x]$ have degree $d \geq 2$, and fix $m \geq 2$. If there exists $a \in \mathbb{P}^1(K)$ with $O_\phi^+(a) \cap K^m$ infinite, then one of the following holds: 
\begin{enumerate}
\item $\phi(x) = cx^j(g(x))^m$ for some $g \in K[x]$, $0 \leq j \leq m-1$, and $c \in K^*$; or 
\item $m = 2$ and there is $c \in K^*$ such that $c\phi(x/c)$ is  
\begin{equation} \label{chebexception}
(-1)^d(T_d(x + 2)) - 2,
\end{equation}
where $T_d$ is the degree-$d$ monic Chebyshev polynomial. 
\end{enumerate}
\end{corollary}

Note that cases (1) and (2) of Corollary \ref{mainpoly} are mutually exclusive, unlike the cases in Theorem \ref{maingenus}.  Indeed, for all $d \geq 2$ we have that $T_d$ maps $-2$ to $2 \cdot (-1)^d$ with multiplicity 1, implying that the map in case (2) maps $-4$ to $0$ with multiplicity 1, and hence is not of the form given in case (1). The polynomials of the form \eqref{chebexception} are conjugates of $T_d$ that contain $0$ in their post-critical set but do not belong to case (1).  For $d = 2, 3, 4, 5$ these maps are:
$x(x + 4), -(x+4)(x+1)^2, x(x+4)(x+2)^2$, and $-(x+4)(x^2 + 3x + 1)^2,$ respectively. 

Our final corollary shows that when $K = \Q$ and $\deg \phi = 2$, we obtain very strong consequences when there is $a \in \Q$ with $O_\phi^+(a) \cap \Q^2$ infinite:
\begin{corollary} \label{quadpoly}
A quadratic polynomial $\phi \in \Q[x]$ has a rational orbit containing infinitely many distinct squares if and only if either
\begin{enumerate}
\item[(a)] \text{$\phi$ is the square of a linear polynomial with rational coefficients}, or 
\item[(b)] 
\text{$\phi(x) = cx^2 + 4x$ with $c \in \Q^*$}.
\end{enumerate} 
\end{corollary} 

The paper is organized as follows. Sections \ref{preliminaries}-\ref{four} contain the geometric portion of the proofs of Theorems \ref{iterative relationship} and \ref{maingenus}. 
In Section \ref{preliminaries} we study the genera of irreducible components of super-elliptic curves, and show in Corollary \ref{rhocor} that $g_n$ remains bounded as $n$ grows if and only if $0$ and $\infty$ satisfy a ramification condition on iterated preimages that we call $m$-branch abundance (see Definition \ref{rhodef}). We also show that if $g_n$ is unbounded, then it grows exponentially with $n$ (Theorem \ref{dichotomy}). 
In Section \ref{m5sec}, we study rational functions with two $m$-branch abundant points in $\PP^1(\C)$ when $m \geq 5$. In Section \ref{prime} we study ramification among iterated preimages of $m$-branch abundant points when $m$ is prime, with a view to understanding the most complicated cases $m = 2$ and $m = 3$; this culminates in two classification results (Theorems \ref{3class} and \ref{2class}). In Section \ref{four} we study maps with two $4$-branch abundant points. In Section \ref{fielddef} we give the results that handle the arithmetic portion of the proofs of Theorems \ref{iterative relationship} and \ref{maingenus}. In Section \ref{firstthmsec} we state useful results of Milnor \cite{milnor} on Latt\`es maps, and combine them with material from the five previous sections to prove Theorem \ref{maingenus}. In Section \ref{pfmain} we give the proof of Theorem \ref{iterative relationship}, which involves checking numerous cases. In Section \ref{pfremain} we give the proofs of the remaining results from the introduction. In Section \ref{example} we present the example mentioned after Theorem \ref{main ML}.

\section{$m$-branch abundant points and the genus of $\phi^n(x) = y^m$} \label{preliminaries}

Recall from the discussion before Theorem \ref{maingenus} the definition of the ramification index $e_\phi(z)$ of a rational function $\phi \in \C(x)$ at $z \in \PP^1(\C)$. 
We refer to $z \in \PP^1(\C)$ as a \textit{ramification point} for $\phi$ if $e_\phi(z) > 1$. 
An easy argument on compositions of power series gives the following special case of the chain rule for ramification indices (see \cite[Section 2.5]{beardon}): 
\begin{equation} \label{chain}
e_{\phi^n}(z) = \prod_{i = 0}^{n-1} e_\phi(\phi^{i}(z)),
\end{equation}
and hence $e_{\phi^n}(z)$ ``remembers" ramification of the map $\phi$ at each of $z, \phi(z), \ldots, \phi^{n-1}(z)$. 
An essential tool throughout the present paper comes in the form of the Riemann-Hurwitz formula (see e.g. \cite[Section 2.7]{beardon} for a proof):  
$$\sum_{z \in \PP^1(\C)} (e_{\phi}(z) - 1) = 2d - 2.$$

For $\alpha \in \C, \phi \in \C(x)$, and $n \geq 0$, we use the standard notation of $\phi^{-n}(\alpha)$ to denote the set $\{\beta \in \C : \phi^n(\beta) = \alpha\}$. We introduce the following terminology:
\begin{definition} \label{rhodef}
Fix $m \geq 2$, let $\phi \in \C(x)$ be non-constant, and let $\alpha \in \mathbb{P}^1(\C)$. Define $\rho_n(\alpha)$ to be the number of $z \in \phi^{-n}(\alpha)$ with $e_{\phi^n}(z)$ not divisible by $m$. We say that $\alpha$ is \textbf{$\boldsymbol m$-branch abundant} for $\phi$ if $\rho_n(\alpha)$ is bounded as $n \to \infty$.  
\end{definition}
Note that if $m_1 \mid m_2$, then $m_1 \nmid e_{\phi^n(z)}$ implies $m_2 \nmid e_{\phi^n(z)}$, and hence if and $\alpha$ is $m_2$-branch abundant for $\phi$, then $\alpha$ is also $m_1$-branch abundant for $\phi$.  We remark that in \cite{abc}, the authors call $\alpha \in \mathbb{P}^1(\C)$ \textit{dynamically ramified} for $\phi$ if the set $\bigcup_{n \geq 1} \{z \in \phi^{-n}(\alpha) : e_{\phi^{n}}(z) = 1\}$ is finite. The definition of an $m$-branch abundant point is weaker in that it only considers $z \in \phi^{-n}(\alpha)$ with $m \nmid e_{\phi^n}(z)$, and moreover it only asserts a bounded number of such points as $n$ grows, rather than finiteness of the full set of such points as $n$ varies. 

A primary goal of this section is to establish a relationship between the existence of $m$-branch abundant points for $\phi$ and the genus of (irreducible factors of) $C_n : \phi^n(x) = y^m$. For curves of this form, the irreducible factors are easily found: 

\begin{proposition} \label{irred}
Let $C$ be the curve defined (over $\C$) by $\psi(x) = y^m$, where $\psi(x) = c \prod_{i = 1}^k (x - \alpha_i)^{e_i} \in \C(x)$ and $e_i \in \Z \setminus \{0\}$ for all $i$. Let $a$ be the greatest positive integer dividing $m$ and all the $e_i$, and put $\lambda(x) = \sqrt[a]{c} \prod_{i = 1}^k (x - \alpha_i)^{e_i/a} \in \C(x)$ for some fixed choice of $\sqrt[a]{c}$.  Let $\zeta_a$ be a primitive $a^{\text{th}}$ root of unity. Then the irreducible factors of $C$ are the curves 
\begin{equation} \label{irredfactors}
y^{m/a} = \zeta_{a}^k\lambda(x), \qquad k = 0, \ldots, a-1. 
\end{equation}
\end{proposition}

\begin{proof}
We show that each curve is irreducible, and then it follows by comparing the degrees in $y$ that they must comprise all the irreducible components of $C$. By assumption $\zeta_{a}^k\lambda(x)$ is not a $p$th power in $\C(x)$ for any prime $p$ dividing $m$, and it follows that $y^{m/a} - \zeta_{a}^k\lambda(x)$ is irreducible as a polynomial in $y$, whence each of the curves in \eqref{irredfactors} is irreducible.  
\end{proof}

We may determine the genus of every curve of the form \eqref{irredfactors} quite explicitly:

\begin{proposition} \label{genus}
Let $C$ and $a$ be as in Proposition \ref{irred}, and put $m' = m/a$ and $e_i' = e_i/a$. Then every irreducible factor of $C$ has the same genus $g$, given by
\begin{equation} \label{genuseq}
g = 1 + \left(\frac{k-1}{2} \right)m' - \frac{1}{2} \left(\gcd(m',e_1' + \cdots + e_k') +  \sum_{i = 1}^k \gcd(m', e_i') \right).
\end{equation}
\end{proposition}

\begin{proof}
A straightforward application of Proposition \ref{irred} and a genus formula for irreducible curves given by variables-separated rational functions first used by Ritt \cite{ritt1}. The first explicit statement and proof of the formula in the general situation seems to be \cite[Proposition 2]{fried2}; for another statement and proof, see \cite[Proposition 4.1]{bilutichy}. Many other authors have used various versions of this formula, e.g. \cite[Proposition 2.6]{az2} and \cite[Corollary 2.1]{pak1}. 
Another proof of the present proposition may be given by noting that the genus of each irreducible factor is equal to the genus of the function field $\C(x, \sqrt[m']{\lambda(x)})$. Then one can directly apply the formula in \cite[Proposition 3.7.3]{stichtenoth} for the genus of a Kummer extension of function fields.
\end{proof}

\begin{corollary} \label{genuscor}
Let $g$ be as in Proposition \ref{genus}, and denote by $t$ the number of $i \in \{1, \ldots, k\}$ such that $m \nmid e_i$. If $t = 0$, then $g = 0$, and if $t > 0$ then 
\begin{equation} \label{genusbound}
\lceil (t/2) - 1 \rceil \leq g \leq (m-1)(t-1)/2,
\end{equation} 
where $\lceil \cdot \rceil$ denotes the ceiling function. 
\end{corollary}

\begin{remark} The bounds are sharp, as evidenced by the hyperelliptic curves $y^2 = x^t - 1$. 
\end{remark}

\begin{proof}
First note that $t = 0$ if and only if $m' = 1$, and in this case \eqref{genuseq} reduces to $g = 0$. Assume now that $t \geq 1$ and $m' \geq 2$. If $m \nmid e_i$, then $m' \nmid e_i'$, giving us $\gcd(m', e_i') \leq m'/2$.  
From \eqref{genuseq} we obtain 
\begin{equation} \label{two}
g \geq 1 + \left(\frac{k-1}{2} \right)m' - \frac{1}{2} \left(m'(1 + (k-t)) +  t \frac{m'}{2}\right) = 1 + m'\left(\frac{t}{4} - 1\right),
\end{equation}
with equality if and only if $m' \mid e_1' + \cdots + e_k'$ and $\gcd(m', e_i') = m'/2$ for each $i$ with $m' \nmid e_i'$. Because $m' \geq 2$, \eqref{two} gives
$g \geq (t/2) - 1.$ This establishes the lower bound of \eqref{genusbound} when $t$ is even. Assume then that $t$ is odd. Then if \eqref{two} is an equality, we have $\gcd(m', e_i') = m'/2$ for an odd number of values of $i$ and $\gcd(m', e_i') = m'$ for the rest. Thus $e_1' + \cdots + e_k' \equiv m'/2 \bmod{m'}$, and therefore $m' \nmid (e_1' + \cdots + e_k')$, a contradiction. We have shown that \eqref{two} is a strict inequality, giving $g > (t/2) - 1$. As $g$ is an integer, we conclude $g \geq  \lceil (t/2) - 1 \rceil$. 

To prove the upper bound of \eqref{genusbound}, note that \eqref{genuseq} gives
$$g \leq 1 + \left(\frac{k-1}{2} \right)m' - \frac{1}{2} \left(m'(k-t) +  t + 1 \right) = \frac{(m'-1)(t-1)}{2} \leq \frac{(m-1)(t-1)}{2},$$
as desired. 
\end{proof}

Write $\phi^n(x) = c \prod_{i = 1}^k (x - \alpha_i)^{e_i}$, and take $t_n$ to be the number of $i \in \{1, \ldots, k\}$ such that $m \nmid e_i$. Then $t_n$ is closely related to the quantity $\rho_n(0) + \rho_n(\infty)$, where $\rho_n$ is defined in Definition \ref{rhodef}. Indeed, $\rho_n(0) + \rho_n(\infty) = t_n$ unless $\infty \in \phi^{-n}(\infty) \cup \phi^{-n}(0)$ and $m \nmid e_{\phi^n}(\infty)$, in which case $\rho_n(0) + \rho_n(\infty) = t_n + 1$. We thus obtain:

\begin{corollary} \label{rhocor}
Let $\phi \in \C(x)$ have degree $d \geq 2$. For $n \geq 1$, let $C_n$ be the curve defined (over $\C$) by $\phi^n(x) = y^m$, let $g_n$ be the genus of every irreducible factor of $C_n$, and put $\rho_n(\phi) := \rho_n(0) + \rho_n(\infty)$, where $\rho_n(0)$ and $\rho_n(\infty)$ are as in Definition \ref{rhodef}. Then either $\rho_n(\phi) = g_n = 0$ or
\begin{equation*} 
\lceil (\rho_n(\phi) - 3)/2 \rceil \leq g_n \leq (m-1)(\rho_n(\phi)-1)/2
\end{equation*}
In particular, $g_n$ is bounded as $n \to \infty$ if and only if both $0$ and $\infty$ are $m$-branch abundant for $\phi$. 
\end{corollary}

A consequence of Corollary \ref{rhocor} is a result on the growth rate of $g_n$ as $n \to \infty$ in the case where $g_n$ is unbounded.

\begin{theorem} \label{dichotomy}
Let $\phi$, $C_n$ and $g_n$ be as in Corollary \ref{rhocor}. If $g_n$ is unbounded as $n \to \infty$, then $g_n \geq \kappa d^n$ for some constant $\kappa$. 
\end{theorem}

\begin{proof}
Because $g_n$ is unbounded, we have that $\rho_n(\phi)$ is unbounded, and without loss of generality say that $\rho_n(0)$ is unbounded. If $\rho_n(\phi) \geq \kappa d^n$, then after possibly adjusting $\kappa$ the same conclusion holds for $g_n$, whence it suffices to give an exponential lower bound for $\rho_n(0)$. 

Because $\rho_n(0)$ is unbounded, the set 
$$Z = \{z \in \PP^1(\C) : \text{$\phi^{k}(z)=0$ and $m \nmid e_{\phi^k}(z)$ for some $k \geq 1$\}}$$
is infinite.  
Observe that $O_\phi^+(0) \cap Z$ must be finite: if $0$ is periodic then $O_\phi^+(0)$ is itself finite, while if $0$ is not periodic then $O_\phi^+(0) \cap Z = \emptyset$.
Consider the set 
$$R = \{c \in \PP^1(\C) : \text{$e_\phi(c) > 1$ and $\phi^i(c) = 0$ for some $i \geq 0$}\}.$$
Now $O_\phi^+(c)\setminus O_\phi^+(0)$ is finite for each $c \in R$, and as $O_\phi^+(0) \cap Z$ is finite, we have that $O_\phi^+(c) \cap Z$ is finite. But $R$ is finite by Riemann-Hurwitz, and thus $\bigcup_{c\in R} O_\phi^+(c)$ contains only finitely many elements of $Z$.  Therefore there exists $z \in Z$ that is not in the orbit of any ramification point of $\phi$.  From the definition of $Z$, $\phi^{k}(z)=0$ for some $k \geq 1$.  Then for all $n\geq k$, $\rho_n(0)\geq (\frac{1}{d^k})(d^n)$, furnishing the desired exponential lower bound. 
\end{proof}

Observe that when combined with Theorem \ref{iterative relationship}, Theorem \ref{dichotomy} yields the result that the sequence $(g_n)_{n \geq 1}$ is either bounded by $1$ or grows exponentially.

\section{Maps with two $m$-branch abundant points, $m \geq 5$} \label{m5sec}

We begin with a definition and proposition that will be useful in proving Theorem \ref{maingenus}.

\begin{definition} \label{trivial} For a fixed integer $m \geq 2$, rational function $\phi \in \C(x)$, and distinct $\alpha_1, \alpha_2 \in \mathbb{P}^1(\C)$, we call $\phi$  \textbf{$\boldsymbol m$-trivial with respect to $\{\alpha_1, \alpha_2\}$} if we have $m \mid e_{\phi}(z)$ for all $z \in \phi^{-1}(\{\alpha_1, \alpha_2\}) \setminus \{\alpha_1, \alpha_2\}$. 
\end{definition}

\begin{proposition} \label{triv}
For any integer $m \geq 2$, a rational function $\phi \in \C(x)$ is $m$-trivial with respect to $\{0, \infty\}$ if and only if it is of the form
\begin{equation} \label{trivform10}
cx^j(\psi(x))^m \qquad \text{with $\psi(x) \in \C(x)$, $0 \leq j \leq m-1$, and $c \in \C^*$}.
\end{equation}
\end{proposition}
\begin{proof}
Suppose that $\phi$ is $m$-trivial with respect to $\{0, \infty\}$. For each $z \in \phi^{-1}(0) \setminus \{0,\infty\}$, the factor $(x-z)$ appears in the numerator of $\phi$ with multiplicity $e_{\phi}(z)$.  The same holds for $z \in \phi^{-1}(\infty) \setminus \{0,\infty\}$ and the denominator of $\phi$.   
Letting $U = \phi^{-1}(0)\setminus \{0,\infty\}$ and $V = \phi^{-1}(\infty)\setminus \{0,\infty\}$, we can write
$$
\phi(x) = c\cdot x^j \cdot  \frac{\prod_{u \in U} (x - u)^m}{\prod_{v \in V} (x - v)^m}
$$
for some $c \in \C^*$ (we cannot have $c = 0$ since $\phi$ is non-constant). Thus $\phi(x) = cx^j \psi(x)^m$ with $\psi(x) = \prod_{u \in U}(x-u)/ \prod_{v \in V} (x-v)$. If necessary, we may absorb $m$th powers of $x$ into $(\psi(x))^m$, allowing us to assume $0 \leq j \leq m-1$.

Suppose now that $\phi(x) = cx^j(\psi(x))^m$. Then $m \mid e_{\phi(z)}$ for all $z \in \phi^{-1}(\{0, \infty\}) \setminus \{0, \infty\}$, and it follows that $\phi$ is $m$-trivial with respect to $\{0, \infty\}$. 
\end{proof}

In light of Proposition \ref{triv} and Corollary \ref{rhocor}, in order to prove Theorem \ref{maingenus} we wish to show that in many cases a map for which $0$ and $\infty$ are $m$-branch abundant must in fact be $m$-trivial with respect to $\{0, \infty\}$. The purpose of this section is to prove this in the case $m \geq 5$, which is done in Theorem \ref{mge5}. There is no special advantage to assuming that $\phi$ has $0$ and $\infty$ as $m$-branch abundant points, and so we assume only that $\phi$ has two distinct $m$-branch abundant points $\alpha_1$ and $\alpha_2$.

We begin with several preparatory lemmas, which will be of use in later sections as well as this one. The first shows that when $m$ is a prime power, $m$-branch abundance of $\alpha \in \PP^1(\C)$ propagates to certain iterated preimages of $\alpha$.
\begin{lemma}\label{primelemma}
Let $\phi \in \C(x)$ and $p$ be prime. Suppose that $\alpha \in \PP^1(\C)$ is $p^r$-branch abundant for $\phi$, where $r \geq 1$, and let $\beta \in \PP^1(\C)$ satisfy $\phi^k(\beta) = \alpha$ for some $k \geq 1$. If $p^r \nmid e_{\phi^k}(\beta)$, then $\beta$ is $p$-branch abundant for $\phi$. Furthermore, if $p \nmid e_\phi(\phi^i(\beta))$ for each $i = 0, 1, \ldots, k-1$, then $\beta$ is $p^r$-branch abundant for $\phi$. 
\end{lemma}

\begin{proof}
Consider $z \in \phi^{-n}(\beta)$, implying in particular that $z \in \phi^{-(n+k)}(\alpha)$. Note that 
\begin{equation} \label{chainram}
e_{\phi^{n+k}}(z) = e_{\phi^k}(\phi^n(z)) \cdot e_{\phi^n}(z) = e_{\phi^k}(\beta) \cdot e_{\phi^n}(z).
\end{equation} 
If $p^r \nmid e_{\phi^k}(\beta)$, then \eqref{chainram} and the primality of $p$ give 
\begin{equation} \label{third}
\#\{z \in \phi^{-n}(\beta) : p \nmid e_{\phi^n}(z)\} \leq \#\{z \in \phi^{-(n+k)}(\alpha) : p^r \nmid e_{\phi^{n+k}}(z)\}.
\end{equation}
Because $\alpha$ is $p^r$-branch abundant, the right-hand side of \eqref{third} is bounded as $n$ grows, and thus $\beta$ is $p$-branch abundant. 

If $p \nmid e_\phi(\phi^i(\beta))$ for $i = 0, 1, \ldots, k-1$, then $p \nmid e_{\phi^k}(\beta)$ by \eqref{chain}. Arguing as in the previous paragraph, it follows that $\beta$ is $p^r$-branch abundant.
\end{proof}

For $\alpha \in \PP^1(\C)$ and $\phi \in \C(x)$, we often wish to consider the union of the sets $\phi^{-n}(\alpha)$ for $n \geq 0$. We thus introduce the following standard definition:
\begin{definition} \label{backdef}
Let $\phi \in \C(x)$ and $\alpha \in \PP^1(\C)$. The \textit{backwards orbit} \label{backorbit} of $\alpha$ under $\phi$ is 
\[
O^-_\phi(\alpha) := \{\beta \in \PP^1(\C) : \text{there exists $n \geq 0$ with $\phi^n(\beta) = \alpha$}\}, \label{backdef},
\] 
For $S \subset \PP^1(\C)$, the backwards orbit $O_\phi^-(S)$ of $S$ is the union of $O_\phi^-(\alpha)$ over $\alpha \in S$.
\end{definition}
Note that, as in the case with forward orbits, we have $\alpha \in O_\phi^{-}(\alpha)$.

The next lemma is crucial in our analysis. We often apply it to a preimage of a $p$-branch abundant point, and hence we use $\beta$ instead of $\alpha$ in the statement.

\begin{lemma}\label{comp3}
Let $S$ be a finite subset of $\PP^1(\C)$, and suppose that $\phi \in \C(x)$, $p$ is prime, and $\beta \in \phi^{-1}(S) \setminus S$ is $p$-branch abundant for $\phi$.
Then there exists $y \in O_\phi^-(\beta)$ satisfying the following conditions:
\begin{enumerate}
\item If $n\geq 0$ is minimal such that $\phi^{n}(y)=\beta$, then $S\cap \{y,\phi(y),\ldots,\phi^{n}(y)\}$ is empty. 
\item $p \mid e_{\phi}(z)$ for all $z \in \phi^{-1}(y)\setminus S$.
\end{enumerate}
Moreover, suppose that there are distinct $\{\beta_1, \ldots, \beta_k\} \subseteq \phi^{-1}(S) \setminus S$ and (not necessarily distinct) primes $p_1, \ldots, p_k$ such that for all $i = 1, \ldots, k$, we have that $\beta_i$ is $p_i$-branch abundant for $\phi$, and $y_i$ satisfies conditions (1) and (2) with respect to $\beta_i$ and $S$. Then $y_i \neq y_j$ for all $i \neq j$.
\end{lemma}

\begin{proof}
If each $z \in \phi^{-1}(\beta)\setminus S$ satisfies $p \mid e_{\phi}(z)$, then we may take $y = \beta$ (note $\beta\notin S$ by assumption). Otherwise, construct a (possibly finite) sequence $\gamma_1, \gamma_2, \ldots$ in $\PP^1(\C)$ as follows. Choose $\gamma_1 \in \phi^{-1}(\beta)\setminus S$ with $p \nmid e_\phi(\gamma_1)$. If $\gamma_i$ is chosen for $i \geq 1$,  then select $\gamma_{i+1} \in \phi^{-1}(\gamma_i)\setminus S$ with $p \nmid e_\phi(\gamma_{i+1})$.  If no such $\gamma_{i+1}$ exists, then the sequence terminates with $\gamma_i$, and thus we may take $y = \gamma_i$ to satisfy conditions (1) and (2) of the theorem. 

By construction, $\gamma_i \not\in S$ for all $i$.  Therefore all the $\gamma_i$ are distinct, for if $\gamma_i = \gamma_j$ for $i > j$, then $\gamma_i$ is periodic under $\phi$ and its orbit is $\{\gamma_i, \gamma_{i-1}, \ldots, \gamma_{j+1}\}$. But $\gamma_i \in O_\phi^-(S)$, and so $O_\phi^+(\gamma_i)$ intersects $S$, implying that $\gamma_\ell \in S$ for some $j < \ell \leq i$, which is a contradiction.

It thus suffices to show that the set $\{\gamma_i : i \geq 1\}$ is finite. Note that by Lemma \ref{primelemma}, each $\gamma_i$ is $p$-branch abundant for $\phi$. Consider the set $R$ of all $c \in \PP^1(\C)$ with $e_\phi(c) > 1$ and $c \in O_\phi^-(S)$. Observe that 
\begin{equation} \label{orbits}
\bigcup_{c\in R} O_\phi^+(c) \subseteq \left( \bigcup_{c\in R} (O_\phi^+(c)\setminus O_\phi^+(S)) \right) \cup O_\phi^+(S),
\end{equation} 
where $O_\phi^+(S) = \bigcup_{s \in S} O_\phi^+(s)$. 
Now for each $c \in R$, we have that $O_\phi^+(c)\setminus O_\phi^+(S)$ is finite, since $c \in O_\phi^-(S)$.  We claim that only finitely many of the $\gamma_i$ lie in $O_\phi^+(S)$. Otherwise, the finiteness of $S$ and the pigeonhole principle imply that infinitely many of the $\gamma_i$ lie in a single orbit $O_\phi^+(s)$ for some $s \in S$. Because
each $\gamma_i$ maps into $S$ under enough iterations of $\phi$, the orbit $O_\phi^+(s)$ visits $S$ infinitely often. The finiteness of $S$ then gives $\phi^{n_1}(s) = \phi^{n_2}(s)$ for some $n_1 \neq n_2$, and hence $O_\phi^+(s)$ is finite, contradicting our supposition that it contains infinitely many $\gamma_i$. Now from \eqref{orbits} we have that only finitely many of the $\gamma_i$ lie in $\bigcup_{c\in R} O_\phi^+(c)$. This implies there are only finitely many $\gamma_i$, since otherwise there is some $\gamma_i$ with no ramification point of $\phi$ in $O_\phi^-(\gamma_i)$, contradicting the $p$-branch abundance of $\gamma_i$.

To prove the last assertion of the lemma, assume to the contrary that $y_i=y_j$ for some $i \neq j$. Let $n_i \geq 0$ be minimal such that $\phi^{n_i}(y_i) = \beta_i$ and let $n_j \geq 0$ be minimal such that $\phi^{n_j}(y_j) = \beta_j$. Since $y_i = y_j$, we cannot have $n_i = n_j$, for then $\beta_i = \beta_j$. Assume without loss of generality that $n_i > n_j$, and note that 
$y_i = y_j$ implies 
\[
\{\beta_j,\phi(\beta_j), \ldots\phi^{n_i-n_j}(\beta_j)\} = \{\phi^{n_j}(y_i), \phi^{n_j+1}(y_i), \ldots\phi^{n_i}(y_i)\} \subseteq \{y_i,\phi(y_i), \ldots,\phi^{n_i}(y_i)\}.
\]
 But $S \cap \{y_i,\phi(y_i),...,\phi^{n_i}(y_i)\} = \emptyset$ by condition (1). Because $n_i - n_j \geq 1$, we have $\phi(\beta_j) \not\in S$, a contradiction. 
\end{proof}

Our next preparatory lemma is an elementary lower bound on ramification indices.
\begin{lemma}\label{lem:count'}
Let $m \in \Z$ with $m \geq 2$, let $T$ be a finite subset of $\PP^1(\C)$ with $\#T = t$, and let $\phi \in \C(x)$ have degree $d \geq 2$.  Let $U = \{z \in \phi^{-1}(T) : m \nmid e_\phi(z)\}$, and put $u = \#U$. Then
\[\sum_{z \in \phi^{-1}(T)} (e_{\phi}(z)-1) \geq \left( dt - u \right) \left(\frac{m-1}{m} \right),\]
where equality holds if and only if $e_\phi(z) \in \{1, m\}$ for all $z \in \phi^{-1}(T)$.
\end{lemma}
\begin{proof}
Let $U' = \phi^{-1}(T) \setminus U$. Because $\sum_{z \in \phi^{-1}(w)} e_\phi(z) = d$ for all $w \in \PP^1(\C)$, we have 
\[
dt = \sum_{z \in \phi^{-1}(T)} e_\phi(z) = \sum_{z \in U} e_\phi(z) +  \sum_{z \in U'} e_\phi(z) =  \sum_{z \in U} (e_\phi(z) - 1) + u +  m\left(\sum_{z \in U'} (r_z - 1) + \#U' \right),
\]
from which we have that $\#U' \leq \frac{dt - u}{m}$, with equality if and only if $e_\phi(z) = 1$ for all $z \in U$ and $r_z = 1$ for all $z \in U'$. The lemma then follows from the observation that 
\[\sum_{z \in \phi^{-1}(T)} (e_{\phi}(z)-1) = dt-\#(\phi^{-1}(T)) = dt - u - \#U'.\]
\end{proof}

Let $\alpha_1$ and $\alpha_2$ be $m$-branch abundant points for $\phi$.  Put
\begin{equation}
\label{bdef}
B = \{\beta \in \PP^1(\C) : \beta \in \phi^{-1}(\{\alpha_1, \alpha_2\}) \setminus \{\alpha_1, \alpha_2\}, m \nmid e_\phi(\beta)\} 
\end{equation}
From Definition \ref{trivial}, one sees immediately that $\phi$ is $m$-trivial with respect to $\{\alpha_1, \alpha_2\}$ if and only if $B$ is empty. 
Now because $m \nmid e_\phi(\beta)$ for each $\beta \in B$, there must be some prime $p_\beta$ and some $r \geq 1$ with $p_\beta^r \mid m$ but $p_\beta^r \nmid e_\phi(\beta)$.  Because $\alpha_1$ and $\alpha_2$ are $m$-branch abundant, they are also $p_\beta^r$-branch abundant, and so by Lemma \ref{primelemma}, $\beta$ is $p_\beta$-branch abundant. We may then apply Lemma \ref{comp3} with $S = \{\alpha_1, \alpha_2\}$ to find for each $\beta \in B$ some $y_\beta \in O_\phi^-(\beta)$ with $p_\beta \mid e_{\phi}(z)$ for each $z \in \phi^{-1}(y_\beta) \setminus \{\alpha_1, \alpha_2\}$. 
We then set 
\begin{equation} \label{ydef}
Y = \{y_\beta : \beta \in B\}.
\end{equation}
By the last assertion of Lemma \ref{comp3}, $Y$ has the same number of elements as $B$.

\begin{lemma}\label{inequality}
Let $m \in \Z$ with $m \geq 2$, let $\phi \in \C(x)$ have degree $d \geq 2$, and let $\alpha_1, \alpha_2 \in \PP^1(\C)$ be distinct $m$-branch abundant points for $\phi$. Let $B$ and $Y$ be as in \eqref{bdef} and \eqref{ydef}, respectively. Put $b = \#B$ and $\ell_Y = \#(\phi^{-1}(Y) \cap \{\alpha_1, \alpha_2\})$. Then
\begin{equation} \label{bound}
b(dm - 2m + 2) + \ell_Y(m-2) \leq 4d - 4.
\end{equation}
\end{lemma}

\begin{proof}
Let $p$ be the smallest prime dividing $m$, so that $p_\beta \geq p$ for all $\beta \in B$. By Lemma \ref{comp3} we have $\#Y = \#B = b$. 
Applying Lemma \ref{lem:count'} with $T = Y$ yields
$$\sum_{z \in \phi^{-1}(Y)} (e_{\phi}(z) - 1)\geq (bd-\ell_Y) \left(\frac{p-1}{p} \right) \geq \frac{bd-\ell_Y}{2}.$$
Now let $u_i = \#\{z \in \phi^{-1}(\alpha_i) : m \nmid e_{\phi}(z)\}$ for $i \in \{1, 2\}$.
Note that $\#\phi^{-1}(\alpha_i) \leq u_i+ (d-u_i)/m$, whence
\[
\sum_{z \in \phi^{-1}(\alpha_i)} (e_{\phi}(z)-1)  = d - \#\phi^{-1}(\alpha_i) \geq d- \left(u_i + \frac{d-u_i}{m}\right).
\]
By condition (1) of Lemma \ref{comp3}, we have that $Y \cap \{\alpha_1, \alpha_2\} = \emptyset$. Hence $\#(\phi^{-1}(\{\alpha_1, \alpha_2\}) \cap \{\alpha_1, \alpha_2\}) \leq 2  - \ell_Y$, and it follows that $u_1+u_2\leq b+2-\ell_Y$.
Thus, 
\begin{align*}
2d-2 = \sum_{z\in \PP^{1}(\C)}(e_{\phi}(z)-1) & \geq  \sum_{z\in \phi^{-1}(\{\alpha_1,\alpha_2\} \cup Y)}(e_{\phi}(z)-1) \\
& \geq d-\left(u_1+\frac{d-u_1}{m} \right)+d-\left(u_2+\frac{d-u_2}{m}\right)+\frac{bd-\ell_Y}{2} \\
& = (2d-(u_1+u_2))\frac{m-1}{m}+\frac{bd-\ell_Y}{2} \\
&\geq (2d-(b+2-\ell_Y))\frac{m-1}{m}+\frac{bd-\ell_Y}{2}.
\end{align*}
Multiplying through by $2m$ and regrouping terms yields the desired inequality. 
\end{proof}

We now prove the main theorem of this section.

\begin{theorem}\label{mge5}
Let $m \in \Z$ with $m\geq 5$.  Then every rational function $\phi \in \C(x)$ with two $m$-branch abundant points $\alpha_1, \alpha_2$ in $\PP^1(\C)$ is $m$-trivial with respect to $\{\alpha_1, \alpha_2\}$.
\end{theorem}
\begin{proof}
We use the notation of Lemma \ref{inequality}, and assume $b \geq 1$ in order to derive a contradiction.  
If $\ell_Y \geq 1$, then applying Lemma \ref{inequality} with $m \geq 5$ gives $b(5d - 8) + 3 \leq 4d - 4$. But $b \geq 1$, so this yields $5d - 5 \leq 4d - 4$, which is impossible because $d \geq 2$.

If $\ell_Y = 0$, then Lemma \ref{inequality} and $b \geq 1$ give $5d - 8 \leq 4d - 4$, which implies $d \leq 4$. 
Hence $m > d$, implying $\phi^{-1}(\{\alpha_1, \alpha_2\}) \subseteq B \cup \{\alpha_1, \alpha_2\}$, and therefore $\#\phi^{-1}(\{\alpha_1, \alpha_2\}) \leq b + 2$. 
Moreover, since $\ell_Y=0$, for each $y_\beta \in Y$ all elements of $\phi^{-1}(y_\beta)$ must have ramification index divisible by $p_\beta$, and in particular every element of $\phi^{-1}(Y)$ has ramification index greater than 1. When $d = 3$, this implies $e_\phi(z) = 3$ for all $z \in \phi^{-1}(Y)$, while for $d = 4$ we have $e_\phi(z) \geq 2$ for all $z \in \phi^{-1}(Y)$. In either case, $\sum_{z\in \phi^{-1}(Y)}(e_{\phi}(z)-1) \geq 2b$. 
Hence for $d \in \{3, 4\}$ we obtain
 \begin{equation} \label{estimate}
 \sum_{z\in \phi^{-1}(Y \cup \{\alpha_1,\alpha_2\})}(e_{\phi}(z)-1) \geq (2d-(b+2))+(2b)= 2d - 2 + b. 
 \end{equation}
Because $b \geq 1$ we have a contradiction to the Riemann-Hurwitz formula. 
When $d = 2$, we have only $\sum_{z\in \phi^{-1}(Y)}(e_{\phi}(z)-1) \geq b,$ and so $\sum_{z\in \phi^{-1}(Y \cup \{\alpha_1,\alpha_2\})}(e_{\phi}(z)-1) \geq (2d - (b+2)) + b = 2d - 2.$ Hence the inequality $\#\phi^{-1}(\{\alpha_1, \alpha_2\}) \geq b + 2$ is in fact an equality, and it follows that $\phi^{-1}(\{\alpha_1, \alpha_2\}) = B \cup \{\alpha_1, \alpha_2\}$. In particular, $\phi(\{\alpha_1, \alpha_2\}) \subseteq \{\alpha_1, \alpha_2\}$, implying that $O_\phi^+(\{\alpha_1, \alpha_2\}) \subseteq \{\alpha_1, \alpha_2\}$. Thus no element of $B$ can be periodic under $\phi$, for otherwise $B \cap O_\phi^+(\{\alpha_1, \alpha_2\}) \neq \emptyset$, contradicting the fact that by definition $B \cap \{\alpha_1, \alpha_2\} = \emptyset$.  Now let $\beta \in B$, and for $n \geq 1$ let $\gamma_n \in \phi^{-n}(\beta)$. Because $\beta$ is not periodic under $\phi$, we must have that $\gamma_n, \phi(\gamma_n), \ldots, \phi^n(\gamma_n)$ are all distinct. But 
\[ e_{\phi^{n}}(\gamma_n)=\prod_{i=0}^{n-1}e_\phi(\phi^{i}(\gamma_n)),\]
and there can be at most two $i$ with $e_\phi(\phi^{i}(\gamma_n)) = 2$, with the rest having $e_\phi(\phi^{i}(\gamma_n)) = 1$. It follows that $e_{\phi^{n}}(\gamma_n) \leq 4$. This holds for arbitrary $n$ and $\gamma_n$, and thus $\beta$ cannot be $m$-branch abundant, because $m \geq 5$. This contradiction completes the proof of the theorem.
\end{proof}

\section{preimage trees of $p$-branch abundant points} \label{prime}

In this section we study $m$-fold ramification among preimages of an $m$-branch abundant point. 
\begin{definition}
Fix $m \in \Z$ with $m \geq 2$, $\phi \in \C(x)$, and $\alpha \in \PP^1(\C)$. Given $z \in \PP^1(\C)$, denote by $r_\phi(z)$ the unordered tuple whose entries are $e_\phi(y)$ as $y$ varies over $\phi^{-1}(z)$. For $n \geq 0$, let $S_n$ be the set of $z \in \phi^{-n}(\alpha)$ with $m \nmid e_{\phi^n}(z)$. 
The \textbf{$\boldsymbol m$-ramification structure of $O^-_\phi(\alpha)$} is 
$$
\bigsqcup_{n \geq 0} \{(z, r_{\phi}(z)) : z \in S_n\}.
$$
\end{definition}
For example, let $T_6$ be the degree-6 monic Chebyshev polynomial, let $m = 2, \phi(x) = T_6(x+2) - 2$, and $\alpha = 0$. Then $S_0 = \{0\}$, $S_n = \{-4, 0\}$ for all $n \geq 1$, and the $2$-ramification structure of $O_\phi^-(0)$ is 
\begin{equation} \label{t6example}
\{(0, (1,1, 2, 2))\} \sqcup \{(-4, (2, 2, 2)), (0, (1,1, 2, 2))\} \sqcup \{(-4, (2, 2, 2)), (0, (1,1, 2, 2))\} \sqcup \cdots
\end{equation}

It is often convenient to represent $m$-ramification structures pictorially. We do this by constructing a diagram whose $n$th row consists of the elements of $S_{n}$, and where a line between $\gamma \in S_{n+1}$ and $\beta \in S_n$ indicates that $\phi(\gamma) = \beta$. We label such a line with $e_\phi(\gamma)$ in the case where $e_\phi(\gamma) > 1$.  To eliminate clutter, we indicate with a double line labeled by $n$ (resp. $n^*$) a set of points each of which has ramification index divisible by $n$ (resp. exactly $n$). To further simplify our diagrams, we omit repetition when it does not add novel information, such as when $S_{n+1}$ is identical to $S_n$. For example, a diagram representing the $2$-ramification structure in \eqref{t6example} is:
\begin{center}
\begin{tikzpicture}
	\Tree 
		[.\node (alpha) {$0$};
			\edge[double]; [.\phantom{$\alpha$} ] 
			[.$0$ ]
			[.\node (beta) {$-4$};
				\edge[double]; [.\phantom{$-$}
				]
			]
		]
	\node [below left=-.1cm of alpha,xshift=-.05cm] {\scriptsize $2^*$};
	\node [below=-.0cm of beta,xshift=.25cm] {\scriptsize $2^*$};
\end{tikzpicture}
\vspace{-0.2 in}
\end{center}
If we replace the two occurrences of $2^*$ by $2$, then the resulting diagram still describes the $2$-ramification structure in \eqref{t6example}, though it also describes others, e.g.,
$$\{(0, (1,1, 2, 4, 6))\} \sqcup \{(-4, (4, 4, 6)), (0, (1,1, 2, 4, 6))\} \sqcup \{(-4, (4, 4, 6)), (0, (1,1, 2, 4, 6))\} \sqcup \cdots$$

The main goal of this section is to study maps $\phi \in \C(x)$ for which $0$ and $\infty$ are $m$-branch abundant with $m \in \{2, 3\}$, which we do in Theorems \ref{3class} and \ref{2class}. Our main tool is a classification of the $p$-ramification structures of $O_\phi^-(\alpha)$, where $p$ is prime and $\alpha$ is a $p$-branch abundant point for $\phi$. This is done in Theorems \ref{class1} and \ref{class2}.

\begin{lemma}\label{lem:ndiv} 
Let $\phi \in \C(x)$ have degree $d \geq 2$, let $p$ be a prime with $p \nmid d$, and suppose that $\alpha \in \PP^1(\C)$ is $p$-branch abundant for $\phi$. Then $\alpha$ is periodic under $\phi$ and there is exactly one $\beta \in \phi^{-1}(\alpha)$ with $p\nmid e_{\phi}(\beta)$. Moreover, $\beta$ must be $p$-branch abundant for $\phi$, and $\beta$ must lie in $O_\phi^+(\alpha)$.
\end{lemma}
\begin{proof}
Because $p\nmid d$, there must be at least one $\beta \in \phi^{-1}(\alpha)$ with $p\nmid e_{\phi}(\beta)$. If $\beta = \alpha$, then evidently $\alpha$ is periodic under $\phi$ and $\beta \in O_\phi^+(\alpha)$. Assume $\beta \in \phi^{-1}(\alpha) \setminus \{\alpha\}$. By Lemma \ref{primelemma} we have that $\beta$ is $p$-branch abundant for $\phi$. Applying Lemma \ref{comp3} with $S = \{\alpha\}$, there exists $y \in O_\phi^-(\beta)$ with $p \mid e_\phi(z)$ for each $z \in \phi^{-1}(y) \setminus \{\alpha\}$. 
If $\alpha \not\in \phi^{-1}(y)$, then from $d=\sum_{z \in \phi^{-1}(y)}{e_{\phi}(z)}$ we have $p \mid d$, contrary to assumption. Hence $\phi(\alpha) = y$, and so $\alpha$ is periodic under $\phi$ and $\beta \in O_\phi^+(\alpha)$. 

It remains to show that $\beta$ is the unique element of $\phi^{-1}(\alpha)$ with ramification index not divisible by $p$. If $\beta'$ is another such element, then by the previous paragraph we have $\beta' \in O_\phi^+(\alpha)$. Thus both $\beta$ and $\beta'$ lie in the cycle $C$ to which $\alpha$ belongs. But the action of $\phi$ on $C$ is one-to-one, so $\phi(\beta) = \alpha = \phi(\beta')$ implies $\beta = \beta'$. 
\end{proof}

\begin{theorem} \label{class1}
Let $\phi \in \C(x)$ have degree $d \geq 2$, let $p$ be a prime with $p \nmid d$, and suppose that $\alpha \in \PP^1(\C)$ is $p$-branch abundant for $\phi$. Then the $p$-ramification structure for $O_\phi^-(\alpha)$ is one of the following, where $k_1, k_2,$ and $k_3$ are positive integers not divisible by $p$, and points named with distinct letters within a given diagram are distinct:

\vspace{0.1 in}
\begin{center}
\begin{tabular}{| l | l | l | l | l |}
\hline
1. & 2. & 3.& 4.&5.\\
\begin{tikzpicture}
	\Tree
		[.\node (alpha) {$\alpha$};
			\edge[double]; [.\phantom{$\alpha$} ] 
			[.$\alpha$ 
				\edge[draw=none]; [.{} 
					\edge[draw=none]; [.{} 
						\edge[draw=none]; [.{} ]
					]
				]
			]
		]
	\node [below left=.00cm of alpha, xshift=.2cm] {\scriptsize $p$};
	\node [below right=-.1cm of alpha, xshift=-.17cm] {\scriptsize $k$};
\end{tikzpicture}
&
\begin{tikzpicture}
	\Tree 
		[.\node (alpha) {$\alpha$};
			\edge[double]; [.\phantom{$\beta$} ] 
			[.\node (beta) {$\beta$};
				\edge[double]; [.\phantom{$\alpha$} ] 
				[.$\alpha$
					 \edge[draw=none]; [.{} 
					 	\edge[draw=none]; [.{} ]
					]
				]
			]
		]
	\node [below left=-.1cm of alpha, xshift=.05cm] {\scriptsize $p$};
	\node [below right=-.1cm of alpha, xshift=-.00cm] {\scriptsize $k_1$};
	\node [below left=.00cm of beta, xshift=.2cm] {\scriptsize $p$};
	\node [below right=-.1cm of beta, xshift=-.1cm] {\scriptsize $k_2$};
\end{tikzpicture}
&
\begin{tikzpicture}
	\Tree 
		[.\node (alpha) {$\alpha$};
			\edge[double]; [.\phantom{$\beta$} ] 
			[.\node (beta) {$\beta$};
				\edge[double]; [.\phantom{$\gamma$} ] 
				[.\node (gamma) {$\gamma$};
					\edge[double]; [.\phantom{$\alpha$} ] 
					[.$\alpha$ 
						\edge[draw=none]; [.{} ]
					]
				]
			]
		]
	\node [below left=-.15cm of alpha, xshift=.00cm] {\scriptsize $2$};
	\node [below right=-.1cm of alpha, xshift=-.00cm] {\scriptsize $k_1$};
	\node [below left=-.15cm of beta, xshift=.00cm] {\scriptsize $2$};
	\node [below right=-.1cm of beta, xshift=-.00cm] {\scriptsize $k_2$};
	\node [below left=-.1cm of gamma, xshift=.1cm] {\scriptsize $2$};
	\node [below right=-.1cm of gamma, xshift=-.1cm] {\scriptsize $k_3$};
\end{tikzpicture}
&
\begin{tikzpicture}
	\Tree 
		[.\node (alpha) {$\alpha$};
			\edge[double] node[left, pos=.2]{\scriptsize $3^*$}; [.\phantom{$\beta$} ] 
			[.\node (beta) {$\beta$};
				\edge[double] node[left, pos=.2] {\scriptsize $3^*$}; [.\phantom{$\gamma$} ] 
				[.\node (gamma) {$\gamma$};
					\edge[double] node[left, pos=.2]{\scriptsize $3^*$}; [.\phantom{$\alpha$} ] 
					[.$\alpha$ 
						\edge[draw=none]; [.{} ]
					]
				]
			]
		]
\end{tikzpicture}
&
\begin{tikzpicture}
	\Tree 
		[.\node (alpha) {$\alpha$};
			\edge[double] node[left, pos=.2]{\scriptsize $2^*$}; [.\phantom{$\beta$} ]
			[.\node (alpha) {$\beta$};
				\edge[double] node[left, pos=.2]{\scriptsize $2^*$}; [.\phantom{$\gamma$} ]
				[.\node (alpha) {$\gamma$};
					\edge[double] node[left, pos=.2]{\scriptsize $2^*$}; [.\phantom{$\delta$} ]
					[.$\delta$ 
						\edge[double] node[left, pos=.3]{\scriptsize $2^*$}; [.\phantom{$\alpha$} ]
						[.$\alpha$ ]
					]
				]
			]
		]
\end{tikzpicture}
\\
&&
$p=2$
&
$p=3$
&
$p=2$
\\\hline
\end{tabular}
\end{center}
\vspace{0.1 in}
\end{theorem}

\begin{proof}
Put $\alpha_1 = \alpha$, and note that by Lemma \ref{lem:ndiv} there is a unique $\alpha_{2} \in \phi^{-1}(\alpha_1)$ with $p \nmid e_\phi(\alpha_{i+1})$ and $\alpha_{2} \in O_\phi^+(\alpha_1)$. By Lemma \ref{primelemma} we have that $\alpha_2$ is $p$-branch abundant for $\phi$, and so we may apply Lemma \ref{lem:ndiv} to $\alpha_2$. Continuing in this fashion, we obtain a sequence $(\alpha_i)_{i \geq 1}$ in $\PP^1(\C)$ of $p$-branch abundant points for $\phi$ that satisfy $\phi(\alpha_{i+1}) = \alpha_i$ and $\alpha_{i+1} \in O_\phi^+(\alpha_{i})$ for all $i \geq 1$. The latter condition implies $O_\phi^+(\alpha_{i+1}) \subseteq O_\phi^+(\alpha_i)$ for all $i \geq 1$, and so $O_\phi^+(\alpha_{i+1}) \subseteq O_\phi^+(\alpha_1)$, implying $\alpha_{i+1} \in O_\phi^+(\alpha_1)$. But  Lemma \ref{lem:ndiv} shows that $\alpha_1$ is periodic under $\phi$. Hence $\alpha_i = \alpha_j$ for some $i > j$, which implies that $\alpha_1 = \alpha_{i - j +1}$. Let $n > 0$ be minimal such that $\alpha_1 = \alpha_{n+1}$.  Note that $n = 1$ gives $p$-ramification structure (1), while $n = 2$ gives $p$-ramification structure (2).

Assume that $n \geq 3$. Applying Lemma \ref{lem:count'} with $T = \{\alpha_i\}$ and summing over $i$ yields
\begin{align}
\sum_{i=1}^{n}{\sum_{z\in\phi^{-1}(\alpha_i)}(e_\phi(z)-1)} & \geq n(d-1)\frac{(p-1)}{p} \geq 3(d-1)\frac{(p-1)}{p}. \label{9bound}
\end{align}
If $p > 3$, we obtain a contradiction to Riemann-Hurwitz. If $p = 3$, then we obtain a similar contradiction unless both the inequalities in \eqref{9bound} are equalities. This holds only when $n = 3$ and $e_\phi(z) \in \{1, 3\}$ for all $z \in \phi^{-1}(\alpha_i)$, the latter by Lemma \ref{lem:count'}. This gives $p$-ramification structure (4). If $p=2$, then 
$n(d-1)(p-1)/p = (n/2)(d-1)$, and \eqref{9bound} contradicts Riemann-Hurwitz unless $n \leq 4$. The case $n=3$ gives $p$-ramification structure $(3)$. If $n = 4$, then the first inequality in \eqref{9bound} must be an equality, and hence we have equality in Lemma \ref{lem:count'} with $T = \{\alpha_i\}$. The latter happens if and only if $e_\phi(z) \in \{1,2\}$ for all $z \in \phi^{-1}(\alpha_i)$, which gives $p$-ramification structure (5). 
\end{proof}

We now move to the more complicated case where $p \mid d$. The following lemma is of central importance both in this section and in subsequent sections.  
\begin{lemma} \label{atmost4}
Let $\phi \in \C(x)$ have degree $d \geq 2$. If $p > 3$ is prime, then $\phi$ has at most two $p$-branch abundant points in $\PP^1(\C)$. Moreover, $\phi$ has at most three $3$-branch abundant points in $\PP^1(\C)$, and at most four $2$-branch abundant points in $\PP^1(\C)$.
If $\phi$ possesses a set $V$ of three $3$-branch abundant points (resp. four $2$-branch abundant points), then all ramification points of $\phi$ lie in $\phi^{-1}(V),$ and
$e_\phi(z) \in \{1, 3\}$ (resp. $e_\phi(z) \in \{1, 2\}$) for all $z \in \PP^1(\C)$.
\end{lemma}

\begin{proof}
Let $p$ be prime, and let $V = V_0 = \{\alpha_1, \ldots, \alpha_k\} \subset \PP^1(\C)$ be a set of distinct $p$-branch abundant points for $\phi$. For $i \geq 1$, put $V_i = \{z \in \phi^{-1}(V_{i-1}) : p \nmid e_\phi(z)\}$. Observe that $V_i$ consists of all $z \in \phi^{-i}(V_0)$ with $p \nmid e_{\phi^i}(z)$, and in particular, $\#V_i = \sum_{n = 1}^k \rho_i(\alpha_n)$ (notation as in Definition \ref{rhodef}). By the $p$-branch abundance of the $\alpha_i$, we have that $(\#V_i)_{i \geq 0}$ is bounded. Hence the sequence cannot be strictly increasing, and so there is $j \geq 1$ with $\#V_{j} \leq \#V_{j-1} $. Assume that $j$ is minimal with this property. Because $\#V_0 = k$, the minimality of $j$ ensures that $\#V_{j-1} \geq k$. Apply Lemma \ref{lem:count'} with $T = V_{j-1}$ to get
\begin{equation} \label{rameqn}
\sum_{z \in \phi^{-1}(V_{j-1})} (e_\phi(z) - 1) \geq ((\#V_{j-1})d - \#V_j)\frac{p-1}{p} \geq (\#V_{j-1})(d-1)\frac{p-1}{p} \geq k(d-1)\frac{p-1}{p}.
\end{equation} 
If $p > 3$, then Riemann-Hurwitz implies $k \leq 2$. If $p = 3$ (resp. $p  = 2$), then Riemann-Hurwitz implies $k \leq 3$ (resp. $k \leq 4$). If $p = k = 3$ or $p = 2, k = 4$, then by Riemann-Hurwitz again we have equality throughout \eqref{rameqn}. In particular, we have $\#V_{j-1} = k$, and the minimality of $j$ then gives $j = 1$.  Equality in \eqref{rameqn} also implies equality in Lemma \ref{lem:count'} with $T = V_{j-1} = V$, and thus $e_\phi(z) \in \{1, p\}$ for all $z \in \phi^{-1}(V)$. Finally, \eqref{rameqn} gives $\sum_{z \in \phi^{-1}(V)} (e_\phi(z) - 1) = 2d - 2$, implying that all ramification points for $\phi$ lie in $\phi^{-1}(V)$, and hence $e_\phi(z) \in \{1, p\}$ for all $z \in \PP^1(\C)$.
\end{proof}

\begin{theorem}\label{class2}
Let $\phi \in \C(x)$ have degree $d \geq 2$, let $p$ be a prime with $p \mid d$, and suppose that $\alpha \in \PP^1(\C)$ is $p$-branch abundant for $\phi$. Then the $p$-ramification structure for $O_\phi^-(\alpha)$ is one of the following, where points named with distinct letters within a given diagram are distinct:

\vspace{0.1 in}
\begin{center}
\begin{tabular}{| l | l | l | l |}
\hline
6. & 7. & 8. & 9. \\
\multicolumn{1}{| c |}{
\begin{tikzpicture}
	\Tree
		[.\node (alpha) {$\alpha$};
			\edge[double]; [.\phantom{$\alpha$}
				\edge[draw=none]; [.{}
					\edge[draw=none]; [.{} ]
				]
			]
		]
	\node [below=.05cm of alpha,xshift=-.15cm] {\scriptsize $p$};
\end{tikzpicture}
}
&
\multicolumn{1}{| c |}{
\begin{tikzpicture}
	\Tree 
		[.\node (alpha) {$\alpha$};
			\edge[double]; [.\phantom{$\alpha$} ] 
			[.$\alpha$ ]
			[.\node (beta) {$\beta$};
				\edge[double]; [.\phantom{$\beta$}
					\edge[draw=none]; [.{} ]
				]
			]
		]
	\node [below left=-.1cm of alpha,xshift=.05cm] {\scriptsize $p$};
	\node [below=.15cm of alpha, xshift=-.12cm] {\scriptsize $a$};
	\node [below right=-.15cm of alpha, xshift=.08cm] {\scriptsize $b$};
	\node [below=-.05cm of beta,xshift=.15cm] {\scriptsize $p$};
\end{tikzpicture}
}
&
\multicolumn{1}{| c |}{
\begin{tikzpicture}
	\Tree 
		[.\node (alpha) {$\alpha$};
			\edge[double]; [.\phantom{$\beta$} ]
			[.\node (beta1) {$\beta_1$};
				\edge[double]; [.\phantom{$\beta$} ]
			]
			[.\node (beta2) {$\beta_2$};
				\edge[double]; [.\phantom{$\beta$}
					\edge[draw=none]; [.{} ]
				]
			]
		]
	\node [below left=-.15cm of alpha,xshift=-.1cm] {\scriptsize $2$};
	\node [below=-.05cm of beta1,xshift=-.12cm] {\scriptsize $2$};
	\node [below=-.05cm of beta2,xshift=-.12cm] {\scriptsize $2$};
	\node [below=.08cm of alpha, xshift=-.15cm] {\scriptsize $a$};
	\node [below right=-.15cm of alpha, xshift=.12cm] {\scriptsize $b$};
\end{tikzpicture}
}
&
\begin{tikzpicture}
	\Tree
		[.\node (alpha) {$\alpha$};
			\edge [double]; [.\phantom{$\beta$} ]
			[.$\alpha$ ]
			[.\node (1) {$\beta_1$};
				\edge[double]; [.\phantom{$\beta$} ]
			]
			[.\node (2) {$\beta_2$};
				\edge[double]; [.\phantom{$\beta$}
					\edge[draw=none]; [.{} ]
				]
			]
		]
	\node [below left=-.25cm of alpha, xshift=-.25cm] {\scriptsize $3^*$};
	\node [below=-.05cm of 1,xshift=-.12cm] {\scriptsize $3^*$};
	\node [below=-.05cm of 2,xshift=-.12cm] {\scriptsize $3^*$};
\end{tikzpicture}
\\
&
	$p \mid (a+b)$
&
$\text{$p=2$; $a, b$ odd}$
&
$p=3$
\\\hline
10. & 11. & 12. &\\
\begin{tikzpicture}
	\Tree 
		[.\node (alpha) {$\alpha$};
			\edge[double]; [.\phantom{$\beta$} ] 
			[.$\alpha$ ]
			[.\node (beta) {$\beta$};
				\edge[double]; [.\phantom{$\gamma$} ]
				[.\node (gamma1) {$\gamma_1$};
					\edge[double]; [.\phantom{$\gamma$} ]
				]
				[.\node (gamma2) {$\gamma_2$};
					\edge[double]; [.\phantom{$\gamma$} ]
				]
			]
		]
	\node [below left=-.2cm of alpha,xshift=-.2cm] {\scriptsize $2^*$};
	\node [below left=-.15cm of beta,xshift=-.1cm] {\scriptsize $2^*$};
	\node [below=.00cm of gamma1,xshift=-.17cm] {\scriptsize $2^*$};
	\node [below=.00cm of gamma2,xshift=-.17cm] {\scriptsize $2^*$};
\end{tikzpicture}
&
\begin{tikzpicture}
	\Tree 
		[.\node (alpha) {$\alpha$};
			\edge[double]; [.\phantom{$\beta$} ]
			[.\node (beta1) {$\beta_1$};
				\edge[double]; [.\phantom{$\beta$} ]
			]
			[.\node (beta2) {$\beta_2$};
				\edge[double]; [.\phantom{$\alpha$} ]
				[.$\alpha$ ]
				[.\node (gamma) {$\gamma$};
					\edge[double]; [.\phantom{$\gamma$} ]
				]
			]
		]
	\node [below left=-.25cm of alpha,xshift=-.25cm] {\scriptsize $2^*$};
	\node [below=-.05cm of beta1,xshift=-.17cm] {\scriptsize $2^*$};
	\node [below left=-.2cm of beta2,xshift=.05cm] {\scriptsize $2^*$};
	\node [below=.00cm of gamma,xshift=-.17cm] {\scriptsize $2^*$};
\end{tikzpicture}
&
\begin{tikzpicture}
	\Tree 
		[.\node (alpha) {$\alpha$};
			\edge[double]; [.\phantom{$\alpha$} ]
			[.$\alpha$ ]
			[.\node (beta1) {$\beta_1$};
				\edge[double]; [.\phantom{$\beta$} ]
			]
			[.\node (beta2) {$\beta_2$};
				\edge[double]; [.\phantom{$\beta$} ]
			]
			[.\node (beta3) {$\beta_3$};
				\edge[double]; [.\phantom{$\beta$} 
					\edge[draw=none]; [.{} ]
				]
			]
		]
	\node [below left=-.1cm of alpha, xshift=-.4cm] {\scriptsize $2^*$};
	\node [below=-.05cm of beta1,xshift=-.17cm] {\scriptsize $2^*$};
	\node [below=-.05cm of beta2,xshift=-.17cm] {\scriptsize $2^*$};
	\node [below=-.05cm of beta3,xshift=-.17cm] {\scriptsize $2^*$};
\end{tikzpicture}
&
\\
$p=2$
&
$p=2$
&
$p=2$
&
\\\hline
\end{tabular}
\end{center}
\vspace{0.1 in}

\end{theorem}

\begin{proof}
For any $z\in \PP^1(\C)$, we define 
\[
u(z) = \#\{\beta\in \phi^{-1}(z) \: : \: p\nmid e_\phi(\beta)\} \qquad \text{and} \qquad u_0(z)=\#\{\beta\in \phi^{-1}(z) \setminus \{z\} \: : \: p\nmid e_\phi(\beta)\}.
\]
Note that $u(z) = u_0(z)$ or $u(z) = u_0(z) + 1$, with the latter holding if and only if $\phi(z) = z$ and $p \nmid e_\phi(z)$. 
We frequently use the observation that $p \mid d$ implies $u(z) \neq 1$ for all $z \in \PP^1(\C)$. For example, if $u_0(\alpha) = 0$, then $u(\alpha) \leq 1$, and so because $p \mid d$ we have $u(\alpha) = 0$, which gives $p$-ramification structure (6).

\textbf{Case 1a:} Let $p\geq 3$ and $u_0(\alpha) \geq 2$. Because $\alpha$ is $p$-branch abundant, Lemma \ref{primelemma} yields the same conclusion for $\beta \in \phi^{-1}(\alpha) \setminus \{\alpha\}$ with $p \nmid e_\phi(\beta)$. If $u_0(\alpha) \geq 3$, we thus have a set of four distinct $p$-branch abundant points for $\phi$, contradicting Lemma \ref{atmost4}.  Thus $u_0(\alpha) = 2$, and we let $\beta_1,\beta_2$ be the two elements of $\phi^{-1}(\alpha) \setminus \{\alpha\}$ with ramification index not divisible by $p$. Then $V = \{\alpha, \beta_1, \beta_2\}$ is a set of three $p$-branch abundant points for $\phi$, and by Lemma \ref{atmost4} this implies $p = 3$ and $e_\phi(z) \in \{1, 3\}$ for all $z \in \phi^{-1}(V)$.  In particular we have $e_\phi(\beta_1) = e_\phi(\beta_2) = 1$. Because $3 \mid d$, we must have $\alpha \in \phi^{-1}(\alpha)$ and $3 \nmid e_\phi(\alpha)$, whence $e_\phi(\alpha) = 1$. Now from Lemma \ref{atmost4}, $V$ contains all $3$-branch abundant points for $\phi$, and it follows from Lemma \ref{primelemma} that $e_\phi(z) = 3$ for all $z \in \phi^{-1}(V) \setminus V$. But $V \cap \phi^{-1}(\{\beta_1, \beta_2\}) = \emptyset$, for otherwise applying $\phi$ gives $\alpha \in \{\beta_1, \beta_2\}$. Hence $e_\phi(z) = 3$ for all $z \in \phi^{-1}(\{\beta_1, \beta_2\})$, giving $3$-ramification structure (9). 

\textbf{Case 1b:} Let $p\geq 3$ and $u_0(\alpha)=1$. Let $\beta$ be the unique element of $\phi^{-1}(\alpha)\setminus\{\alpha\}$ with $p\nmid e_\phi(\beta)$. Because $p \mid d$ we have $u(\alpha)=2$, whence $\phi(\alpha)=\alpha$.  From Lemma \ref{primelemma} we have that $\beta$ is $p$-branch abundant for $\phi$.  By our work in Case 1a, $u_0(\beta)=2$ implies that $\phi(\beta)=\beta$, contradicting $\alpha \neq \beta$.  If $u_0(\beta)=1$ then because $p \mid d$ we have $u(\beta) = 2 > u_0(\beta)$, and so $\phi(\beta)=\beta$, again a contradiction.  Thus we have $u_0(\beta)=0$, and therefore $u(\beta) = 0$, giving $p$-ramification structure (7).  

\textbf{Case 2a:} Let $p=2$ and $u_0(\alpha) \geq 3$. Arguing similarly to Case 1a, we must have $u_0(\alpha) = 3$. Let $\beta_1,\beta_2, \beta_3$ be the three elements of $\phi^{-1}(\alpha) \setminus \{\alpha\}$ with odd ramification index. Then $V = \{\alpha, \beta_1, \beta_2, \beta_3\}$ is a set of four $2$-branch abundant points for $\phi$, and by Lemma \ref{atmost4} this implies $e_\phi(z) \in \{1, 2\}$ for all $z \in \phi^{-1}(V)$. Because $d$ is even, we must have $\alpha \in \phi^{-1}(\alpha)$ and $2 \nmid e_\phi(\alpha)$, whence $e_\phi(\alpha) = 1$. Lemma \ref{atmost4} shows that $V$ contains all $2$-branch abundant points for $\phi$, and it follows from Lemma \ref{primelemma} that $e_\phi(z) = 2$ for all $z \in \phi^{-1}(V) \setminus V$. But $V \cap \phi^{-1}(\{\beta_1, \beta_2, \beta_3\}) = \emptyset$, for otherwise applying $\phi$ gives $\alpha \in \{\beta_1, \beta_2, \beta_3\}$. Hence $e_\phi(z) = 2$ for all $z \in \phi^{-1}(\{\beta_1, \beta_2, \beta_3\})$, giving $2$-ramification structure (12).

\textbf{Case 2b:}  
Let $p=2$ and $u_0(\alpha)=2$. 
The latter implies $u(\alpha) \in \{2,3\}$, but $u(\alpha)$ must be even because $d$ is, giving $u(\alpha)=2$ and hence $\alpha \not\in \phi^{-1}(\alpha)$. Let  
$\beta_1,\beta_2$ be the two elements of $\phi^{-1}(\alpha) \setminus \{\alpha\}$ with odd ramification index. Note that $\beta_1$ and $\beta_2$ are both 2-branch abundant by Lemma \ref{primelemma}, and so Lemma \ref{atmost4} implies $u_0(\beta_i) \leq 3$ for $i = 1, 2$. Moreover, neither of the $\beta_i$ can have $u_0(\beta_i)=3$, because then $\phi(\beta_i)=\beta_i$ by Case 2a, giving a contradiction.  

Suppose that $u_0(\beta_i) = 2$ for some $i$ (say without loss of generality $i =2$), and let $z_1, z_2$ be the elements of $\phi^{-1}(\beta_2)\setminus \{\beta_2\}$ with odd ramification index. Then $V = \{\alpha, \beta_1, \beta_2, z_1, z_2\}$ is a set of $2$-branch abundant points for $\phi$, and Lemma \ref{atmost4} gives $\#V \leq 4$. But $\#\{\beta_1, \beta_2, z_1, z_2\} = 4$ and $\alpha \not\in \{\beta_1, \beta_2\}$ by construction, whence $\alpha \in \{z_1, z_2\}$. Without loss of generality say $\alpha = z_1$. Because $\#V = 4$, Lemma \ref{atmost4} shows that $e_\phi(z) \in \{1, 2\}$ for all $z \in \phi^{-1}(V)$. Lemma \ref{atmost4} also shows that $V$ contains all $2$-branch abundant points for $\phi$, and it follows from Lemma \ref{primelemma} that $e_\phi(z) = 2$ for all $z \in \phi^{-1}(V) \setminus V$. Note that $\phi(V) = \{\alpha, \beta_2\}$, and so if $V \cap \phi^{-1}(\{\beta_1, z_2\}) \neq \emptyset$, then applying $\phi$ gives $\{\beta_1, z_2\} \cap \{\alpha, \beta_2\} \neq \emptyset$, which is impossible. Hence $e_\phi(z) = 2$ for all $z \in \phi^{-1}(\{\beta_1, z_2\})$, which gives 2-ramification structure (11). 

Suppose that $u_0(\beta_i) \leq 1$ for $i = 1, 2$. Because $\beta_i \neq \alpha$, we have $\phi(\beta_i) \neq \beta_i$, and thus $u(\beta_i) = u_0(\beta_i)$ for $i = 1, 2$. Because $2 \mid d$ we cannot have $u(\beta_i) = 1$, which proves that $u(\beta_1)=u(\beta_2)=0$. This gives 2-ramification structure (8). 

\textbf{Case 2c:} Let $p=2$ and $u_0(\alpha) = 1$. Then there exists a unique $\beta \in \phi^{-1}(\alpha)\setminus \{\alpha\}$ with $e_\phi(\beta)$ odd. Because $u(\alpha) \neq 1$, we must have $u(\alpha) > u_0(\alpha)$, implying that $\phi(\alpha) = \alpha$ and $e_\phi(\alpha)$ is odd. Note that $u_0(\beta) \leq 3$ by Lemma \ref{atmost4}. If $u_0(\beta) = 3$, then by Case 2a we have $\phi(\beta) = \beta$, contradicting $\beta \neq \alpha$. We also cannot have $u_0(\beta) = 1$, for then $u(\beta) = 2$, and so again $\phi(\beta) = \beta$.  If $u_0(\beta) = 0$, we have $2$-ramification structure (7).

Suppose then that $u_0(\beta) = 2$, and let $\gamma_1,\gamma_2$ be the two elements of $\phi^{-1}(\beta) \setminus \{\beta\}$ with odd ramification index. Note that $\alpha \not\in \{\gamma_1, \gamma_2\}$, for otherwise $\phi(\alpha) = \alpha$ gives the contradiction $\beta = \alpha$. Thus $V = \{\alpha, \beta, \gamma_1, \gamma_2\}$ is a set of four 2-branch abundant points for $\phi$, and Lemma \ref{atmost4} shows that $e_\phi(z) \in \{1, 2\}$ for all $z \in \phi^{-1}(V)$. Lemma \ref{atmost4} also shows that $V$ contains all $2$-branch abundant points for $\phi$, and it follows from Lemma \ref{primelemma} that $e_\phi(z) = 2$ for all $z \in \phi^{-1}(V) \setminus V$. Note that $\phi(V) = \{\alpha, \beta\}$, and so if $V \cap \phi^{-1}(\{\gamma_1, \gamma_2\}) \neq \emptyset$, then applying $\phi$ gives $\{\gamma_1, \gamma_2\} \cap \{\alpha, \beta\} \neq \emptyset$, which is impossible. Hence $e_\phi(z) = 2$ for all $z \in \phi^{-1}(\{\gamma_1, \gamma_2\})$, which gives 2-ramification structure (10). 
\end{proof}

If $\alpha \in \PP^1(\C)$ is $p$-branch abundant for $\phi \in \C(x)$, define 
$$
Ab(\alpha) = \bigcup_{n \geq 1} \{z \in \phi^{-n}(\alpha) : p \nmid e_{\phi^n}(z)\},
$$
and note that by Lemma \ref{primelemma}, $Ab(\alpha)$ consists of $p$-branch abundant points for $\phi$. In the notation of Theorems \ref{class1} and \ref{class2}, the named points in each $p$-ramification structure comprise $Ab(\alpha)$. Note that it follows from Lemma \ref{primelemma} that if $\alpha_1$ and $\alpha_2$ are $p$-branch abundant points for $\phi$ and $\alpha_2 \in Ab(\alpha_1)$, then $Ab(\alpha_2) \subseteq Ab(\alpha_1)$. If $\alpha_1, \ldots, \alpha_n$ are $p$-branch abundant points for $\phi$, we write $Ab(\alpha_1, \ldots, \alpha_n)$ for $\bigcup_{i=1}^n Ab(\alpha_i)$.

\begin{theorem} \label{3class} 
Let $\phi \in \C(x)$ have degree $d \geq 2$, and assume that $A = \{\alpha_1, \alpha_2\} \subset \PP^1(\C)$ is a set of distinct $3$-branch abundant points for $\phi$. Suppose that $\phi$ is not $3$-trivial with respect to $A$, and let $\mu$ be a M\"obius transformation exchanging $\alpha_1$ and $\alpha_2$.  Then for either $\phi$ or $\mu \circ \phi \circ \mu^{-1}$, one of the following holds:
\begin{enumerate}
\item[(3A)] $O_\phi^-(\alpha_1)$ has $3$-ramification structure (4), $\alpha_2 \in Ab(\alpha_1)$, and $\phi(\alpha_2) = \alpha_1$; 
\item[(3B)] $O_\phi^-(\alpha_1)$ (resp. $O_\phi^-(\alpha_2)$) has $3$-ramification structure (2) (resp. (1)), and $Ab(\alpha_1) \cap Ab(\alpha_2) = \emptyset$;
\item[(3C)] $O_\phi^-(\alpha_1)$ has $3$-ramification structure (9), $\alpha_2 \in Ab(\alpha_1)$, and $\phi(\alpha_2) = \alpha_1$.
\end{enumerate}
Moreover, in all cases we have
\begin{equation} \label{strong}
\text{all ramification points of $\phi$ lie in $\phi^{-1}(Ab(\alpha_1, \alpha_2))$, and $e_\phi(z) \in \{1, 3\}$ for all $z \in \PP^1(\C)$.}
\end{equation}
\end{theorem}

\begin{remark}
The conditions in \eqref{strong} are invariant under M\"obius conjugation, and thus hold for both $\phi$ and $\mu \circ \phi \circ \mu^{-1}$.
\end{remark}

\begin{proof}
Let $O_\phi^-(\alpha_1)$ have $3$-ramification structure (a) and $O_\phi^-(\alpha_1)$ have $3$-ramification structure (b), where we use the numbering of Theorems \ref{class1} and \ref{class2}. Replacing $\phi$ with $\mu \circ \phi \circ \mu^{-1}$ if necessary, we assume that $a \geq b$. In the case where $a = b$ clearly it is not necessary to replace $\phi$ by $\mu \circ \phi \circ \mu^{-1}$ in order to obtain $a \geq b$, and so we are free to make this replacement for other purposes. Because $\phi$ is assumed to be non-$3$-trivial with respect to $A$, we must have $a \not\in \{1, 6\}$. 

Suppose first that $3 \nmid \deg \phi$. If $a = 4$ then $Ab(\alpha_1)$ contains three $3$-branch abundant points for $\phi$, and hence by Lemma \ref{atmost4} we have that $\alpha_2 \in Ab(\alpha_1)$ and \eqref{strong} holds. Because $Ab(\alpha_1)$ consists of a 3-cycle, we may replace $\phi$ with $\mu \circ \phi \circ \mu^{-1}$ if necessary to obtain $\phi(\alpha_2) = \alpha_1$. This gives (3A).
Suppose that $a = b = 2$, and note that $Ab(\alpha_i)$ is invariant under $\phi$ for $i = 1, 2$. It follows that either $Ab(\alpha_1) \cap Ab(\alpha_2) = \emptyset$ or $Ab(\alpha_1) = Ab(\alpha_2)$. The former contradicts Lemma \ref{atmost4}, while the latter implies that $\phi$ is $3$-trivial with respect to $A$. This leaves us with $a = 2$ and $b = 1$. In this case $\phi(\alpha_2) = \alpha_2$, and so $Ab(\alpha_2) \cap Ab(\alpha_1) = \emptyset$, which is (3B). Lemma \ref{atmost4} then gives that \eqref{strong} holds. 

Suppose now that $3 \mid \deg \phi$.  If $a = 9$ then $Ab(\alpha_1)$ contains three $3$-branch abundant points for $\phi$, and hence by Lemma \ref{atmost4} we have $\alpha_2 \in Ab(\alpha_1)$ and \eqref{strong} holds. Replacing $\phi$ by $\mu \circ \phi \circ \mu^{-1}$ if necessary, we have $\phi(\alpha_2) = \alpha_1$, giving (3C). 
If $a = b = 7$, then both $\alpha_1$ and $\alpha_2$ are fixed points of $\phi$, and thus $Ab(\alpha_1)$ and $Ab(\alpha_2)$ are disjoint, contradicting Lemma \ref{atmost4}. If $a = 7, b = 6$, and $\alpha_2 \not \in Ab(\alpha_1)$, then Lemma \ref{atmost4} gives $e_\phi(z) \in \{1,3\}$ for all $z \in \PP^1(\C)$, and in particular $d = \sum_{z \in \phi^{-1}(\alpha_1)} e_\phi(z) \equiv 2 \bmod{3}$, contrary to supposition. Hence $\alpha_2 \in Ab(\alpha_1)$, implying that $\phi$ is $3$-trivial with respect to $A$.
\end{proof}

\begin{theorem} \label{2class}
Let $\phi \in \C(x)$ have degree $d \geq 2$, and assume that $A = \{\alpha_1, \alpha_2\} \subset \PP^1(\C)$ is a set of distinct $2$-branch abundant points for $\phi$. Suppose that $\phi$ is not $2$-trivial with respect to $A$, and let $\mu$ be a M\"obius transformation exchanging $\alpha_1$ and $\alpha_2$. Then for either $\phi$ or $\mu \circ \phi \circ \mu^{-1}$, one of the following holds: 
\begin{enumerate}
\item[(2A)] $O_\phi^-(\alpha_1)$ has $2$-ramification structure (5), $\alpha_2 \in Ab(\alpha_1)$, and $\phi(\alpha_2) = \alpha_1$;
\item[(2B)] $O_\phi^-(\alpha_1)$ has $2$-ramification structure (5), $\alpha_2 \in Ab(\alpha_1)$, $\phi(\alpha_2) \neq \alpha_1$, and $\phi^2(\alpha_2) = \alpha_1$;
\item[(2C)] $O_\phi^-(\alpha_1)$ (resp. $O_\phi^-(\alpha_2)$) has $2$-ramification structure (3) (resp. (1)), and $Ab(\alpha_1) \cap Ab(\alpha_2) = \emptyset$;
\item[(2D)] $O_\phi^-(\alpha_1)$ has $2$-ramification structure (3), $\alpha_2 \in Ab(\alpha_1)$, and $\phi(\alpha_2) = \alpha_1$;
\item[(2E)] $O_\phi^-(\alpha_1)$ and $O_\phi^-(\alpha_1)$ have $2$-ramification structure (2), and $Ab(\alpha_1) \cap Ab(\alpha_2) = \emptyset$;
\item[(2F)] $O_\phi^-(\alpha_1)$ (resp. $O_\phi^-(\alpha_2)$) has $2$-ramification structure (2) (resp. (1)), and $Ab(\alpha_1) \cap Ab(\alpha_2) = \emptyset$;
\item[(2G)] $O_\phi^-(\alpha_1)$ has $2$-ramification structure (12), $\alpha_2 \in Ab(\alpha_1)$, and $\phi(\alpha_2) = \alpha_1$;
\item[(2H)] $O_\phi^-(\alpha_1)$ has $2$-ramification structure (11), $\alpha_2 \in Ab(\alpha_1)$, $\phi(\alpha_2) = \alpha_1$, and $\phi(\alpha_1) = \alpha_2$;
\item[(2I)] $O_\phi^-(\alpha_1)$ has $2$-ramification structure (11), $\alpha_2 \in Ab(\alpha_1)$, $\phi(\alpha_2) = \alpha_1$, and $\phi(\alpha_1) \neq \alpha_2$;
\item[(2J)] $O_\phi^-(\alpha_1)$ has $2$-ramification structure (11), $\alpha_2 \in Ab(\alpha_1)$, $\phi(\alpha_2) \neq \alpha_1$, and $\phi^2(\alpha_2) = \alpha_1$;
\item[(2K)] $O_\phi^-(\alpha_1)$ has $2$-ramification structure (10), $\alpha_2 \in Ab(\alpha_1)$, and $\phi(\alpha_2) = \alpha_1$;
\item[(2L)] $O_\phi^-(\alpha_1)$ has $2$-ramification structure (10), $\alpha_2 \in Ab(\alpha_1)$, $\phi(\alpha_2) \neq \alpha_1$, and $\phi^2(\alpha_2) = \alpha_1$;
\item[(2M)] $O_\phi^-(\alpha_1)$ has $2$-ramification structure (8), $\alpha_2 \in Ab(\alpha_1)$, and $\phi(\alpha_2) = \alpha_1$;  
\item[(2N)] $O_\phi^-(\alpha_1)$ and $O_\phi^-(\alpha_1)$ have $2$-ramification structure (7), and $Ab(\alpha_1) \cap Ab(\alpha_2) = \emptyset$;
\item[(2O)] $O_\phi^-(\alpha_1)$ (resp. $O_\phi^-(\alpha_2)$) has $2$-ramification structure (7) (resp. (6)), and $Ab(\alpha_1) \cap Ab(\alpha_2) = \emptyset$.
\end{enumerate}
Moreover, in all cases except (2D), (2F), (2M), and (2O), we have
\begin{equation} \label{strong2}
\text{all ramification points of $\phi$ lie in $\phi^{-1}(Ab(\alpha_1, \alpha_2))$, and $e_\phi(z) \in \{1, 2\}$ for all $z \in \PP^1(\C)$.}
\end{equation}
\end{theorem}

\begin{proof}
Similarly to the proof of Theorem \ref{3class}, we let $O_\phi^-(\alpha_1)$ have $2$-ramification structure (a) and $O_\phi^-(\alpha_1)$ have $2$-ramification structure (b), and we assume that $a \geq b$. Because $\phi$ is assumed to be non-$2$-trivial with respect to $A$, we must have $a \not\in \{1, 6\}$. 

Suppose that $2 \nmid \deg \phi$. If $a = 5$ then $Ab(\alpha_1)$ contains four $2$-branch abundant points for $\phi$, and hence by Lemma \ref{atmost4} we have that $\alpha_2 \in Ab(\alpha_1)$ and \eqref{strong} holds. Replacing $\phi$ with $\mu \circ \phi \circ \mu^{-1}$ if necessary, we have either $\phi(\alpha_2) = \alpha_1$ or $\phi(\alpha_2) \neq \alpha_1$ and $\phi^2(\alpha_2) = \alpha_1$, giving (2A) and (2B), respectively.

If $a = 3$, then observe that $b \in \{1, 2, 3\}$, and it follows that both $Ab(\alpha_1)$ and $Ab(\alpha_2)$ are invariant under $\phi$. If $Ab(\alpha_1) \cap Ab(\alpha_2) = \emptyset$, then from Lemma \ref{atmost4} we must have that $b = 1$ and \eqref{strong2} holds. This gives (2C). 
If $Ab(\alpha_1) \cap Ab(\alpha_2) \neq \emptyset$, then the invariance of $Ab(\alpha_1)$ and $Ab(\alpha_2)$ under $\phi$ implies $\alpha_2 \in Ab(\alpha_1)$, and hence $b = 3$. Replacing $\phi$ by $\mu \circ \phi \circ \mu^{-1}$ if necessary, we have $\phi(\alpha_2) = \alpha_1$. This is (2D). 

If $a = b = 2$, then as in the previous paragraph we have that $Ab(\alpha_i)$ is invariant under $\phi$ for $i = 1, 2$. It follows that either $Ab(\alpha_1) \cap Ab(\alpha_2) = \emptyset$ or $Ab(\alpha_1) = Ab(\alpha_2)$. In the former case, Lemma \ref{atmost4} shows that \eqref{strong2} holds, giving (2E). In the latter case, $\phi$ is $2$-trivial with respect to $A$. This leaves us with $a = 2$ and $b = 1$. In this case $\phi(\alpha_2) = \alpha_2$, and so $Ab(\alpha_2) \cap Ab(\alpha_1) = \emptyset$, which is (2F). 

Suppose now that $2 \mid \deg \phi$.  If $a \in \{10, 11, 12\},$ then $Ab(\alpha_1)$ contains four $2$-branch abundant points for $\phi$, and hence by Lemma \ref{atmost4} we have $\alpha_2 \in Ab(\alpha_1)$ and \eqref{strong} holds. If $a = 12$, then because $\alpha_1 \neq \alpha_2$, we must have $\phi(\alpha_2) = \alpha_1$, and so (2G) holds. If $a = 11$ and $\phi(\alpha_2) = \alpha_1$, then either both $\alpha_1$ and $\alpha_2$ lie in the 2-cycle that is part of 2-ramification structure (11), or only $\alpha_1$ lies in said 2-cycle. These give (2H) and (2I), respectively. If $a = 11$ and $\phi(\alpha_2) \neq \alpha_1$, then $\alpha_1$ and $\phi(\alpha_2)$ must lie in the 2-cycle that is part of 2-ramification structure (11), and thus $\phi^2(\alpha_2) = \alpha_1$, giving (2J). If $a = 10$, then we clearly have either (2K) or (2L). 

If $a = 8$ and $Ab(\alpha_1) \cap Ab(\alpha_2) = \emptyset$, then from Lemma \ref{atmost4} we have $b = 6$. Then applying Lemma \ref{lem:count'} with $T = Ab(\alpha_1) \cup Ab(\alpha_2)$ gives
$$
\sum_{z \in \phi^{-1}(T)} (e_\phi(z) - 1) \geq \frac{4d-2}{2} > 2d - 2.
$$
Hence $\alpha_2 \in Ab(\alpha_1)$, and necessarily $\phi(\alpha_2) = \alpha_1$, giving (2M).

If $a = b = 7$, then both $\alpha_1$ and $\alpha_2$ are fixed points of $\phi$, and thus $Ab(\alpha_1)$ and $Ab(\alpha_2)$ are disjoint. By Lemma \ref{atmost4}, we have that \eqref{strong2} holds, giving (2N). If $a = 7, b = 6$, and $Ab(\alpha_2) \cap Ab(\alpha_1) \neq \emptyset$, then $\alpha_2 \in Ab(\alpha_1)$, implying that $\phi$ is $2$-trivial with respect to $A$. If $Ab(\alpha_2) \cap Ab(\alpha_1) \neq \emptyset$, we have (2O).
\end{proof}

\section{Maps with two $m$-branch abundant points, $m = 4$} \label{four}

In this section we study rational functions with two 4-branch abundant points $\alpha_1$ and $\alpha_2$. In Theorem \ref{main4thm}, we show that either such a map is $4$-trivial with respect to $\{\alpha_1, \alpha_2\}$ (see Definition \ref{trivial}), or the 4-ramification structure of $O_\phi^-(\alpha_1)$ has a very restricted form, and in particular $\alpha_2 \in Ab(\alpha_1)$ with $\phi(\alpha_2) = \alpha_1$. This is done in Theorem \ref{main4thm}.

\begin{theorem} \label{51}
Suppose $\phi \in \C(x)$ has degree $d$ with $d$ odd, and let $\alpha_1, \alpha_2 \in \PP^1(\C)$ be distinct $4$-branch abundant points for $\phi$.  Then $\phi$ is $4$-trivial with respect to $\{\alpha_1, \alpha_2\}$.
\end{theorem}
\begin{proof} Any point $\alpha$ that is 4-branch abundant for $\phi$ is also $2$-branch abundant for $\phi$, and by the classification of $2$-branch abundant points (Theorem \ref{class1}), $\alpha$ must be periodic with its orbit consisting only of points with odd ramification index. If $w \in \phi^{-1}(\alpha)$ has ramification index divisible by $2$ but not by $4$, then by Lemma \ref{primelemma} we have that $w$ is $2$-branch abundant, and hence by Theorem \ref{class1} $w$ must be periodic.  But then $w \in O_\phi^+(\alpha)$, and the evenness of $e_\phi(w)$ gives a contradiction. Furthermore, again by Theorem \ref{class1}, $\phi^{-1}(\alpha)$ must have a unique element with odd ramification index.  Thus, there is $x_{\alpha} \in\phi^{-1}(\alpha)$ with odd ramification index such that every member of $\phi^{-1}(\alpha)\setminus \{x_\alpha\}$ has ramification index divisible by $4$. 

Suppose now that $\alpha_1, \alpha_2,$ and $\alpha_3$ are distinct 4-branch abundant points for $\phi$.
By Lemma \ref{lem:count'}, we have
\[\sum_{c\in \phi^{-1}(\{\alpha_1,\alpha_2,\alpha_3\})}(e_{\phi}(c)-1)\geq 3((d-1)\left(3/4\right)) = \frac{9}{8}(2d-2)>2d-2,
\]
contradicting Riemann-Hurwitz.
Hence if $\alpha_1$ and $\alpha_2$ are distinct $4$-branch abundant for $\phi$, they are the only such points. 
Now $x_{\alpha_1}$ and  $x_{\alpha_2}$ are also $4$-branch abundant, and hence $\{x_{\alpha_1}, x_{\alpha_2}\} = \{\alpha_1,\alpha_2\}$. 
Therefore $\phi^{-1}( \{\alpha_1,\alpha_2\}) \setminus  \{\alpha_1,\alpha_2\}$ is empty, as desired. 
\end{proof}

\begin{lemma} \label{atmost3}
Let $\phi \in \C(x)$ have even degree $d \geq 2$. If $\phi$ has two 4-branch abundant points in $\PP^1(\C)$,  then $\phi$ has at most three 2-branch abundant points in $\PP^1(\C)$. 
\end{lemma}

\begin{proof}
Let $A = \{\alpha_1, \alpha_2\} \subset \PP^1(\C)$ be a set of two 4-branch abundant points for $\phi$, and suppose that $V \subset \PP^1(\C)$ is a set of four 2-branch abundant points for $\phi$. Because $\alpha_1$ and $\alpha_2$ are 2-branch abundant, we have $A \subseteq V$, and we take $V = \{\alpha_1, \alpha_2, v_1, v_2\}$.  By Lemma \ref{atmost4}, $V$ is the complete set of 2-branch abundant points, and $e_\phi(z) \in \{1, 2\}$ for each $z \in \phi^{-1}(V)$. It then follows from Lemma \ref{primelemma} that every element of $\phi^{-1}(A)$ is 2-branch abundant for $\phi$, and so 
$\phi^{-1}(A) \subseteq V$. Hence
\begin{equation} \label{W1}
2d = \sum_{z \in \phi^{-1}(A)} e_\phi(z) \leq \sum_{z \in V} e_\phi(z) \leq 8,
\end{equation}
and it follows that $d \in \{2, 4\}$. If $d = 4$, then we have equality in \eqref{W1}, implying that $\phi^{-1}(A) = V$ and $e_\phi(z) = 2$ for all $z \in V$. Because $\phi^{-1}(A) = V$ we must have $\phi^{-1}(\{v_1, v_2\}) \cap V = \emptyset$, for otherwise applying $\phi$ gives the impossible $\{v_1, v_2\} \cap A \neq \emptyset$. If $e_\phi(u) = 2$ for all $u \in \phi^{-1}(\{v_1, v_2\})$, then $\# \phi^{-1}(\{v_1, v_2\}) = 2d/2 = 4$, and together with $e_\phi(z) = 2$ for $z \in V$ we have a contradiction to Riemann-Hurwitz (recall $d = 4$ here). Hence 
$e_\phi(u) = 1$ for some $u \in \phi^{-1}(\{v_1, v_2\})$. But then $u$ is 2-branch abundant and $u \not\in V$, contradicting Lemma \ref{atmost4}.

Finally, suppose $d = 2$. Let $U = \{z \in \phi^{-1}(\{v_1, v_2\}) : e_\phi(z) = 1\}$, and note that the set $\phi^{-1}(A) \cup U$ consists of 2-branch abundant points for $\phi$, and so is a subset of $V$, and hence has at most four elements. Let $r_1$ (resp. $r_2$) be the number of ramification points of $\phi$ in $\phi^{-1}(A)$ (resp. $\phi^{-1}(\{v_1, v_2\}$)), and note that  $\#\phi^{-1}(A) = 4 - r_1$ and $\#U = 4 - 2r_2$. Thus $4 - r_1 + 4 - 2r_2 \leq 4$, implying $4 \leq r_1 + 2r_2$. Because $d = 2$ and $\phi^{-1}(A) \cap \phi^{-1}(\{v_1, v_2\} = \emptyset$ (otherwise $A \cap \{v_1, v_2\} \neq \emptyset$), we have $r_1 + r_2 \leq 2$. It follows that $r_2 = 2$ and $r_1 = 0$. Now $r_1 = 0$ implies $e_\phi(z) = 1$ for all $z \in V$. In addition, $r_1 = 0$ and $\phi^{-1}(A) \subseteq V$ give $\phi^{-1}(A) = V$, and hence $v_1$ and $v_2$ are 4-branch abundant by Lemma \ref{primelemma}. Therefore $\phi^{-1}(\{v_1, v_2\})$ consists of $2$-branch abundant points, again by Lemma \ref{primelemma}. But $r_2 = 2$ and $e_\phi(z) = 1$ for all $z \in V$ imply $\phi^{-1}(\{v_1, v_2\}) \cap V = \emptyset$, contradicting Lemma \ref{atmost4}.
\end{proof}

\begin{theorem} \label{main4thm}
Let $\phi \in \C(x)$ have degree $d \geq 2$, and assume that $A = \{\alpha_1, \alpha_2\} \subset \PP^1(\C)$ is a set of distinct $4$-branch abundant points for $\phi$. Suppose that $\phi$ is not $4$-trivial with respect to $A$, and let $\mu$ be a M\"obius transformation exchanging $\alpha_1$ and $\alpha_2$.  Then for either $\phi$ or $\mu \circ \phi \circ \mu^{-1}$, the $4$-ramification structure for $O_\phi^-(\alpha_1)$ is one of the following, where points named with distinct letters within a given diagram are distinct:
\begin{center}
\begin{tabular}{| l c | l c |}
\hline13. && 14. &\\
&
\begin{tikzpicture}
	\Tree
		[.\node (alpha1) {$\alpha_1$};
			\edge [double]; [.\phantom{$\alpha$} ]
			[.$\alpha_1$ ]
			[.\node (alpha2) {$\alpha_2$};
				\edge [double]; [.\phantom{$\beta$} ]
				[.\node (beta) {$\beta$};
					\edge [double]; [.\phantom{$\beta$} ]
				]
			]
		]
	\node [below left=-.15cm of alpha1, xshift=-.1cm] {\scriptsize $4^*$};
	\node [below left=-.2cm of alpha2, xshift=.15cm] {\scriptsize $4^*$};
	\node [below right=-.1cm of alpha2, xshift=-.2cm] {\scriptsize$2$};
	\node [below=-.05cm of beta, xshift=-.15cm] {\scriptsize $2^*$};
\end{tikzpicture}
&&
\begin{tikzpicture}
	\Tree
		[.\node (alpha1) {$\alpha_1$};
			\edge [double]; [.\phantom{$\beta$} ]
			[.$\alpha_1$ ]
			[.\node (alpha2) {$\alpha_2$};
				\edge [double]; [.\phantom{$\beta$} ]
			]
			[.\node (beta) {$\beta$};
				\edge [double]; [.\phantom{$\beta$} 
					\edge [draw=none]; [.{} ]
				]
			]
		]
	\node [below left=-.25cm of alpha1, xshift=-.25cm] {\scriptsize $4^*$};
	\node [below right=-.2cm of alpha1, xshift=.3cm] {\scriptsize $2$};
	\node [below=-.05cm of alpha2, xshift=-.2cm] {\scriptsize $4^*$};
	\node [below=-.05cm of beta, xshift=-.2cm] {\scriptsize $2^*$};
\end{tikzpicture}
\\\hline
\end{tabular}
\end{center}
\end{theorem}

\begin{proof} Because $\phi$ is not $4$-trivial with respect to $A$, Theorem \ref{51} shows that $d$ is even. Let
\[
B = \{z \in \phi^{-1}(A) \setminus  A : 4 \nmid e_\phi(z)\},
\] 
and observe that by Lemma \ref{primelemma}, $B$ consists of $2$-branch abundant points for $\phi$. It follows from Lemma \ref{atmost3} that $\#B \leq 1$. If $B$ is empty, then $\phi$ is $4$-trivial with respect to $A$, which gives a contradiction. Hence $\#B = 1$, and we take $B = \{\beta\}$. Observe that $4 \nmid e_\phi(\beta)$ and $4 \mid e_\phi(z)$ for each $z \in \phi^{-1}(A) \setminus A$ with $z \neq \beta$. But also $\sum_{z\in \phi^{-1}(A)}e_\phi(z)=2d$ is divisible by 4 (since $d$ is even), whence we must have $\phi^{-1}(A) \cap A \neq \emptyset$, which implies
\begin{equation} \label{alphas}
A \cap \phi(A) \neq \emptyset.
\end{equation}

If $2\nmid e_\phi(\beta)$, then Lemma \ref{primelemma} gives that $\beta$ is 4-branch abundant for $\phi$, and so $W = \{\alpha_1, \alpha_2, \beta\}$ is a set of three 4-branch abundant points for $\phi$. Then $U = \#\{\phi^{-1}(W) \setminus W : 4 \nmid e_\phi(z)\}$ consists of $2$-branch abundant points for $\phi$, and Lemma \ref{atmost3} implies that $U$ is empty. Applying Lemma \ref{lem:count'} with $T = W$ gives the contradiction
\begin{equation*}
\sum_{z \in \phi^{-1}(W)} (e_\phi(z) - 1) \geq (3d - 3) \cdot (3/4) > 2d - 2. 
\end{equation*}

Therefore $e_\phi(\beta) \equiv 2 \bmod{4}$. 
Suppose now that there is $v \in \phi^{-1}(\beta)$ with $e_\phi(v)$ odd. Because $d$ is even, there must also be $v' \in \phi^{-1}(\beta)$ with $e_\phi(v')$ odd and $v' \neq v$. By Lemma \ref{primelemma}, $\beta, v$, and $v'$ are all 2-branch abundant, and so by Lemma \ref{atmost3} we have $\#\{\alpha_1, \alpha_2, \beta, v, v'\} = 3$. But $\beta \not\in \{v, v'\}$, for otherwise applying $\phi$ gives the impossible $A \ni \beta$. Hence $\{v, v'\} = A$, and thus $\{\beta\} = \phi(A)$, contradicting \eqref{alphas}.

Thus all elements of $\phi^{-1}(\beta)$ have even ramification index. Let $R = \{z \in \phi^{-1}(A) : 4 \mid e_\phi(z)\}$, and observe that 
$\phi^{-1}(A) \subseteq A \cup \{\beta\} \cup R$, with equality if and only if $A \subseteq \phi^{-1}(A)$. We claim that 
\begin{align} 
\#\phi^{-1}(A) \leq 3 + \frac{2d-2-e_\phi(\beta)}{4}, \qquad \text{with equality holding if and only if} \label{conditions1} \\
\text{$e_\phi(z) = 4$ for all $z \in R$, \quad $A \subseteq \phi^{-1}(A)$, \quad and $e_\phi(\alpha_1) = e_\phi(\alpha_2) = 1.$} \label{conditions2}
\end{align}
To see why, let $a = \#(A \cap \phi^{-1}(A))$, and write $e_\phi(z) = 4r_z$ for each $z \in R$. Then $\#\phi^{-1}(A) = a + 1 + \#R$. To compute $\#R$, observe that
\[
2d = \sum_{\phi^{-1}(A)} e_\phi(z) = \left(\sum_{z \in (A \cap \phi^{-1}(A))} e_\phi(z) \right) + e_\phi(\beta) + 4\left( \sum_{z \in R}(r_z - 1) + \#R \right),
\]
from which it follows that
\[
\#\phi^{-1}(A) = a + 1 + \frac{1}{4} \left(2d - e_\phi(\beta) - \sum_{z \in (A \cap \phi^{-1}(A))} e_\phi(z) \right) -  \sum_{z \in R}(r_z - 1) \leq a + 1 + \frac{1}{4} \left(2d - e_\phi(\beta) - a \right),
\]
with equality holding if and only if $e_\phi(z) = 1$ for all $z \in (A \cap \phi^{-1}(A))$ and $e_\phi(z) = 4$ for all $z \in R$. 
But $a \in \{1,2\}$, and from this one has $a + 1 + (2d - e_\phi(\beta) - a)/4 \leq 3 + (2d - 2 - e_\phi(\beta))/4$, with equality holding if and only if $a = 2$. This proves the statements in \eqref{conditions1} and \eqref{conditions2}. Recall that all elements of $\phi^{-1}(\beta)$ have even ramification index, and apply this together with \eqref{conditions1} to get 
\begin{align} \label{4est}
\sum_{z\in \phi^{-1}(A \cup \{\beta\})}(e_\phi(z)-1) & = 2d - \#(\phi^{-1}(A)) + \sum_{z\in \phi^{-1}(\beta)} (e_\phi(z)-1) \\
& \geq 2d- \left(3 + \frac{2d-2-e_\phi(\beta)}{4} \right)+\frac{d}{2} \nonumber \\
& = 2d - 3 + \frac{2 + e_\phi(\beta)}{4} \nonumber \\
& \geq 2d- 2, \nonumber
\end{align}
with equality holding if and only if the conditions in \eqref{conditions2} hold, and also $e_\phi(\beta)=2$ and $e_\phi(z) = 2$ for all $z\in \phi^{-1}(\beta)$. Because $d$ is even, we must have $A \subseteq \phi^{-1}(\alpha_1)$ or $A \subseteq \phi^{-1}(\alpha_2)$; replacing $\phi$ by $\mu \circ \phi \circ \mu^{-1}$ if necessary, we assume the former. If $d\equiv 2 \pmod{4}$, we obtain 4-ramification structure (13) for $O_\phi^-(\alpha_1)$, and if $d\equiv 0 \pmod{4}$ we obtain 4-ramification structure (14) for $O_\phi^-(\alpha_1)$.
\end{proof}

\section{Field of definition of $\phi$ and its components} \label{fielddef}

Many of our main results require showing that if $\phi$ is defined over a subfield $K$ of $\C$, then certain irreducible factors of the numerator and denominator of iterates of $\phi$ may also be defined over $K$. In view of potential future applications, and because it entails no additional work, we state the results of this section for arbitrary fields of characteristic zero. 

\begin{lemma} \label{anypolys}
Let $F$ be a field of characteristic zero and $\overline{F}$ an algebraic closure of $F$. Given $h \in \overline{F}[x]$ and $m \geq 2$, let $g \in \overline{F}[x]$ be the monic polynomial of maximal degree such that $h(x) = f(x)(g(x))^m$ for some $f \in \overline{F}[x]$. If $h$ has coefficients in $F$, then so do both $f$ and $g$.
\end{lemma}

\begin{remark}
The assumption that $F$ have characteristic zero is necessary, as illustrated by the case where $\ell$ is prime, $F = \mathbb{F}_\ell(t)$, $f(x) = x$, $g(x) = (x - \sqrt[\ell]{t})$, and $m = \ell$. 
\end{remark}

\begin{proof} 
Let 
\begin{align*} 
R_1 & = \{\text{roots of $f$ that are not roots of $g$}\}, \\ 
R_2 & = \{\text{roots of $g$ that are not roots of $f$}\}, \\ \textrm{Crit}(\phi)
R_ 3 &= \{\text{roots of both $f$ and $g$}\}. 
\end{align*}
These are pairwise disjoint subsets of $\overline{F}$. 
The maximality of the degree of $g$ implies that 
$e_h(\alpha) < m$ for each $\alpha \in R_1$, $m \mid e_h(\alpha)$ for each $\alpha \in R_2$, and each $\alpha \in R_3$ satisfies $m > e_h(\alpha)$ and 
$m \nmid e_h(\alpha)$. Because the set of roots of $h$ is $R_1 \cup R_2 \cup R_3$ and $h \in F[x]$, each $\sigma \in G_F := \Gal(\overline{F}/F)$ permutes $R_1 \cup R_2 \cup R_3$. We also have $e_h(\alpha) = e_h(\sigma(\alpha))$, and it follows that $\sigma(R_i) = R_i$ for $i = 1, 2, 3$. Now the set of roots of $f$ is $R_1 \cup R_3$, and the set of roots of $g$ is $R_2 \cup R_3$. Let $c_{f}$ be the leading coefficient of $f$, and observe that each of $f/c_f$ and $g$ are monic polynomials whose set of roots is preserved by the action of $G_F$. Because $F$ has characteristic zero, $\overline{F}/F$ is Galois, and thus the fixed field of $G_F$ is $F$, implying that $f/c_f$ and $g$ are both in $F[x]$. But $c_f$ is the leading coefficient of $h$, and thus is in $F$. Hence $f \in F[x]$. 
\end{proof}

We remark here that by definition a rational function $\phi$ is defined over $F$ (written $\phi \in F(x)$) if there are relatively prime $p, q \in F[x]$ with $\phi = p/q$. If $\phi \in F(x)$ and $f, g \in \overline{F}[x]$ with $\phi = f/g$ and $\gcd(f,g) = 1$, then we have $pg = fq$, whence $cf = p$ and $cg = q$ for some $c \in \overline{F}$. 
If $f$ and $g$ are monic then $c$ equals the leading coefficient of $p$ (or $q$), and hence $c \in F$, giving that $f, g \in F[x]$.

\begin{theorem} \label{fieldthm}
Let $F$ be a field of characteristic zero, $\overline{F}$ an algebraic closure of $F$, and $\phi \in \overline{F}(x)$. Let $\textrm{Crit}(\phi)$ be the set of all $\alpha \in \PP^1(\overline{F})$ with $e_\phi(\alpha) > 1$. For each $\alpha \in \textrm{Crit}(\phi)$, write $e_\phi(\alpha) = q_\alpha m + r_\alpha$, with $0 < r_\alpha < m$. Let $$\psi(x) = \prod_{\alpha \in \textrm{Crit}(\phi)} (x - \alpha)^{q_\alpha}.$$ If $\phi \in F(x)$, then $\psi(x)$ and $\phi(x)/(\psi(x))^m$ are both in $F(x)$. 
\end{theorem}

\begin{proof}
Assume $\phi \in F(x)$. Write $\psi = g_1/g_2$, where each $g_i \in \overline{F}[x]$ is monic and $\gcd(g_1, g_2) = 1$, and write $\phi(x)/(\psi(x))^m = f_1/f_2$, where each $f_i \in \overline{F}[x]$ and $\gcd(f_1, f_2) = 1$. Because $\phi \in F[x]$, there is $c \in \overline{F}$ with $cf_i(x)(g_i(x))^m \in F[x]$ for $i = 1, 2$. By Lemma \ref{anypolys} we have $g_i \in F[x]$ and $cf_i \in F[x]$. Hence $\psi \in F(x)$ and $\phi(x)/(\psi(x))^m \in F(x)$, the latter since 
$\phi(x)/(\psi(x))^m = (cf_1)/(cf_2)$.
\end{proof}

\section{Proof of Theorem \ref{maingenus}} \label{firstthmsec}

To prove Theorem \ref{maingenus}, 
we must relate the ramification structure of backward orbits of $m$-branch abundant points to global properties of $\phi$. When these ramification structures have certain properties, $\phi$ must descend from an endomorphism of an algebraic group -- either $\mathbb{G}_m$ or an elliptic curve. For this we cite some results from the invaluable paper of Milnor \cite{milnor}. We call $z \in \PP^1(\C)$ \textit{exceptional} for $\phi$ if the backwards orbit $\bigcup_{n=1}^{\infty} \phi^{-n}(z)$ is finite, and we denote by $\mathcal{E}_\phi$ the collection of all exceptional points for $\phi$. Recall that we denote the postcritical set of $\phi$ by \text{Postcrit$(\phi)$} (see the paragraph before Theorem \ref{maingenus} for the definition). A rational function $\phi \in \C(x)$ of degree at least two is a \textit{finite quotient of an affine map} if there is a flat surface $\C / \Lambda$ (where $\Lambda \subset \C$ is a lattice), an affine self-map of $\C/\Lambda$ given by $L(t) = at + b$, and a finite-to-one holomorphic map $\Theta : \C/\Lambda \to \PP^1(\C) \setminus \mathcal{E}_\phi$ satisfying $\phi \circ \Theta = \Theta \circ L$. As stated prior to Theorem \ref{maingenus}, we call $\phi$ a Latt\`es map when $\Lambda$ has rank two, and hence $\C / \Lambda$ is a torus. Milnor states a useful ramification-based characterization of Latt\`es maps, which we make heavy use of in the proof of Theorem \ref{maingenus}. 

\begin{theorem}[Milnor \cite{milnor}, Theorem 4.1] \label{milnorlattes}
\label{milnor}
Let $\phi \in \C(x)$ be a rational function and $\mathcal{E}_\phi$ its set of exceptional points. Then $\phi$ is a finite quotient of an affine map if and only if there exists an integer-valued function $r(z)$ on $\PP^1(\C) \setminus \mathcal{E}_\phi$ that satisfies $r(\phi(z))=e_\phi(z)r(z)$ and takes the value 1 outside of \text{Postcrit$(\phi)$}.   
\end{theorem}

We can extend $r$ to a function from $\PP^1(\C)$ to $\Z \cup \{\infty\}$ by taking $r(z)=\infty$ if $z\in \mathcal{E}_\phi$. When $\phi$ is a finite quotient of an affine map, its \textit{signature} is the sequence of values $r$ takes on \text{Postcrit$(\phi)$}. It is true, though not obvious, that the existence of the map $r$ as in Theorem \ref{milnor} implies the finiteness of the post-critical set of $\phi$ (see the proof of Theorem 4.1 in \cite{milnor}). In Theorem $4.5$ and Remark $4.7$ of \cite{milnor}, Milnor shows that there are only six possible signatures. They are $(2, 2, \infty)$ and $(\infty, \infty)$, which give maps conjugate to Chebyshev polynomials and power maps, respectively; and $(2,2,2,2)$, $(3,3,3)$, $(2,4,4)$, and $(2,3,6),$ which give Latt\`es maps. We summarize this as follows: 

\begin{theorem}[Milnor \cite{milnor}] \label{milnor2}
Let $\phi \in \C(x)$ be a rational function. Then $\phi$ is a Latt\`es map if and only if there exists a function $r: \PP^1(\C) \to \Z$ satisfying $r(\phi(z))=e_\phi(z)r(z)$ and taking the value 1 outside of \text{Postcrit$(\phi)$}. In this case, the signature of $\phi$ is one of $(2,2,2,2)$, $(3,3,3)$, $(2,4,4)$, or $(2,3,6).$
\end{theorem}

We immediately obtain a corollary that will be useful in the proof of Theorem \ref{maingenus}.

\begin{corollary} \label{lattescor}
Let $m \in \Z$ with $m \geq 2$, let $\phi \in \C(x)$ have degree $d \geq 2$, and assume that $A = \{\alpha_1, \alpha_2\} \subset \PP^1(\C)$ is a set of distinct $m$-branch abundant points for $\phi$. Suppose that $\phi$ is not $m$-trivial with respect to $A$, and let $\mu$ be a M\"obius transformation exchanging $\alpha_1$ and $\alpha_2$. Then $m \leq 4$. If $m = 4$ (resp. 3), then $\phi$ is a Latt\`es map of signature (2,4,4) (resp. (3,3,3)) with $r(\alpha_1) = r(\alpha_2) = m$. 
If $m = 2$, then unless $\phi$ or $\mu \circ \phi \circ \mu^{-1}$ satisfies (2D), (2F), (2M), or (2O) of Theorem \ref{2class}, $\phi$ is a Latt\`es map of signature (2,2,2,2) with $r(\alpha_1) = r(\alpha_2) = 2$. 
\end{corollary} 

\begin{proof}
It is well-known that the collection of Latt\`es maps is invariant under M\"obius conjugation;
hence for the present corollary it suffices to show that either $\phi$ or $\mu \circ \phi \circ \mu^{-1}$ is Latt\`es. 
Because $\phi$ is not $m$-trivial with respect to $A$, Theorem \ref{mge5} shows $m \leq 4$. If $m = 4$, then for either $\phi$ or $\mu \circ \phi \circ \mu^{-1}$, $O_\phi^+(\alpha_1)$ has 4-ramification structure (13) or (14) in Theorem \ref{main4thm}, and we let $\beta$ be as in those 4-ramification structures. Observe that \eqref{4est} implies that $e_\phi(z) = 1$ for $z \not\in \phi^{-1}(\{\alpha_1, \alpha_2, \beta\}$, and so taking $r(\alpha_1) = r(\alpha_2) = 4$ and $r(\beta) = 2$ and applying Theorem \ref{milnor2} shows that $\phi$ is Latt\`es of signature (2,4,4). If $m = 3$, then it follows from Theorem \ref{3class} that we may take $r(z) = 3$ for $z \in Ab(\alpha_1, \alpha_2)$ and $r(z) = 1$ otherwise and apply Theorem \ref{milnor2} to show that $\phi$ is Latt\`es of signature (3,3,3). If $m = 2$ and neither $\phi$ nor $\mu \circ \phi \circ \mu^{-1}$ satisfies (2D), (2F), (2M), or (2O) of Theorem \ref{2class}, then it follows from Theorem \ref{3class} that we may take $r(z) = 2$ for $z \in Ab(\alpha_1, \alpha_2)$ and $r(z) = 1$ otherwise and apply Theorem \ref{milnor2} to show that $\phi$ is Latt\`es of signature (2,2,2,2).
\end{proof}

\begin{proof}[Proof of Theorem \ref{maingenus}]
Fix $m \geq 2$, let $K$ be a subfield of $\C$, let $\phi \in K(x)$ have degree $d \geq 2$, and let $g_n$ be defined as in the discussion before Theorem \ref{iterative relationship}. 

Suppose that $g_n$ is bounded as $n \to \infty$. By Corollary \ref{rhocor} we have that $0$ and $\infty$ are $m$-branch abundant points for $\phi$. If $\phi$ is $m$-trivial with respect to $\{0, \infty\}$, then Proposition \ref{triv} shows that $\phi(x) = cx^j(\psi(x))^m$ with $\psi \in \C(x)$, $0 \leq j \leq m-1$, and $c \in \C^*$. We may apply Theorem \ref{fieldthm} to conclude that $\psi \in K(x)$ and $c \in K^*$.

Assume that $\phi$ is not $m$-trivial with respect to $\{0, \infty\}$. We apply Corollary \ref{lattescor} with $\mu(x) = 1/x$. If $m = 4$ (resp. $m = 3$), then Corollary \ref{lattescor} shows that we are in case (2) (resp. (3)) of the present theorem. If $m = 2$ and neither $\phi$ nor $\mu \circ \phi \circ \mu^{-1}$ satisfies (2D), (2F), (2M), or (2O) of Theorem \ref{2class}, then we are in case (4) of the present theorem. 

If $m = 2$ and one of $\phi$ or $\mu \circ \phi \circ \mu^{-1}$ satisfies (2D) in Theorem \ref{2class}, then we take $\alpha_1 = 0$ and $\alpha_2 = \infty$, giving a $3$-cycle $C \mapsto \infty \mapsto 0 \mapsto C$ ($C \in \C^*$) in 2-ramification structure (3) from Theorem \ref{class1}. 
Observe that $\phi(x) = B \prod_{r \in R}(x-r) \prod_{p \in P} (x-p)^{-1}$, where $B \in \C^*$ and $R$ (resp. $P$) is the set of roots (resp. poles) of $\phi$, with multiplicity. From 2-ramification structure (3) we have that all roots of $\phi$ except $\infty$, and all poles of $\phi$ except $C$, occur to even multiplicity. Hence
\begin{equation} \label{firstfcn}
\phi(x) = B\frac{f(x)^2}{(x-C)g(x)^2},
\end{equation}
for $f, g \in \C[x]$ monic with $\deg g \geq \deg f$ and $\gcd(f(x), (x-C)g(x)) = 1$. From Theorem \ref{fieldthm} we have $B(x - C) \in K[x]$ and $f/g \in K(x)$, and hence $B, C \in K^*$ and by the remark before Theorem \ref{fieldthm} we have $f, g \in K[x]$. 

Subtracting $C$ from both sides of \eqref{firstfcn} and doing some algebra yields 
\[
\phi(x) - C = \frac{B}{(x-C)g(x)^2} (f(x)^2 -  (C/B)(x-C)g(x)^2).
\]
Because $0$ is the only preimage of $C$ under $\phi$ with odd ramification index, we must have
\begin{equation} \label{conic2}
f(x)^2 -  (C/B)(x-C)g(x)^2 = bxh(x)^2, \quad b \in \C^*.
\end{equation}  
Because the left-hand side of \eqref{conic2} is in $K[x]$, so also must be $bxh(x)^2$. By Theorem \ref{fieldthm} we have $b \in K$ and $h \in K[x]$. Putting $x = 0$ in \eqref{conic2} gives $-B \in K^2$ (note that $f(0), g(0) \neq 0$ since $\phi(0) \not\in \{0, \infty\}$), and putting $x = C$ then gives $b \in CK^2$ ($f(C) \neq 0$ by assumption, whence $h(C) \neq 0$). Letting $D, E \in K$ satisfy $D^2 = -B$ and $b = CE^2$, we take $f_1(x) = Df(x) \in K[x]$ and $h_1(x) = DEh(x)$ to obtain 
$\phi(x) = -\frac{f_1(x)^2}{(x-C)g(x)^2}$ with $f_1(x)^2/D^2 - (C/B)(x-C)g(x)^2 = bxh_1(x)^2/(DE)^2$, i.e., $f_1(x)^2 +  C(x-C)g(x)^2 = Cxh_1(x)^2$. Writing $f$ for $f_1$ and $h$ for $h_1$, we have obtained the form (5a) in Theorem \ref{maingenus}. Note that $Cxh(x)^2$ has odd degree, and so we must have $\deg g \geq \deg f$, and hence we do not need to make this stipulation separately.

If $m= 2$ and one of $\phi$ or $\mu \circ \phi \circ \mu^{-1}$ satisfies (2F) in Theorem \ref{2class}, then we take $\alpha_1 = 0$ and $\alpha_2 = \infty$, giving a $2$-cycle $C \mapsto 0 \mapsto C$ ($C \in \C^*$) in 2-ramification structure (2) from Theorem \ref{class1}. If one of $\phi$ or $\mu \circ \phi \circ \mu^{-1}$ satisfies (2M) in Theorem \ref{2class}, then we take $\alpha_1 = 0$ and $\alpha_2 = \infty$, so that $\infty$ and $C$ are the preimages of $0$ having odd multiplicity. If one of $\phi$ or $\mu \circ \phi \circ \mu^{-1}$ satisfies (2O) in Theorem \ref{2class}, then we take $\alpha_1 = 0$ and $\alpha_2 = \infty$, so that $0$ is a fixed point in 2-ramification structure (7) of Theorem \ref{class2} with unique non-zero preimage $C$ of odd multiplicity. In each case, we argue as in case (2D) above to show that $\phi$ has form (5b), (5c), or (5d), respectively, and that \eqref{kpart2} holds. We leave the details to the reader. 

We now prove the `only if' part of the theorem. Suppose that $\phi$ satisfies one of conditions (1)-(5). We show that $0$ and $\infty$ are $m$-branch abundant for $\phi$, which by Corollary \ref{rhocor} shows that $g_n$ is bounded as $n \to \infty$. If $\phi$ satisfies condition (1), the desired conclusion follows from Proposition \ref{triv}. If $\phi$ satisfies conditions (2)-(4), then $\phi$ is Latt\`es with $r(0) = r(\infty) = m$, where $r$ is the function in Theorem \ref{milnor2}. It follows from the definition of $r$ that for each $n \geq 1$ we have 
$$\text{$e_\phi(z) = m$ for all $z \in \phi^{-n}(0) \setminus \text{Postcrit}(\phi)$}.$$
By Theorem \ref{milnor2}, $\text{Postcrit}(\phi)$ has at most four elements, and thus in the notation of Definition \ref{rhodef}, we have that $\rho_n(0) \leq 4$ for all $n$. Hence $0$ is $m$-branch abundant for $\phi$, and an identical argument shows the same conclusion for $\infty$.  If $\phi$ satisfies one conditions (5a)-(5d), then by construction $0$ and $\infty$ are $2$-branch abundant points for $\phi$.  
\end{proof}

We now discuss the parameterizations of maps in cases (5a) - (5d) in Theorem \ref{maingenus} mentioned in the introduction.  We begin with a detailed analysis of the case (5a). 
\begin{proposition} \label{paramprop}
Let $f,g \in \C[x]$ and $C \in \C \setminus \{0\}$. The following are equivalent:
\begin{enumerate}
\item $\gcd(f,g) = 1$, $f(C) \neq 0$, and $f(x)^2 + C(x-C)g(x)^2 = Cxh(x)^2$ for some $h(x) \in \C[x]$.
\item There exist $P, Q \in \C[x]$ satisfying $\gcd(P,Q) = 1,$ $Q(C) \neq 0$, $CP(0) \neq Q(0)$, 
\begin{align*} f(x) & = C^2P(x)^2(x-C) - 2CP(x)Q(x)(x-C) - CQ(x)^2, \; \text{and} \\
g(x) & = -CP(x)^2(x-C) - 2CP(x)Q(x) + Q(x)^2.
\end{align*}
\end{enumerate} 
\end{proposition}

\begin{proof}
Given $P$ and $Q$ as in (2), one checks that $f(x)^2 + C(x-C)g(x)^2 = Cxh(x)^2$ for $h(x) = Q(x)^2 + CP(x)^2(x-C)$. The assumption that $Q(C) \neq 0$ implies that $f(C) \neq 0$. Moreover, $f(x) + Cg(x) = -2CxP(x)Q(x)$, and so if $r \in \C$ is a common root of $f$ and $g$, then $r = 0$, $P(r) = 0$, or $Q(r) = 0$. Observe that $f(0) = -C(CP(0) - Q(0))^2$, and the assumption that $CP(0) \neq Q(0)$ forces $f(0) \neq 0$. Hence $r \neq 0$. If $P(r) = 0$, then $0 = f(r) = -CQ(r)^2$, contradicting $\gcd(P,Q) = 1$. If $Q(r) = 0$, then $r \neq C$ by assumption, and so $0 = f(r) = C^2P(r)(r-C)$ also contradicts $\gcd(P,Q) = 1$. It follows that $\gcd(f,g) = 1$. 

Given $f,g$ as in (1), observe that $f(x)^2 + C(x-C)g(x)^2 = Cxh(x)^2$ is equivalent to $(f(x)/h(x))^2 + C(x-C)(g(x)/h(x))^2 = Cx$. We thus look for solutions $\alpha, \beta \in \C(x)$ to the equation $\alpha(x)^2 + C(x-C)\beta(x)^2 = Cx$. Clearly $\alpha(x) = C$ and $\beta(x) = 1$ is one such solution, and because our equation is a conic, we use projection to find all other solutions. Letting $s$ and $t$ be variables and $\gamma$ an undetermined constant in $\C(x)$, the line $s = \gamma t + (1-C\gamma)$ passes through $(C,1)$. Substituting $\beta = \gamma \alpha + (1-C\gamma)$ into our conic and dividing through by $\alpha - C$ gives the solution 
$$
\alpha = \frac{C^2\gamma^2(x-C) - 2C\gamma(x-C) - C}{1 + C\gamma^2(x-C)}
$$
We then use $\beta = \gamma \alpha + (1-C\gamma)$ to obtain $\beta = \frac{-C\gamma^2(x-C) - 2C\gamma + 1}{1 + C\gamma^2(x-C)}$. Writing $\gamma(x) = P(x)/Q(x)$ with $\gcd(P,Q) = 1$ and clearing denominators gives the expressions for $f$ and $g$ in part (2) of the Proposition. Because $f(C) \neq 0$, we must have $Q(C) \neq 0$. We must also have $CP(0) \neq Q(0)$, for otherwise $f(0) = g(0) = 0$, contradicting $\gcd(f,g) = 1$. 
\end{proof}

For maps of the form (5b), a similar analysis gives the parameterization 
\begin{align*} f(x) & = CP(x)^2 - 2CP(x)Q(x) - (x-C)Q(x)^2, \\
g(x) & = -CP(x)^2 - 2(x-C)P(x)Q(x) + (x-C)Q(x)^2.
\end{align*}
Taking $C = -4$, $Q(x) = 2$, and $P(x) = 1$ gives $f(x) = -4(x+1)$ and $g(x) = 4$, leading to $\phi(x) = -(x+4)(x+1)^2$, which is $-(T_3(x+2)) + 2$, one of the maps mentioned in the paragraph following Theorem \ref{maingenus}. 

As for maps of the form (5c), note that $B(x-C)f(x)^2 - Cg(x)^2 = -Ch(x)^2$ is equivalent to $(B/C)(x-C)f(x)^2 = (g(x) + h(x))(g(x) - h(x))$. From $\gcd(f,g) = 1$ it follows that $\gcd(g,h) = 1$, and so $\gcd(g+h, g-h) = 1$. Thus one of $g+h, g-h$ is a square in $\C[x]$ while the other is $(x-C)$ times a square, and the squares multiply to $f(x)$. It follows that 
\begin{align*}
g(x) & = a(x-C)P(x)^2 + bQ(x)^2 \\
f(x) & = P(x)Q(x) &
\end{align*}
for some $P, Q \in \C[x]$ with $ab = B/4C$ and $\gcd((x-C)P(x),Q(x)) = 1$. Clearly any such $P, Q$ give a solution to $B(x-C)f(x)^2 - Cg(x)^2 = -Ch(x)^2$. 

Maps of the form (5d) may be handled with a similar analysis, though there are two cases: when one of $g+h, g-h$ is a square in $\C[x]$ and the other is $x(x-C)$ times a square; and when one is $x$ times a square and the other is $(x-C)$ times a square.

\section{Proof of Theorem \ref{iterative relationship}} \label{pfmain}

As a stepping stone to proving Theorem \ref{iterative relationship}, we give a useful result on $m$-trivial maps. We require first some notation. 
For an integer $n \geq 1$, let $P_n$ be the (possibly empty) set of primes dividing $n$. Fix an integer $j \geq 1$, and let $n \geq 1$ satisfy $P_n \subseteq P_j$. Define $w_j(n)$ to be the smallest nonnegative exponent $\ell$ such that $n \mid j^\ell$ . More explicitly, if $n = p_1^{e_1}p_2^{e_2} \cdots p_k^{e_k}$, and $f_i = v_{p_i}(j)$ for $i = 1, \ldots, k$, where $v_{p_i}$ denotes the $p_i$-adic valuation, then $w_j(n) = \max_i \lceil (e_i/f_i) \rceil$.  For relatively prime integers $a, b \geq 1$, we denote the order of $a$ in $(\Z/b\Z)^*$ by $\ord(a \bmod{b})$. 

\begin{lemma} \label{trivmain}
Let $m \geq 2$, let $K$ be a subfield of $\C$, let $\phi \in K(x)$ have degree $d \geq 2$, and assume that $\phi$ is $m$-trivial with respect to $\{0, \infty\}$. Let $\phi(x) = cx^j(\psi_0(x))^m$ as in Proposition \ref{triv}, and let $g_n$ be as in the discussion preceding Theorem \ref{iterative relationship}. Then $g_n = 0$ for all $n \geq 1$ and there exist integers $r > s \geq 0$ such that 
\begin{equation} \label{rseqn}
\text{$\phi^r(x) = \phi^s(x)(\psi(x))^m$ for some $\psi \in K(x)$.}
\end{equation}
When $j = 0$, \eqref{rseqn} holds if and only if $s \geq 1$. When $j > 0$, let $t$ be the minimal positive integer with $c^t \in K^m$, and let $m'$ (resp. $t'$) be the maximal divisor of $m$ (resp. $t$) relatively prime to $j$. Then \eqref{rseqn} holds if and only if 
\begin{equation} \label{values}
s \geq w_j(m/m') \qquad \text{and} \qquad \begin{cases} t' \mid (r-s) & \text{if $j = 1$} \\ \lcm [\ord(j \bmod{m'}), \ord(j \bmod{t'(j-1)})] \mid (r-s) & \text{if $j > 1$} \end{cases} 
\end{equation}
In all cases there exists $r \leq m$ such that \eqref{rseqn} holds. 
\end{lemma}

\begin{proof}
Because $\phi \in K(x)$, we may apply Theorem \ref{fieldthm} to conclude that $\psi_0 \in K(x)$ and $c \in K^*$. We describe the image of $\phi^n$ in $K(x)^*/K(x)^{*m}$ for all $n \geq 1$.  
Because $\phi(x) \equiv c x^j \pmod{K(x)^{*m}}$, we have 
\begin{equation} \label{explicit}
\phi^n(x) \equiv c^{1 + j + \cdots + j^{n-1}} x^{j^n} \pmod{K(x)^{*m}}
\end{equation}
for $n \geq 1$. It follows immediately from Proposition \ref{genus} that $g_n = 0$ for all $n \geq 1$. Note that if $j = 0$ then \eqref{explicit} gives $\phi^r(x) \equiv \phi(x)$ for all $r \geq 1$, and we may take $r = 2 \leq m$. 

Assume for the rest of the proof that $j \geq 1$. We now show that \eqref{rseqn} holds if and only if \eqref{values} does. It follows from \eqref{explicit} that 
$\phi^r(x) \equiv \phi^s(x) \pmod{K(x)^{*m}}$ for $r > s \geq 0$ is equivalent to  $x^{j^r} \equiv x^{j^s} \pmod{K(x)^{*m}}$ and $c^{1 + j + \cdots + j^{r-1}} \in K^{*m}$ (if $s = 0$) or
$c^{1 + j + \cdots + j^{r-1}} \equiv c^{1 + j + \cdots + j^{s-1}} \pmod{K^{*m}}$ (if $s \geq 1$). This in turn is equivalent to: 
\begin{align}
& j^r   \equiv j^s   \pmod{m} \label{one} \quad \text{and} \\
& j^s + \cdots + j^{r-1}   \equiv 0  \pmod{t} \label{too} 
\end{align}
Now \eqref{one} holds if and only if $m \mid j^s(j^{r-s} - 1)$. Observe that $\gcd(m', j) = 1$ implies that $\gcd(m', j^s) = 1$, and hence $m' \mid (j^{r-s} - 1)$. This holds if and only if $\ord(j \bmod{m'}) \mid r-s$. Moreover, every prime dividing $m/m'$ also divides $j$, and so we have that $m/m'$ and $(j^{r-s} - 1)$ are relatively prime, whence $(m/m') \mid j^s$. By the definition of $w_j$, this holds if and only if $s \geq w_j(m/m')$. Similarly, \eqref{too} holds if and only if $t \mid j^s(1 +  \cdots + j^{r-s-1})$, and as above this is equivalent to $t' \mid (1 +  \cdots + j^{r-s-1})$ and $(t/t') \mid j^s$. The former is equivalent to $t' \mid r-s$ (if $j = 1$) and $t'(j-1) \mid (j^{r-s} - 1)$, i.e., $\ord(j \bmod{t'(j-1)}) \mid r-s$ (if $j > 1$). Note that the minimality of $t$ implies that $t \mid m$, and so $t' \mid m'$ and $(t/t') \mid (m/m')$. Hence $(t/t') \mid j^s$ is implied by $(m/m') \mid j^s$.  

It remains to show that there exists $r \leq m$ such that \eqref{rseqn} holds. From \eqref{values}, taking 
\[s = w_j(m/m') \qquad \text{and} \qquad r = \begin{cases} s + t' & \text{if $j = 1$} \\ s+ \lcm [\ord(j \bmod{m'}), \ord(j \bmod{t'(j-1)})] & \text{if $j > 1$} \end{cases} 
 \]
satisfies \eqref{rseqn}, and so it is enough to show that 
$$w_j(m/m') + \lcm [\ord(j \bmod{m'}), \ord(j \bmod{t'(j-1)})] \leq m \qquad \text{and} \qquad w_j(m/m') + t' \leq m.
$$ Because $t' \mid m'$, we have that both $\ord(j \bmod{t'(j-1)})$ and $\ord(j \bmod{m'})$ divide $\ord(j \bmod{m'(j-1)})$. But $j$ belongs to the subgroup
$$
\{ g \in (\Z/m'(j-1)\Z)^* : g \equiv 1 \pmod{(j-1)}\},
$$ 
which has at most $m'$ elements, whence $\ord(j \bmod{m'(j-1)}) \leq m'$, and so 
\[
\lcm [\ord(j \bmod{m'}), \ord(j \bmod{t'(j-1)})] \leq m'.
\]
Hence it suffices in both the $j > 1$ and $j = 1$ cases to show that $w_j(m/m') + m' \leq m$. If $m = m'$, then $w_j(m/m') = 0$, and we are done. If $m \neq m'$, then write $m/m' = p_1^{e_1} \cdots p_k^{e_k}$, with $k \geq 1$. Let $e_\ell = \max_i e_i$, write $e = e_\ell$ and $p = p_\ell$, and note that $m' \leq m/p^e$. Hence we must show $e + (m/p^{e}) \leq m$. But $p^{e} \mid m$ and $e \geq 1$, and so $1 + e \leq p^{e} \leq m$. Using $e/(p^{e}-1) \leq 1$ gives $e/(p^{e}-1) + e \leq m$, and dividing by $e$ and combining terms gives $p^{e}/(p^{e}- 1) \leq m/e$. Taking reciprocals gives $1 - (1/p^e) \geq e/m$, which gives $m \geq e + (m/p^e)$, as desired. 
\end{proof}

Before proving Theorem \ref{iterative relationship}, we give one more preliminary result that will aid in our analysis. 

\begin{lemma} \label{multiplier}
Let $K$ be a subfield of $\C$, let $\phi \in K(x)$ be a Latt\`es map satisfying $\phi(\infty) = \infty$, and write $\phi(x) = Mf(x)/g(x)$ with $f, g \in K(x)$ monic. If $\phi$ has signature (2,4,4) and $r(\infty) = 4$, where $r$ is the function in Theorem \ref{milnor2}, then $M^2 \in K^4$. If $\phi$ has signature (3,3,3) and $r(\infty) = 3$, then $M \in K^3$. 
\end{lemma}

\begin{proof} 
By definition, $r(\phi(z)) = e_\phi(z)r(z)$ for all $z \in \PP^1(\C)$. By assumption, $\phi(\infty) = \infty$ and $r(\infty) \neq 0$, and hence we must have $e_\phi(\infty) = 1$. Therefore $\deg f = 1 + \deg g$.
Now the multiplier $\lambda_\infty(\phi)$ of $\phi$ at $\infty$ is defined to be $(1/\phi(1/x))'$ evaluated at $x = 0$ (see \cite[Exercise 1.13]{jhsdynam}). Because $\deg f = 1 + \deg g$, one easily deduces that $\lambda_\infty(\phi) = 1/M$. Put $n_\phi = 3$ if $\phi$ has signature (3,3,3), and $n_\phi = 4$ if $\phi$ has signature (2,4,4), and recall from Theorems \ref{milnorlattes} and \ref{milnor2} that $\phi$ is a finite quotient of a linear map $L : \C / \Lambda \to \C / \Lambda$, where $\Lambda \subset \C$ is a lattice. 
By \cite[Corollary 3.9]{milnor}, the multiplier at any fixed point $z_0$ of $\phi$ has the form $(\omega a)^{r(z_0)}$, where $\omega^{n_\phi} = 1$. Hence $M$ is of the form 
$a^{n_\phi}$. From \cite[Theorem 5.1]{milnor}, we have $a\Lambda \subset \Lambda$ and $\zeta_{n_\phi} \Lambda = \Lambda$, where $\zeta_{n_\phi}$ is a primitive $(n_\phi)th$ root of unity. It follows that $a \in \Lambda$ and $\Lambda = \Z[i]$ if $n_\phi = 4$ and $\Lambda = \Z[e^{2\pi i/3}]$ if $n_\phi = 3$. Therefore $[K(a) : K] \leq 2$, and hence $[K(M^{1/n_\phi}) : K] \leq 2$. Because $n_\phi \geq 3$, this implies $x^{n_\phi} - M$ is reducible over $K$, and by a well-known theorem (e.g. \cite[Theorem 8.1.6]{karpilovsky}), it follows that either $n_\phi = 3$ and $M \in K^3$ or $n_\phi = 4$ and one of $M \in K^2$ or $M \in -4K^4$ holds. In either of the cases for $n_\phi = 4$ we have $M^2 \in K^4$, which proves the lemma.
\end{proof}

\begin{proof}[Proof of Theorem \ref{iterative relationship}]
Fix $m \geq 2$, let $K$ be a subfield of $\C$, let $\phi \in K(x)$ have degree $d \geq 2$, and let $g_n$ be defined as in the discussion before Theorem \ref{iterative relationship}.  
By Corollary \ref{rhocor} it suffices to show that $0$ and $\infty$ are $m$-branch abundant points for $\phi$ if and only if $\phi^r(x) = \phi^s(x)$ in $K(x)^*/K(x)^{*m}$ for $r > s \geq 0$ with $r \leq m$ if $m \geq 3$ and $r \leq 6$ if $m = 2$. One direction is easy: if there are $r$ and $s$ satisfying the requisite properties, then for all $n \geq r$ we have $\phi^n(x) = \phi^j(x)$ in $K(x)^*/K(x)^{*m}$ for some $j \in \{0, \ldots, r-1\}$, and hence all $z \in \phi^{-n}(0)$ with $m \nmid e_{\phi^n}(z)$ lie in the set $\bigcup_{j = 0}^{r-1} \phi^{-j}(0)$, which is independent of $n$. Hence $0$ is $m$-branch abundant for $\phi$. Observe that 
\begin{equation} \label{invariance}
\text{$\phi^r(x) = \phi^s(x)(\psi(x))^m$ implies $\phi_1^r(x) = \phi_1^s(x)(\psi_1(x))^m$,} 
\end{equation} 
where $\phi_1(x) = 1/\phi(1/x)$ and $\psi_1(x) = 1/\psi(1/x)$; note in particular that if $\psi \in K(x)$ then $\psi_1 \in K(x)$. Because $\phi_1^{-n}(0) = \phi^{-n}(\infty)$, we have that $\infty$ is also $m$-branch abundant for $\phi$. 

Assume henceforth that $0$ and $\infty$ are $m$-branch abundant; we will show there exist $r$ and $s$ as described in the previous paragraph.  From \eqref{invariance} and the remark following, it suffices to show that for all $\phi$, the desired conclusion holds for either $\phi$ or $\mu \circ \phi \circ \mu^{-1}$, where $\mu(x) = 1/x$. 

If $\phi$ is $m$-trivial with respect to $\{0, \infty\}$, then the desired conclusion follows from Lemma \ref{trivmain}. If $\phi$ is not $m$-trivial with respect to $\{0, \infty\}$, then $m \leq 4$ by Theorem \ref{mge5}, and $O_\phi^-(0)$ and $O_\phi^-(\infty)$ are described in one of Theorems \ref{2class}, \ref{3class}, or \ref{main4thm}, according to whether $m = 2, 3,$ or $4$. We consider each of these cases separately. 

\smallskip

\noindent \textbf{Case 1:  $m = 4$.}
If $O_\phi^-(\alpha_1)$ has $4$-ramification structure (13) for either $\phi(x)$ or $1/\phi(1/x))$, then we take $\alpha_1 = \infty$ and $\alpha_2 = 0$, and we let $\beta$ be the unique preimage of $0$ with ramification index 2. Then 
\begin{equation*} 
\phi(x) = M\frac{(x - \beta)^2f(x)^4}{xg(x)^4} 
\end{equation*}
where the numerator and denominator are relatively prime, $f$ and $g$ are monic, 
and $M \in \C^*$. As with the function in \eqref{firstfcn}, we use Theorem \ref{fieldthm} to conclude that $M \in K$ and $f, g,$ and $(x - \beta)^2$ are all in $K[x]$. Therefore $\phi(x) \equiv M(x - \beta)^2/x \pmod{K(x)^{*4}}$, and hence 
\begin{equation} \label{phieqn}
\phi^2(x) \equiv M(\phi(x) - \beta)^2/\phi(x) \pmod{K(x)^{*4}}.
\end{equation} 
Note that $(x - \beta)^2 \in K[x]$ implies $2\beta \in K$, and so $\beta \in K$.  Now, 
\begin{equation*} 
\phi(x) - \beta = \frac{1}{xg(x)^4}[M(x - \beta)^2f(x)^4 - \beta xg(x)^4].
\end{equation*}
Let $u(x) = M(x - \beta)^2f(x)^4 - \beta xg(x)^4$. The roots (with multiplicity) of $u$ are the preimages (with multiplicity) of $\beta$ under $\phi$, and hence $u(x) = bh(x)^2$, with $h \in \C[x]$ monic and $b \in \C \setminus \{0\}$.  Applying Theorem \ref{fieldthm} again, 
we have $b \in K$ and $h \in K[x]$. This shows that $\phi(x) - \beta \in \frac{b}{x}K(x)^{*2} = bxK(x)^{*2}$, and so $(\phi(x) - \beta)^2 \in  b^2x^2K(x)^{*4}$. Now $bh(\beta)^2 = u(\beta) = -\beta^2g(\beta)^4$, and because $g(\beta) \neq 0$ (otherwise $\phi(\beta) \neq 0$, contrary to supposition), we have $-b \in K^2$, and squaring gives $b^2 \in K^4$. Therefore $(\phi(x) - \beta)^2 \in x^2K(x)^{*4}$. Similarly, putting $x = 0$ in $u(x)$ yields $Mb \in K^2$, and hence $M^2 \in K^4$ (one could also use Lemma \ref{multiplier} to derive this latter fact). 
Returning to \eqref{phieqn} now gives
$$
\phi^2(x) \equiv M\frac{(\phi(x) - \beta)^2}{\phi(x)} \equiv x^2 \cdot \frac{x}{(x - \beta)^2} \equiv x^3(x-\beta)^2 \pmod{K(x)^{*4}}. 
$$
Thus, modulo $K(x)^{*4}$, we have $\phi^3(x) \equiv (\phi(x))^3 (\phi(x) - \beta)^2 \equiv M^3(x-\beta)^2/x \equiv \phi(x)$, where the last equivalence follows because $M^2 \in K^4$.  
Hence \eqref{itreleqn} holds with $r = 3$ and $s = 1$, and from Proposition \ref{genus} we have $g_n = 1$ for all $n \geq 1$. 

If $O_\phi^-(\alpha_1)$ has $4$-ramification structure (14) for either $\phi(x)$ or $1/\phi(1/x))$, then we take $\alpha_1 = \infty$ and $\alpha_2 = 0$, and we let $\beta$ be the unique preimage of $\infty$ with ramification index 2. Then 
$$\phi(x) = \frac{M f(x)^4}{(x(x-\beta)^2g(x)^4)} \qquad \text{and} \qquad \phi(x) - \beta = \frac{u(x)}{(x(x-\beta)^2g(x)^4)},$$ 
with 
\begin{equation} \label{ux}
u(x) := Mf(x)^4 -  \beta x(x-\beta)^2g(x)^4 = bh(x)^2.
\end{equation} Taking $x = \beta$ or $x = 0$ in \eqref{ux} yields $b/M \in K^2$, and thus $b^2M^2 \in K^4$, but no further information. However, by Corollary \ref{lattescor} we have that $\phi$ is Latt\`es of signature (2,4,4) with $r(\infty) = 4$, and so from Lemma \ref{multiplier} we get $M^2 \in K^4$, whence $b^2 \in K^4$. One now obtains $\phi^2(x) \equiv x^3(x-\beta)^2 \pmod{K(x)^{*4}}$ and $\phi^3(x) \equiv \phi(x) \pmod{K(x)^{*4}}$ using an argument virtually identical to the previous case. The same conclusions about $r$, $s$, and $g_n$ hold. 

\begin{remark} The same general template as in the $m = 4$ case is applied to further cases below, and we omit certain details. For example, Theorem \ref{fieldthm} is frequently applied in subsequent cases to show that relevant polynomials and constants are defined over $K$. Hence \textit{from now on we assume that $f$ and $g$ are monic relatively prime polynomials with coefficients in $K$, and that $b, b_1, b_2 \in K$ and $h, h_1, h_2 \in K[x]$ are monic}. 
\end{remark}

\noindent \textbf{Case 2:  $m = 3$.}
We invoke Theorem \ref{3class} with $\mu(x) = 1/x$. 

If either $\phi(x)$ or $1/\phi(1/x)$ satisfies (3A), then we take $\alpha_1 = \infty$ and $\alpha_2 = 0$, and let $\gamma$ be the unique preimage of $0$ with ramification index 1. Hence
$$\phi(x) = \gamma \frac{(x-\gamma)f(x)^3}{xg(x)^3}, \quad \phi(x) - \gamma = \frac{u(x)}{xg(x)^3},$$ where $u(x) := \gamma(x-\gamma)f(x)^3 -  \gamma xg(x)^3 = bh(x)^3$ and the initial $\gamma$ in $\phi$ is because $\phi(\infty) = \gamma$. Putting $x = 0$ in $u(x)$ gives $-b/\gamma^2 \in K^3$, and so $b \gamma \in K^3$, implying that $\phi(x) - \gamma \in \gamma^2 x^2K(x)^{*3}$. 
It is then straightforward to check that 
$
\phi^2(x) 
\equiv \gamma^2(x-\gamma)^2 \pmod{K(x)^{*3}},
$ and $\phi^3(x) \equiv x \pmod{K(x)^{*3}}$. Hence \eqref{itreleqn} holds with $r = 3$ and $s = 0$, and from Proposition \ref{genus} we have $g_n = 0$ for all $n \geq 1$.  

If either $\phi(x)$ or $1/\phi(1/x)$ satisfies (3B), then we take $\alpha_1 = 0$ and $\alpha_2 = \infty$, and let $\gamma$ be the unique preimage of $0$ with ramification index 1. 
Writing $\phi(x) = M(x-\beta)f(x)^3/g(x)^3$ and arguing as in the previous case, one obtains $\phi^2(x) \in xK(x)^{*3}$. 
Thus \eqref{itreleqn} holds with $r = 2$ and $s = 0$, and $g_n = 0$ for all $n \geq 1$.

If either $\phi(x)$ or $1/\phi(1/x)$ satisfies (3C), then we take $\alpha_1 = \infty$ and $\alpha_2 = 0$, and let $\beta$ be the unique element of $\phi^{-1}(\infty) \setminus \{0, \infty\}$ with ramification index 1. Then
$
\phi(x) = M f(x)^3/(x(x-\beta)g(x)^3)$ and $\phi(x) - \beta = u(x)/(x(x-\beta)g(x)^3)$
with $u(x) :=  Mf(x)^3 -  \beta x(x-\beta)g(x)^3 = bh(x)^3$. Putting $x = 0$ or $x = \beta$ gives $b^2M \in K^3$ but no further information. However, by Corollary \ref{lattescor} we have that $\phi$ is Latt\`es of signature (3,3,3) with $r(\infty) = 3$, and so from Lemma \ref{multiplier} we get $M \in K^3$, whence $b \in K^3$. One now easily calculates
$
\phi^2(x) \equiv x^2(x-\beta)^2 \equiv \phi(x) \pmod{K(x)^{*3}}$. Thus \eqref{itreleqn} holds with $r = 2$ and $s = 1$, and from Proposition \ref{genus} we have $g_n = 1$ for all $n \geq 1$.

\smallskip

\noindent \textbf{Case 3:  $m = 2$.}
We invoke Theorem \ref{2class} with $\mu(x) = 1/x$. 

If either $\phi(x)$ or $1/\phi(1/x)$ satisfies (2A), then we take $\alpha_1 = \infty$ and $\alpha_2 = 0$, and let $\gamma$ be the unique preimage of $0$ with multiplicity 1, and $\delta$ be the unique preimage of $\gamma$ with multiplicity 1. Thus 
$$
\phi(x) = \delta \frac{(x-\gamma)f(x)^2}{xg(x)^2}, \quad \phi(x) - \gamma = \frac{u_1(x)}{xg(x)^2}, \quad \phi(x) - \delta = \frac{u_2(x)}{xg(x)^2},
$$
where 
$$u_1(x) := \delta (x-\gamma)f(x)^2 - \gamma xg(x)^2 = b_1(x-\delta)h_1(x)^2 \quad \text{and} \quad u_2(x) := \delta (x-\gamma)f(x)^2 - \delta xg(x)^2 = b_2h_2(x)^2.$$ 
Taking $x = 0$ in $u_1(x)$ yields $b_1 \in \gamma K^2$, and taking $x = 0$ in $u_2(x)$ gives $b_2 \in -\delta \gamma K^2$. 
Then one calculates 
$\phi^2(x) \equiv 
 \gamma (x-\gamma) (x - \delta) \pmod{K(x)^{*2}}$, $\phi^3(x) \equiv -\gamma \delta (x-\delta) \pmod{K(x)^{*2}}$, and $\phi^4(x) \equiv x \pmod{K(x)^{*2}}$, so that \eqref{itreleqn} holds with $r = 4$ and $s = 0$, and $g_n = 0$ for all $n \geq 1$. 

If either $\phi(x)$ or $1/\phi(1/x)$ satisfies (2B), then we take $\alpha_1 = \infty$ and $\alpha_2 = 0$, and let $\gamma$ be the unique preimage of $\infty$ with multiplicity 1, and $\delta$ be the unique preimage of $0$ with multiplicity 1. Arguing as in the (2A) case gives $\phi(x) \equiv \delta (x-\gamma)(x-\delta) \pmod{K(x)^{*2}} , \phi^2(x) \equiv  \gamma \delta x \pmod{K(x)^{*2}},  \phi^3(x) \equiv \gamma (x-\gamma)(x-\delta) \pmod{K(x)^{*2}},$ and $\phi^4(x) \equiv x \pmod{K(x)^{*2}}$. Hence \eqref{itreleqn} holds with $r = 4$ and $s = 0$, and $g_n = 0$ for all $n \geq 1$. 

If either $\phi(x)$ or $1/\phi(1/x)$ satisfies (2C), then we take $\alpha_1 = 0$ and $\alpha_2 = \infty$, and let $\beta$ be the unique preimage of $0$ with odd multiplicity, and $\gamma$ be the unique preimage of $\beta$ with odd multiplicity. Thus
$$
\phi(x) = \frac{M(x-\beta)f(x)^2}{g(x)^2}, \quad \phi(x) - \beta = \frac{u_1(x)}{g(x)^2}, \quad \phi(x) - \gamma = \frac{u_2(x)}{g(x)^2},
$$
where 
$$u_1(x) := M(x-\beta)f(x)^2 - \beta g(x)^2 = b_1(x-\gamma)h_1(x)^2, \quad u_2(x) := M(x-\beta)f(x)^2 - \gamma g(x)^2 = b_2xh_2(x)^2.$$  Substituting $x = \gamma$ and $x = \beta$ into $u_1(x)$ 
gives $\beta M (\gamma - \beta) \in K^2$ and $b_1\beta M (\gamma - \beta) \in K^2$, respectively. Hence $b_1 \in K^2$. Similar reasoning using $u_2$ gives $b_2 \in K^2$.  
We now have the following equivalencies modulo $K(x)^{*2}$: 
$
\phi(x) \equiv M(x-\beta), \phi^2(x) \equiv M(x - \gamma), \phi^3(x) \equiv Mx, \phi^4(x) \equiv (x-\beta), \phi^5(x) \equiv (x - \gamma),$ and $\phi^6(x) \equiv x$, showing that \eqref{itreleqn} holds with $r = 6$ and $s = 0$, and $g_n = 0$ for all $n \geq 1$.

If either $\phi(x)$ or $1/\phi(1/x)$ satisfies (2D), then we take $\alpha_1 = 0$ and $\alpha_2 = \infty$ and let $C$ be the unique preimage of $\infty$ with odd multiplicity. Then \eqref{firstfcn} and \eqref{conic2} give $\phi^2(x) \equiv Cx(x-C) \pmod{K(x)^{*2}},$ and  $\phi^3(x) \equiv 
Cbx \equiv x \pmod{K(x)^{*2}}$, showing that \eqref{itreleqn} holds with $r = 3$ and $s = 0$, and $g_n = 0$ for all $n \geq 1$.  

If either $\phi(x)$ or $1/\phi(1/x)$ satisfies (2E), then we take $\alpha_1 = \infty$, $\alpha_2 = 0$, $\beta_1$ to be the unique preimage of $\infty$ with odd ramification index, and $\beta_2$ to be the unique preimage of $0$ with odd ramification index. 
This gives
$$
\phi(x) = \beta_1 \frac{(x-\beta_2)f(x)^2}{(x- \beta_1)g(x)^2}, \quad \phi(x) - \beta_1 = \frac{u_1(x)}{(x-\beta_1)g(x)^2}, \quad \phi(x) - \beta_2 = \frac{u_2(x)}{(x-\beta_1)g(x)^2},
$$
where 
\begin{align*}
u_1(x) & := \beta_1(x-\beta_2)f(x)^2 - \beta_1(x- \beta_1)g(x)^2 = b_1h_1(x)^2, \\ 
u_2(x) & := \beta_1(x-\beta_2)f(x)^2 - \beta_2 (x-\beta_1)g(x)^2 = b_2xh_2(x)^2.
\end{align*}
Taking $x = \beta_1$ in $u_1(x)$ and $u_2(x)$ gives $b_1\beta_1(\beta_1 - \beta_2) \in K^2$ and $b_2\beta_1(\beta_1 - \beta_2) \in K^2$, which together imply $b_1b_2 \in K^2$. It is now straightforward to check that $\phi^2(x) \equiv b_1b_2x \equiv x \pmod{K(x)^{*2}}$, whence \eqref{itreleqn} holds with $r = 2$ and $s = 0$, and $g_n = 0$ for all $n \geq 1$.

If either $\phi(x)$ or $1/\phi(1/x)$ satisfies (2F), then we take $\alpha_1 = \infty$ and $\alpha_2 = 0$, and we let $C$ be the unique preimage of $0$ with odd ramification index. Then
$
\phi(x) = B(x-C)f(x)^2/g(x)^2$ and $\phi(x) - C = u(x)/g(x)^2
$
with $B, C \in K^*$ and $u(x) := B(x-C)f(x)^2 -  Cg(x)^2 = bxh(x)^2$.  
Putting $x = 0$ in $u(x)$ gives $-B \in K^2$, and putting $x = C$ then gives $b \in K^2$. It easily follows that $\phi^2(x) \equiv x \pmod{K(x)^{*2}}$. Hence \eqref{itreleqn} holds with $r = 2$ and $s = 0$, and $g_n = 0$ for all $n \geq 1$.

If either $\phi(x)$ or $1/\phi(1/x)$ satisfies (2G), then we take $\alpha_1 = \infty$ and $\alpha_2 = 0$, and we let $\beta_1, \beta_2$ be the elements of $\phi^{-1}(\infty) \setminus \{0, \infty\}$ with ramification index 1. Then for $i = 1, 2$ we have 
$$
\phi(x) = M\frac{f(x)^2}{x(x-\beta_1)(x-\beta_2)g(x)^2}, \quad \phi(x) - \beta_i = \frac{u_i(x)}{x(x-\beta_1)(x-\beta_2)g(x)^2},
$$
where 
$$u_i(x) := Mf(x)^2 - \beta_i x(x-\beta_1)(x-\beta_2)g(x)^2 = b_ih_i(x)^2.$$ 
Putting $x = 0$ in $u_i(x)$ gives $Mb_i \in K^2$ for $i = 1, 2$ and multiplying yields $b_1b_2 \in K^2$. One now calculates $\phi^2(x) \equiv x(x-\beta_1)(x-\beta_2) \pmod{K(x)^{*2}}$ and $\phi^3(x) \equiv Mx(x-\beta_1)(x-\beta_2) \equiv \phi(x) \pmod{K(x)^{*2}}$. Hence \eqref{itreleqn} holds with $r = 3$ and $s = 1$, and $g_n = 1$ for all $n \geq 1$.

If either $\phi(x)$ or $1/\phi(1/x)$ satisfies (2H), then we take $\alpha_1 = \infty$ and $\alpha_2 = 0$, we let $\beta$ be the unique element of $\phi^{-1}(\infty) \setminus \{0\}$ with ramification index 1, and we let $\gamma$ be the unique element of $\phi^{-1}(0) \setminus \{\infty\}$ with ramification index 1. Then
$$
\phi(x) = C\frac{(x-\gamma)f(x)^2}{x(x-\beta)g(x)^2}, \quad \phi(x) - \gamma = \frac{u_1(x)}{x(x-\beta)g(x)^2}, \quad \phi(x) - \beta = \frac{u_2(x)}{x(x-\beta)g(x)^2},
$$
where 
\begin{align*}
u_1(x) & := C(x-\gamma)f(x)^2 - \gamma x(x-\beta)g(x)^2 = b_1h_1(x)^2 \\ u_2(x) & := C(x-\gamma)f(x)^2 - \beta x(x-\beta)g(x)^2 = b_2h_2(x)^2.
\end{align*}
Putting $x = 0$ in $u_1(x)$ gives $-C\gamma b_1 \in K^2$, putting $x = \gamma$ in $u_1(x)$ gives $(\beta-\gamma) b_1 \in K^2$ and putting $x = \beta$ in $u_1(x)$ gives $C(\beta - \gamma) b_1 \in K^2$. In particular, $C \in K^2$. Putting $x = 0$ in $u_2(x)$ yields $-C -\gamma b_2 \in  K^2$, and thus $b_1b_2 \in K^2$. Now we obtain $\phi^2(x) \equiv b_1b_2x(x-\beta)(x - \gamma) \equiv \phi(x) \pmod{K(x)^{*2}}$. Hence \eqref{itreleqn} holds with $r = 2$ and $s = 1$, and $g_n = 1$ for all $n \geq 1$.  

If either $\phi(x)$ or $1/\phi(1/x)$ satisfies (2I), then we take $\alpha_1 = \infty$ and $\alpha_2 = 0$, and we let $\beta = \phi(\infty)$, and $\gamma$ the unique element of $\phi^{-1}(\beta) \setminus \{\infty\}$ with ramification index 1. An argument similar to that of case (2H) shows that we have the following equivalences modulo $K(x)^{*2}$: 
$\phi^2(x) \equiv -\beta \gamma (x-\gamma), \phi^3(x) \equiv -\gamma x (x-\beta), \phi^4(x) \equiv (x-\gamma)$, and $\phi^5(x) \equiv \beta x (x-\beta) \equiv \phi(x)$. Thus \eqref{itreleqn} holds with $r = 5$ and $s = 1$, and $g_n = 0$ for all $n \geq 1$. 

If either $\phi(x)$ or $1/\phi(1/x)$ satisfies (2J), then we take $\alpha_1 = \infty$ and $\alpha_2 = 0$, and we let $\beta = \phi(0)$, and let $\gamma$ be the unique element of $\phi^{-1}(\infty) \setminus \{\beta\}$ with ramification index 1. Arguing as in case (2H), we obtain $\phi^2(x) \equiv x \pmod{K(x)^{*2}}$, and thus \eqref{itreleqn} holds with $r = 2$ and $s =  1$, and we have $g_n = 0$ for all $n \geq 1$. 

If either $\phi(x)$ or $1/\phi(1/x)$ satisfies (2K), then we take $\alpha_1 = \infty$ and $\alpha_2 = 0$, and we let $\gamma_1$ and $\gamma_2$ be the two preimages of $0$ with ramification index 1.
Thus
$$
\phi(x) = \frac{M (x-\gamma_1)(x-\gamma_2)f(x)^2}{(xg(x)^2)}, \quad \phi(x) - \gamma_1 = \frac{u_1(x)}{(xg(x)^2)}, \quad \phi(x) - \gamma_2 = \frac{u_2(x)}{(xg(x)^2)},
$$
where $u_i(x) := M(x-\gamma_1)(x-\gamma_2)f(x)^2 - \gamma_i xg(x)^2 = b_ih_i(x)^2$ for $i = 1, 2$. Putting $x = 0$ in $u_i(x)$ yields $M b_i \gamma_1 \gamma_2 \in K^2$, whence $b_1b_2 \in K^2$, and it follows that $(\phi(x) - \gamma_1)(\phi(x) - \gamma_2) \in K(x)^{*2}$. 
One now calculates $\phi^2(x) \equiv x(x-\gamma_1)(x-\gamma_2) \pmod{K(x)^{*2}}$ and $\phi^3(x) \equiv Mx(x-\gamma_1)(x-\gamma_2) \equiv \phi(x) \pmod{K(x)^{*2}}$. Hence \eqref{itreleqn} holds with $r = 3$ and $s = 1$, and $g_n = 1$ for all $n \geq 1$.

If either $\phi(x)$ or $1/\phi(1/x)$ satisfies (2L), then we take $\alpha_1 = \infty$ and $\alpha_2 = 0$, and we let $\beta = \phi(0)$ and $\gamma$ be the non-zero preimage of $\beta$ with ramification index 1. Writing $\phi(x) = M f(x)^2 / ((x-\beta)g(x)^2)$ and arguing as in case (2K), we obtain $\phi(x) \in M(x - \beta)K(x)^{*2}$, $\phi(x) - \beta \in Mx(x-\beta)(x-\gamma)K(x)^{*2}$, and 
$\phi(x) - \gamma \in M(x - \beta)K(x)^{*2}$. It follows that $\phi^2(x) \equiv x(x-\beta)(x-\gamma) \pmod{K(x)^{*2}}$, $\phi^3(x)  \equiv Mx(x-\beta)(x-\gamma) \pmod{K(x)^{*2}}$, and $\phi^4(x) \equiv \phi^2(x) \pmod{K(x)^{*2}}$. Thus \eqref{itreleqn} holds with $r = 4$ and $s = 2$, and we have $g_1 = 0$ and $g_n = 1$ for all $n \geq 2$. This is the only case where $g_n$ is non-constant.

If either $\phi(x)$ or $1/\phi(1/x)$ satisfies (2M), then we take $\alpha_1 = \infty$ and $\alpha_2 = 0$, and we let $C$ be the unique non-zero preimage of $\infty$ with odd ramification index.
Writing
$
\phi(x) = B(x-C)f(x)^2/g(x)^2
$ 
with $B \in K^*$, we obtain $\phi(x) - C \in -CK(x)^{*2}$, and hence $\phi^n(x) \equiv -BC \pmod{K(x)^{*2}}$ for all $n \geq 2$. Thus \eqref{itreleqn} holds with $r = 3$ and $s = 2$, and $g_n = 0$ for all $n \geq 1$.

If either $\phi(x)$ or $1/\phi(1/x)$ satisfies (2N), then we take $\alpha_1 = \infty$ and $\alpha_2 = 0$, and we let $\beta_1$ (resp. $\beta_2$) be the unique preimage of $\infty$ (resp. 0) with ramification index 1. We have 
$$
\phi(x) = M\frac{x(x-\beta_2)f(x)^2}{(x-\beta_1)g(x)^2}, \quad \phi(x) - \beta_1 = \frac{u_1(x)}{(x-\beta_1)g(x)^2}, \quad \phi(x) - \beta_2 = \frac{u_2(x)}{(x-\beta_1)g(x)^2},
$$
where $u_i(x) := Mx(x-\beta_2)f(x)^2 - \beta_i (x-\beta_1)g(x)^2 = b_ih_i(x)^2$ for $i = 1, 2$. Putting $x = 0$ in $u_1(x)$ yields $b_1 \in K^2$. Putting $x = \beta_2$ in $u_i(x)$ yields $\beta_i(\beta_1 - \beta_2)b_i \in K^2$, and putting $x = \beta_1$ in $u_i(x)$ yields $M \beta_1(\beta_1 - \beta_2)b_i \in K^2$.  The latter immediately implies $b_1b_2 \in K^2$, so $b_2 \in K^2$. 
Using $\beta_1(\beta_1 - \beta_2)b_1 \in K^2$ and $M \beta_1(\beta_1 - \beta_2)b_1 \in K^2$ implies $M \in K^2$, and it quickly follows that $\phi^2(x) \equiv x(x-\beta_1)(x-\beta_2) \equiv \phi(x) \pmod{K(x)^2}$. Hence \eqref{itreleqn} holds with $r = 2$ and $s = 1$, and $g_n = 1$ for all $n \geq 1$.

If either $\phi(x)$ or $1/\phi(1/x)$ satisfies (2O), then we take $\alpha_1 = \infty$ and $\alpha_2 = 0$, and we let $C$ be the unique non-zero preimage of $0$ with odd ramification index. Writing
$
\phi(x) = B x(x-C)f(x)^2/g(x)^2
$ 
with $B \in K^*$, we obtain $\phi(x) - C \in -CK(x)^{*2}$, whence $\phi^2(x) \equiv -Cx(x-C) \pmod{K(x)^{*2}}$
and $\phi^3(x) \equiv  \phi(x) \pmod{K(x)^{*2}}$. Thus \eqref{itreleqn} holds with $r = 3$ and $s = 1$, and $g_n = 0$ for all $n \geq 1$. 
\end{proof}

\section{Proofs of remaining results} \label{pfremain}
 
\begin{proof}[Proof of Corollary \ref{powercor}]
Let $K$ be a finitely generated field of characteristic zero, fix $m \geq 2$, let $\phi \in K(x)$ have degree at least two, and assume there exists $a \in \mathbb{P}^1(K)$ such that $O_\phi^+(a) \cap \mathbb{P}^1(K)^m$ is infinite. Let $C_n$ be the curve given by $\phi^n(x) = y^m$; in the notation of the discussion following Conjecture \ref{adml}, we then have $X = \PP^1$, $Y = \PP^1$, $\lambda(x) = x^m$, and $Z_n = C_n$. Then \eqref{infrat} implies that $C_n(K)$ is infinite for all $n \geq 1$, and it follows from Faltings' Theorem that $g_n \leq 1$ for all $n \geq 1$. Hence $\phi$ falls into one of the cases in Theorem \ref{maingenus} and satisfies \eqref{kpart1} and \eqref{kpart2}, and $\phi$ also satisfies \eqref{itreleqn} with $\psi \in K(x)$.
\end{proof}

\begin{proof}[Proof of Theorem \ref{main ML}]
Let $K$ be a finitely generated field of characteristic zero, let $\phi, \lambda \in K(x)$ have degree at least two, and suppose that $\lambda$ is M\"obius-conjugate (over $K$) to a power map.  
From equation \eqref{Mobinv} in the introduction, it suffices to show that the set $\{n \in \N : \phi^n(a) \in \mathbb{P}^1(K)^m\}$ satisfies the conclusions of the theorem for any $\phi \in K(x)$ and $a \in \mathbb{P}^1(K)$.
Take $a \in \PP^1(K)$, and note that the theorem holds when $O_{\phi}(a) \cap \mathbb{P}^1(K)^m$ is finite, by the discussion following the statement of Theorem \ref{main ML}. Suppose for the rest of the proof that $O_\phi^+(a) \cap \mathbb{P}^1(K)^m$ is infinite. We will show that $\{n \in \N : \phi^n(a) \in \mathbb{P}^1(K)^m\}$ is a union of at most three arithmetic progressions with modulus at most $m$ (or 6 if $m = 2$). As in the proof of Corollary \ref{powercor}, we use Faltings' Theorem to derive $g_n \leq 1$ for all $n \geq 1$, and hence Corollary \ref{rhocor} gives that $0$ and $\infty$ are $m$-branch abundant points for $\phi$. 

By Theorem \ref{iterative relationship}, there are $r > s \geq 0$ with $\phi^r(x) \equiv \phi^s(x) \pmod{K(x)^{*m}}$. The sequence
$(\phi^n(x))_{n \geq 0}$ in the group $K(x)^*/K(x)^{*m}$ therefore has this form:
\begin{equation*}
x, \phi(x), \ldots, \phi^s(x), \ldots \phi^{r-1}(x), \phi^s(x), \ldots \phi^{r-1}(x), \phi^s(x), \ldots
\end{equation*}
Observe that if $\psi \in K(x)$, $b \in K$, and $\psi(b) \not\in \{0, \infty\}$, then $\psi(b) \in K^m$ if and only if $\tilde{\psi}(b) \in K^m$ for every $\tilde{\psi} \in K(x)$ with $\psi(x) \equiv \tilde{\psi}(x) \pmod{K(x)^{*m}}$. Let $J = \{s \leq n \leq r-1 : \phi^n(a) \in K^m\}$. If $O_\phi^+(a) \cap \{0, \infty\} = \emptyset$, then it follows that $\{n \in \N: \phi^n(a) \in \PP^1(K)^m\} = I \cup F$ with 
\begin{equation} \label{if}
I = \bigcup_{j \in J} (j + (r-s)\mathbb{N}) \quad \text{and} \quad F = \{0 \leq n < s : \phi^{n}(a) \in K^m\}. 
\end{equation}

Suppose first that $\phi$ is $m$-trivial with respect to $\{0, \infty\}$, and write $\phi(x) = cx^j(\psi_0(x))^m$ as in Proposition \ref{triv}. Let $t$ be the minimal positive integer such that $c^t \in K^{m}$. If $j = 0$ and $t > 1$, then we have $c \not\in K^m$, and hence $\phi(b) \not\in K^m$ for all $b \in \mathbb{P}^1(K)$ with $\phi(b) \not\in \{0,\infty\}$. This forces
$O_\phi^+(a) \cap \mathbb{P}^1(K)^m$ to be finite, a contradiction. If $j = 0$ and $t = 1$, then $c \in K^m$, whence $\phi(b) \in K^m$ for all $b \in \mathbb{P}^1(K)$, implying that $\{n \in \N : \phi^n(a) \in \mathbb{P}^1(K)^{m}\} = \ell + M\N$ with $M=1$ and $\ell = 0$ (if $a \in K^m$) or $\ell = 1$ (otherwise). 

If $j > 0$, then because $0 < j < m$ and the order of any zero or pole of $\psi^m$ is divisible by $m$, we must have $\phi(0) \in \{0, \infty\}$ and $\phi(\infty) \in \{0, \infty\}$. The infinitude of $O_\phi^+(a)$ then implies that $O_\phi^+(a) \cap \{0, \infty\} = \emptyset$. We could now use \eqref{if} to show that $\{n \in \N: \phi^n(a) \in \PP^1(K)^m\}$ is a union of finitely many arithmetic progressions, but we wish to prove the stronger statement that it is a \textit{single} arithmetic progression. Let $\ell$ be the minimal non-negative integer with $\phi^{\ell}(a) \in \mathbb{P}^1(K)^{m}$, and because $\phi^{\ell}(a) \neq \infty$ we may write $\phi^{\ell}(a) = b^m$ for some $b \in K$. From \eqref{explicit} and the fact that $O_\phi^+(a) \cap \{0, \infty\} = \emptyset$, for all $u \geq 1$ we have 
$
\text{$\phi^{\ell + u}(a) \in K^m$ if and only if $c^{1 + j + \cdots + j^{u-1}} \in K^m$.}
$
This in turn is equivalent to
\begin{equation} \label{congruence}
1 + j + \cdots + j^{u-1} \equiv 0 \bmod{t}. 
\end{equation} If $\gcd(t,j) \neq 1$, then \eqref{congruence} cannot hold for any $u \geq 1$, giving the contradiction $O_\phi^+(a) \cap \mathbb{P}^1(K)^{m} = \{\phi^{\ell}(a)\}$. 
Therefore $\gcd(t,j) = 1$. 
If $j = 1$, then \eqref{congruence} holds if and only if $u$ is a multiple of $t$, and we have 
$\{n \in \N : \phi^n(a) \in \mathbb{P}^1(K)^{m}\} = \ell + t\N$. 
If $j > 1$, then note that $j$ is relatively prime to both $j-1$ and $t$, and let $M$ be the order of $j$ in $(\Z/t(j-1)\Z)^*$. Then \eqref{congruence} is equivalent to $j^{u}-1 \equiv 0 \bmod{t(j-1)}$,
which holds if and only if $u$ is a multiple of $M$. Hence $\{n \in \N : \phi^n(a) \in \mathbb{P}^1(K)^{m}\} = \ell + M\N$. 

Suppose that $\phi$ is not $m$-trivial with respect to $\{0, \infty\}$. Then $m \leq 4$ by Theorem \ref{mge5}, and so either $\phi(x)$ or $\phi_1(x) := 1/\phi(1/x)$ is described by one of Theorems \ref{3class}, \ref{2class}, or \ref{main4thm}. Note that if $O_\phi^+(0)$ and $O_\phi^+(\infty)$ are both finite, then $O_{\phi_1}(0)$ and $O_{\phi_1}(\infty)$ are also finite. Hence if either $\phi$ or $\phi_1$ satisfies any of the conclusions of Theorems \ref{3class}, \ref{2class}, or \ref{main4thm} \textit{except for (2M) and (2O) in Theorem \ref{2class}}, we have that $O_\phi^+(0)$ and $O_\phi^+(\infty)$ are both finite, whence the infinitude of $O_\phi^+(a)$ implies that $O_\phi^+(a) \cap \{0, \infty\} = \emptyset$. We may then use \eqref{if} to conclude that $\{n \in \N : \phi^n(a) \in \mathbb{P}^1(K)^{m}\}$ consists of at most $s$ arithmetic progressions of modulus 0 (i.e., singletons) plus at most $k$ infinite arithmetic progressions of modulus dividing $r-s$, where $k = (r-s)/2$ if $r-s$ is even and $k = s - 1$ if $r-s$ is odd. From the proof of Theorem \ref{iterative relationship} we have in each case that $s + k \leq 3$, $r-s \leq m$ for $m \geq 3$, and $r-s \leq 6$ for $m = 2$.  

Finally, suppose that either $\phi$ or $\phi_1$ satisfies (2M) or (2O) of Theorem \ref{2class}. Because $O_\phi^+(a)$ is infinite, each of $0$ and $\infty$ can appear at most once in the sequence $(\phi^n(a))_{n \geq 0}$. In case (2M) we have $\phi^n(x) \equiv \phi^2(x) \pmod K(x)^{*2}$ for all $n \geq 2$, and the infinitude of $O_\phi^+(a) \cap \PP^1(K)^m$ implies that
$\phi^n(a) \in \PP^1(K)^m$ for all $n \geq 2$. It follows that $\{n \in \N : \phi^n(a) \in \mathbb{P}^1(K)^m\}$ is a union of at most two arithmetic progressions. In case (2O), observe that precisely one of $0$, $\infty$ has infinite forward orbit, and hence at most one of them can appear in $O_\phi^+(a)$; without loss of generality, say this is $\infty$. From the last paragraph of the proof of Theorem \ref{iterative relationship}, we may take $r = 3$ and $s = 1$ in \eqref{itreleqn}, and the infinitude of $O_\phi^+(a) \cap \PP^1(K)^m$ implies that one of the following holds: $\phi^n(a) \in \PP^1(K)^m$ for all $n \geq 1$; $\phi^{2n}(a) \in \PP^1(K)^m$ for all $n \geq 1$ and $\phi^{2n-1}(a) \not\in \PP^1(K)^m$ for all $n \geq 1$ except at most one value of $n$ with $\phi^{2n-1}(a) = \infty$; or $\phi^{2n-1}(a) \in \PP^1(K)^m$ for all $n \geq 1$ and $\phi^{2n}(a) \not\in \PP^1(K)^m$ for all $n \geq 1$ except at most one value of $n$ with $\phi^{2n}(a) = \infty$. In each of these cases, $\{n \geq 1 : \phi^n(a) \in \mathbb{P}^1(K)^m\}$ is a union of at most two arithmetic progressions, and thus
$\{n \in \N : \phi^n(a) \in \mathbb{P}^1(K)^m\}$ is a union of at most three arithmetic progressions. 
\end{proof}

Finally, we prove Corollaries \ref{mainpoly} and \ref{quadpoly}. The following lemma aids in the proof of Corollary \ref{mainpoly}. 

\begin{lemma}[\cite{schinzelbook}, Lemma 6 on p. 26] \label{schinzel}
Let $F$ be a field of characteristic $\neq 2$, and suppose that
$$
(Q(x) - q_1)(Q(x) - q_2) = (x-\xi_1)(x-\xi_2)(R(x))^2,
$$
for $Q, R \in F[x]$, $q_1, q_2, \xi_1, \xi_2 \in F$, $q_1 \neq q_2$, $\xi_1 \neq \xi_2$. Then $Q = L \circ T_{\deg Q} \circ M^{-1}$, where
$$
L(x) = \frac{q_1 - q_2}{4}x + \frac{q_1 + q_2}{2} \quad \text{and} \quad M(x) = \frac{\xi_1 - \xi_2}{4}x + \frac{\xi_1 + \xi_2}{2}.
$$
\end{lemma}

\begin{proof}[Proof of Corollary \ref{mainpoly}]
Let $K$ be a finitely generated field of characteristic zero, fix $m \geq 2$, let $\phi \in K[x]$ have degree $d \geq 2$, and assume there exists $a \in \PP^1(K)$ such that $O_\phi^+(a) \cap \mathbb{P}^1(K)^m$ is infinite. As in the proof of Corollary \ref{powercor}, we use Faltings' Theorem to derive $g_n \leq 1$ for all $n \geq 1$; indeed, in this case we may use Siegel's theorem to show $g_n = 0$ for all $n \geq 1$, though we do not need this stronger conclusion. Corollary \ref{rhocor} then gives that $0$ and $\infty$ are $m$-branch abundant points for $\phi$. If $\phi$ is $m$-trivial with respect to $\{0, \infty\}$, then Proposition \ref{triv} and Lemma \ref{anypolys} imply condition (1) of the present corollary. 

Suppose that $\phi$ is not $m$-trivial with respect to $\{0, \infty\}$. Then $m \leq 4$ by Theorem \ref{mge5}, and so either $\phi(x)$ or $\phi_1(x) := 1/\phi(1/x)$ is described by one of Theorems \ref{3class}, \ref{2class}, or \ref{main4thm}. Minor modifications to the proof of Lemma \ref{atmost4} show that $\phi$ has at most one $3$-branch abundant point in $\C$, and at most two $2$-branch abundant points in $\C$. Indeed, in the proof of Lemma \ref{atmost4}, let $V = \{\alpha_1, \ldots, \alpha_k\} \subset \C$ be a set of $p$-branch abundant points for $\phi$. Because $\sum_{z \in \C} (e_\phi(z) - 1) = d-1$, the bound in \eqref{rameqn} implies that $k = 1$ if $p = 3$ and $k \leq 2$ if $p = 2$, as desired. Hence $\phi$ has at most three $2$-branch abundant points in $\PP^1(\C)$, and at most two $3$-branch abundant points in $\PP^1(\C)$, and in both cases one of these is a fixed point whose only preimage is itself. The same statements hold for $\phi_1(x)$. This rules out all cases of Theorems \ref{3class}, \ref{2class}, and \ref{main4thm}, except for (2F) (where $d$ is odd) and (2O) (where $d$ is even) in Theorem \ref{2class}. In both of those cases, let $\{\infty, 0, \beta\}$ be the $2$-branch abundant points for $\phi$, and note that $\infty$ must be the fixed point. It follows that the conditions of Lemma \ref{schinzel} are satisfied with $\{q_1, q_2\} = \{\xi_1, \xi_2\} = \{0, \beta\}$, and thus $L(x) = -(\beta/4)(\epsilon_L x - 2)$ and $M(x) =  -(\beta/4)(\epsilon_M x - 2)$, with $\epsilon_L, \epsilon_M \in \{1, -1\}$. Setting $c = -4/\beta$, we then have 
\begin{equation} \label{chebeqn}
c\phi(x/c) = \epsilon_L (T_d(\epsilon_M(x+2))) - 2.
\end{equation} 
If $d$ is odd, then $T_d$ is an odd function, $T_d$ fixes both $2$ and $-2$, and $\phi(0) \neq 0$ from 2-ramification structure (2F). Putting $x = 0$ in \eqref{chebeqn} gives $\epsilon_L \epsilon_M = -1$. Because $d$ is odd, $T_d$ is an odd function, and so in both cases $\epsilon_L = 1, \epsilon_M = -1$ and $\epsilon_L = -1, \epsilon_M = 1$, we have $c\phi(x/c) = - (T_d(x+2)) - 2$. If $d$ is even, then $T_d$ is an even function, $T_d(\pm 2) = 2$, and $\phi(0) = 0$ from 2-ramification structure (2O). Putting $x = 0$ in \eqref{chebeqn} then gives $\epsilon_L = -1$, implying $c\phi(x/c) = T_d(\pm(x+2)) - 2 = T_d(x+2) - 2$.

It remains only to show that $c \in K$. But $\phi \in K[x]$ by assumption, and $\phi(x) \in (x-\beta)\C[x]^{*2}$ if $d$ is odd and $\phi(x) \in x(x-\beta)\C[x]^{*2}$ if $d$ is even. From Theorem \ref{fieldthm} we have $\beta \in K$, whence $c \in K$. 
\end{proof}

\begin{proof}[Proof of Corollary \ref{quadpoly}]
Let $\phi \in \Q[x]$ have degree 2, and suppose that $\phi$ has a rational orbit containing infinitely many distinct squares. Then Corollary \ref{mainpoly} implies that either (1) $\phi(x) = c(g(x))^2$ for some $g \in \Q(x)$ or (2) $c\phi(x/c) = T_2(x+2) - 2 = x^2 + 4x$ with $c \in \Q^*$. In case (1), we must have $c \in \Q^{*2}$, for otherwise $\phi$ has no rational orbits with infinitely many distinct squares; hence $\phi$ satisfies (a) of the present corollary. In case (2), putting $x = cX$ gives $\phi(X) = (c^2X^2 + 4cX)/c = cX^2 + 4X$, and so $\phi$ satisfies (b) of the present corollary. 

Assume now that $\phi$ satisfies (a) or (b) of the present corollary. For maps satisfying (a), all infinite orbits contain infinitely many distinct squares. 
For maps satisfying (b), a simple calculation shows that $\phi^2(x) = \phi(x)(g_2(x))^2$ for some $g_2 \in \Q[x]$, and it immediately follows that $\phi^2(x) = \phi(x)(g_n(x))^2$ for some $g_n \in \Q[x]$ for each $n \geq 1$. Hence for $a \in \Q$, $O_\phi^+(a)$ contains infinitely many squares if and only if $a$ is the $x$-coordinate of a rational point on the curve $C : y^2 = cx^2 + 4x$. But $C$ has genus zero and the rational point $(0,0)$, and thus $C(\Q)$ is infinite. By Northcott's theorem \cite{northcott}, $\phi$ has only finitely many rational points with finite orbits, and hence there must be a rational orbit of $\phi$ containing infinitely many distinct squares. 
\end{proof}

\section{An example} \label{example}

In this section we present an example of a rational function $\phi \in \Q(x)$ of degree 2 and $a \in \Q$ such that $O_\phi^+(a) \cap \PP^1(\Q)^2$ is infinite and $\{n \in \N : \phi^n(a) \in \PP^1(\Q)^2\}$ cannot be written as union of fewer than three arithmetic progressions. By the proof of Theorem \ref{main ML}, such an example cannot be $2$-trivial with respect to $\{0, \infty\}$, and hence satisfies one of the conditions of Theorem \ref{2class}. Our example satisfies (2O) of this theorem, and hence has the form (5d) in Theorem \ref{maingenus}. In the notation of that theorem, set $f(x) = 1$ and $g(x) = x - s$. The discriminant of $(Bx(x-C) - C(x-s)^2)/(-C)$ is $BC(BC - 4Cs + 4s^2)$, and to make this discriminant zero we take $B = 4s(C-s)/C$. We wish for $s$ to have a rational preimage under $\phi$, and so we find that the discriminant of the numerator of $\phi(x) - s$ is $16s^2(C-s)^3/C$. We wish for this to be a square, and hence we take $(C-s)/C = d^2$, i.e. $s = C(1-d^2)$. Doing so gives
$$
\phi^{-1}(s) = \left\{\frac{C(d+1)^2}{2d+1}, -\frac{C(d-1)^2}{2d-1} \right\}.
$$
Letting $v = C(d+1)^2/(2d+1)$, we have the orbit 
$v \mapsto s \mapsto \infty \mapsto B \mapsto \phi(B) \mapsto \cdots.$
From the last paragraph of the proof of Theorem \ref{iterative relationship}, we have the following equivalences modulo $\Q(x)^{*2}$:
\begin{equation} \label{functions}
\phi(x) \equiv Bx(x-C), \; \phi^2(x) \equiv -Cx(x-C), \; \text{and $\phi^n(x) \equiv \phi^{n-2}(x)$ for all $n \geq 3$}. 
\end{equation}
We wish to have $\{n \in \N: \phi^n(v) \in \PP^1(\Q)^2\} = \{0, 2\} \cup \{2n + 1 : n \geq 0\}$, which cannot be written as a union of fewer than three arithmetic progressions. Assume for a moment that $O_\phi^+(v)$ is infinite, and in particular $\phi^n(v) \in \Q$ for all $n \neq 2$; we will justify this later. From \eqref{functions}, it is sufficient for $v$ and $s$ to be in $\Q^2$ and $\phi(B) \not\in \Q^2$. 
Clearly $v  \equiv C(2d+1) \bmod{\Q^{2}}$ and $s \equiv C(1-d^2) \bmod{\Q^{2}}$, and one calculates $\phi(B) \equiv -C \bmod{\Q^{2}}$. 
If $d \in \Q$ satisfies $C(2d+1) \in \Q^{2}, C(1-d^2) \in \Q^{2},$ and $-C \not\in \Q^{2}$, then the elliptic curve 
$$
E: y^2 = (2x+1)(1-x^2).
$$
has a point in $E(\Q)$ with $x$-coordinate $d$. This curve has conductor 24, and is isomorphic to curve 24a1 in Cremona's table \cite{cremona}. It has rank zero over $\Q$ and $\Q$-torsion subgroup $\Z/2\Z \otimes \Z/4\Z$. Among the seven finite torsion points are five with $x \in \{0, \pm 1, -1/2\}$, and if $d$ takes any of these values then either $v = 0$, $s = 0$, or $s = C$, which are impossible in our setting. The other two points are $(x, y) = (-2, \pm 3)$, and so we must have $d = -2$. With this choice, we may take $C = -3t^2$ for any $t \in \Q \setminus \{0\}$, giving $s = 9t^2$ and $v = t^2$. Hence
\begin{equation*} \label{type76}
\phi(x) = \frac{144t^2x(x + 3t^2)}{(x-9t^2)^2} \qquad t \in \Q \setminus \{0\},
\end{equation*}
is the unique family in $\Q(x)$ satisfying our conditions. It remains to show that the orbit
\begin{equation} \label{orbitt}
O_\phi^+(t^2) = \left\{t^2, 9t^2, \infty, 144t^2, 3 \left( \frac{112t}{5} \right)^2, \left( \frac{151872t}{11869} \right)^2, 3 \left( \frac{17917453568t}{807305405} \right)^2, \ldots \right\}
\end{equation}
is infinite. Observe that if \eqref{orbitt} is infinite for $t = 1$, then the same holds for all $t \in \Q \setminus \{0\}$. When $t = 1$ we obtain the map $\phi_1(x) = \frac{144x(x + 3)}{(x-9)^2}$, which has good reduction at the primes $5$ and $7$. Writing $\Fp$ for the finite field with $p$ elements, one checks that in $\PP^1(\mathbb{F}_5)$, $\phi_1$ has a fixed point and a two-cycle and no other periodic points, while in $\PP^1(\mathbb{F}_7)$, $\phi_1$ has a fixed point and no other periodic points. It follows from \cite[Theorem 2.21]{jhsdynam} that all periodic points of $\phi_1$ in $\Q$ have period one or two. A simple calculation shows that the numerator of $\phi_1^2(x) - x$ is irreducible, and so $\phi_1$ has no two-cycles in $\Q$. Hence the only periodic point for $\phi_1$ in $\Q$ is the fixed point $0$. Thus $O_{\phi_1}(1)$ is infinite, as desired.

\section*{Acknowledgements} We thank Tommy Occhipinti, Bjorn Poonen, Joseph Silverman, Thomas Tucker, and Fedor Pakovich, for helpful comments and references. We extend special thanks to Michael Zieve for extensive and illuminating discussions of the literature surrounding this paper. We are also grateful to the anonymous referees, who furnished numerous suggestions for improving the paper's exposition, including the proof of Proposition \ref{paramprop}.

\bibliographystyle{plain}

\end{document}